\documentclass[11pt]{article}

\newcommand{\nl}{\left|\left|} \newcommand{\nr}{\right|\right|}
\newcommand{\tl}{\textlatin} \newcommand{\lp}{\left(}
\newcommand{\rp}{\right)} \newcommand{\lb}{\left\lbrace}
\newcommand{\rb}{\right\rbrace} 

\usepackage {epsfig,amsmath,euscript}

\usepackage[latin1]{inputenc}

\usepackage[english]{babel}
\usepackage{color}
\usepackage{subcaption}
\usepackage{amsmath}
\usepackage{amsthm}

\usepackage{amssymb}

\usepackage{ae}
\usepackage{float}
\usepackage[T1]{fontenc}
\usepackage{graphicx}
\usepackage{epstopdf}
\usepackage{longtable}
\usepackage{color}
\definecolor{rred}{rgb}{0.7,0.0,0.2}
\definecolor{bblue}{rgb}{0.2,0.0,0.7}

\newcommand{\secref}[1]{Section \ref{sec:#1}}

\newtheorem{definition}{Definition}
\newtheorem{proposition}{Proposition}
\newtheorem{remark}{Remark}
\newtheorem{theorem}{Theorem}
\newtheorem{corollary}{Corollary}
\newcommand{\seclab}[1]{\label{sec:#1}}

\newcommand{\applab}[1]{\label{app:#1}}
\newcommand{\appref}[1]{App. \ref{app:#1}}
\newcommand{\eqlab}[1]{\label{eq:#1}}
\renewcommand{\eqref}[1]{(\ref{eq:#1})}
\newcommand{\eqsref}[2]{(\ref{eq:#1}) and~(\ref{eq:#2})}

\newcommand{\figref}[1]{Fig.~\ref{fig:#1}}
\newcommand{\figlab}[1]{\label{fig:#1}}
\newcommand{\propref}[1]{Proposition~\ref{proposition:#1}}
\newcommand{\proplab}[1]{\label{proposition:#1}}
\newcommand{\corref}[1]{Corollary~\ref{cor:#1}}
\newcommand{\corlab}[1]{\label{cor:#1}}
\newcommand{\defnref}[1]{Definition~\ref{definition:#1}}
\newcommand{\defnlab}[1]{\label{definition:#1}}

\newcommand{\remref}[1]{Remark~\ref{remark:#1}}
\newcommand{\remlab}[1]{\label{remark:#1}}
\newcommand{\thmref}[1]{Theorem~\ref{theorem:#1}}
\newcommand{\thmlab}[1]{\label{theorem:#1}}
\newcommand{\tablab}[1]{\label{tab:#1}}
\newcommand{\tabref}[1]{Table~\ref{tab:#1}}

\newcommand{\transv}{\mathrel{\text{\tpitchfork}}}
\makeatletter
\newcommand{\tpitchfork}{%
  \vbox{
    \baselineskip\z@skip
    \lineskip-.52ex
    \lineskiplimit\maxdimen
    \m@th
    \ialign{##\crcr\hidewidth\smash{$-$}\hidewidth\crcr$\pitchfork$\crcr}
  }%
}
\makeatother

\title{Regularization and geometry of piecewise smooth systems with intersecting discontinuity sets}

\author{P. Kaklamanos and K. Uldall Kristiansen\footnote{P. Kaklamanos: School of Mathematics, University of Edinburgh, Bayes Centre, 47 Potterrow, Edinburgh, EH8 9BT, United Kingdom (P.Kaklamanos@sms.ed.ac.uk); K. Uldall Kristiansen: Department of Applied Mathematics and Computer Science, Technical University of Denmark, 2800
		Kgs. Lyngby, DK (krkri@dtu.dk)}}

\begin{document}
\maketitle 

\begin{abstract}
In this work, we study the dynamics of piecewise smooth systems on a {codimension-2 transverse intersection of two codimension-1 discontinuity sets}. The Filippov convention can be extended to such {intersections}, but this approach does not provide a unique sliding vector and, as opposed to the classical sliding vector-field on codimension-1 discontinuity manifolds, there is no agreed notion of stability in the codimension-2 context. From a modelling perspective, one may interpret this lack of determinacy as a fact that additional modelling is required; knowing the four adjacent vector-fields is not enough to define a unique forward flow. In this paper, we provide additional information to the system by performing a regularization of the piecewise smooth system, introducing two regularization functions and a small perturbation parameter. Then, based on singular perturbation theory, we define sliding and stability of sliding through a critical manifold of the singularly perturbed, regularized system. We show that this {notion of} sliding vector-field coincides with the Filippov one. The regularized system gives a parameterized surface, the \textit{canopy} \cite{jeffrey2014a}, independent of the regularization functions. This surface serves as our natural basis to derive new and simple geometric criteria on the existence, multiplicity and stability of the sliding flow, depending only on the smooth vector fields around the {intersection}. Interestingly, we are able to show that if there exist two sliding vector-fields then one is a saddle and the other is of focus/node/center type. This means that there is at most one stable sliding vector-field. We then investigate the effect of the choice of the regularization functions, and, using a blowup approach, we demonstrate the mechanisms through which sliding behavior can appear or disappear on the intersection and describe what consequences this has on the dynamics on the adjacent codimension-1 discontinuity sets. This blowup method also shows that {the PWS limit of the regularization 
may be well-defined, even in cases where the Filippov sliding vector-field is nonunique}. Finally, we show the existence of canard explosions of regularizations of PWS systems in $\mathbb R^3$ that depend on a single unfolding parameter.
\end{abstract}
%

\pagestyle{myheadings}
\thispagestyle{plain}

\section{Introduction}\seclab{Introduction}

A piecewise smooth (PWS) system \protect\cite{filippov1988differential, MakarenkovLamb12} consists of finitely many ordinary differential equations
\begin{align*}
 \dot{\textbf x} = X_i(\textbf x),\textbf x\in \mathcal Q_i\subset \mathbb R^n,
\end{align*}
where each $X_i$ is a smooth vector-field. The regions $\mathcal Q_i$ are open sets separated by a codimension-1 set $\Sigma$, called the switching manifold. 

PWS systems occur in a many applications, including problems in contact mechanics (impact, friction, gears, rocking blocks, etc), electronics (switches diodes and DC/DC converters), control engineering and many others. See \protect\cite{Bernardo08, MakarenkovLamb12} for a more complete list of applications and further references. 

Mathematically, PWS systems do not in general define a closed dynamical system. Points within one region $\mathcal Q_i$ can reach  $\Sigma$ in finite time by following $X_i$. From such a point on the switching manifold, it may not be possible to follow another vector-field without jumping in phase space. In this case, one can define a \textit{sliding vector-field} as the convex combination of the vectors, say $X_1$ and $X_2$, that appear, in the generic situation, on either side of $\Sigma$. The sliding vector-field is unique when it exists. This approach is called the Filippov convention and it enables the continuation of orbits that cannot escape $\Sigma$ by following the prescribed vectors $X_i$. The subset of $\Sigma$, where a sliding vector-field can be defined, is called the sliding region. 

A PWS system following the Filippov convention is called a Filippov system. Such systems have received some attention over the past years, see e.g. \protect\cite{jeffrey_geometry_2011,jeffrey_two-fold_2009} where generic bifurcations of these vector-fields are described. Now, even though Filippov systems do possess local forward flows, forward uniqueness can break down in a number of ways. One prominent example of such a breakdown, is the two-fold, where orbits of e.g. $X_1$ and $X_2$ have tangencies with $\Sigma$ at the same point. From such a point, several forward orbits may exist. {It is interesting from a mathematical point of view and from a modelling perspective to replace the PWS system with a more regular one for which the PWS system is an idealisation, and analyse how solutions of the regular model behave as the system approaches the PWS idealisation.} As an example, it is possible to view the PWS system as a singular limit of a smooth, regularized vector-field obtained by gluing the PWS vector-fields, on either side of the discontinuity set, together in a smooth monotonic fashion. Interestingly, when a Filippov system possesses a sliding region, then this \textit{regularized system} possesses an invariant slow manifold as a graph of the sliding region. On this slow manifold, the flow converges to the sliding flow as the regularised system approaches (pointwise) the PWS one. This result is independent of the details of the regularization. Hence, one may view this as an approach to ``derive'' the Filippov sliding vector-field. The regularization approach to PWS systems was used in recent references  \protect\cite{dieci2011a,dieci2013regularizing,guglielmi2017a,KristiansenHogan2015,krihog2,KristiansenHogan2018}. In \protect\cite{KristiansenHogan2018}, for example, it was shown, using techniques from geometric singular perturbation theory, that the PWS two-fold possesses a distinguished orbit that the regularized system follows sufficiently close to its PWS limit. This result is again independent of the details of the regularization - it only depends upon the PWS system. In this sense, one can view a Filippov system as a zero order model which can be ``corrected'' by the regularization approach and the use of singular perturbation theory. Such higher order corrections ``resolve'' ambiguities of the simpler model.  

{In this paper, we are interested in the local situation where $\Sigma$ is not a manifold but the union of two local codimension-1 manifolds $\Pi_f$ and $\Pi_g$ that intersect transversally in a codimension-2 submanifold $\Lambda=\Pi_f\transv \Pi_g$. }
Locally, $\Sigma$ then divides the PWS system into four quadrants
$\mathcal R_1,\ldots, \mathcal R_4$ near $\Lambda$, see \figref{PWSDiagram}. Such systems appear, for example, in gene regulatory networks, see e.g. \protect\cite{acary2014a,guglielmi2015a} and references therein.   Reaching $\Lambda$ by following $X_1,X_2,X_3$ or $X_4$ (compactly $X_{1-4}$) is not generic. It is only generic when following the sliding vector-field within the sliding region of $\Sigma$. 
We will in this paper be working in $\mathbb R^3$ where $\Sigma$ is $2D$ and $\Lambda$ is $1D$.

It is possible to extend the Filippov convention to $\Lambda$ on $\Sigma$ and define a sliding vector-field as a convex combination of the four adjacent vector-fields that is tangent to $\Lambda$. This approach is taken in \protect\cite{jeffrey2014a}, for instance, where the author constructs a parametrized surface from this convex combination, called canopy, and argues that the sliding vector field is defined by the point of intersection of this surface with the tangent space of the codimension-2 discontinuity. But the sliding vector-field is not necessarily unique \protect\cite{jeffrey2014a}, there can be two choices, and Filippov's approach is therefore inherently ill-posed. Moreover, in \protect\cite{jeffrey2014a}, it also says that ``There is no simple criterion for determining a priori how many sliding vectors will exist
in general (at least a general criterion is not yet known). One must solve the system and
investigate how many valid vectors there are within the convex canopy $\mathcal{F}$ that are tangent
to the discontinuity surface $\mathcal{D}$,'' see \protect\cite[p.1091]{jeffrey2014a}. To deal with the ill-posedness of the Filippov approach, \protect\cite{jeffrey2014a} defines a ``dummy system'' which introduces a slow-fast system on a blowup of $\Lambda$. This leads to a notion of stability of the sliding flow \protect\cite[p.1091, Sec. 4]{jeffrey2014a}, and in the closing remarks of the paper it is stated that the justification of the dummy system and its connection to applications, together with the the issue of (non)uniqueness of solutions, remain open problems.

In this paper, we apply the regularization approach to study PWS systems with {intersecting switching manifolds of codimension-1} as idealisations of smooth vector-fields having very rapid transitions across both $\Pi_f$ and $\Pi_g$. In this way we arrive at the canopy described by Jeffrey in \protect\cite{jeffrey2014a} through an associated layer problem of the singularly perturbed, regularized system. We then undertake a geometric analysis of this surface that allows us to derive general and explicit conditions on the existence and multiplicity of the sliding flow, {by studying quadrilateral projections similar to the ones introduced in \protect\cite{dieci2017moments}}. See \secref{para} and \secref{existence}.  Our general approach to the problem is to relate the sliding vector-field to the dynamics on the critical manifold of the singular perturbed regularized system. See \secref{definition} {and \protect\cite{KristiansenHogan2015,KristiansenHogan2018,krihog2,teixeira2012regularization}}.  This approach also gives rise to a natural definition of stability of the sliding vector-field, see \defnref{slide2}. We can then study bifurcations of the sliding vector-field using standard techniques of dynamical systems theory. The regularization approach also provides a justification of the dummy system used in \protect\cite{jeffrey2014a}, as it turns out that this system is in fact related to our layer problem for a particular choice of regularization functions (see \remref{dummy}). 

%

The paper is organized as follows: In \secref{preliminaries}, we first present some basic concepts from PWS systems and introduce a regularization of a PWS system across a single codimension-1 discontinuity set. We also demonstrate in \thmref{holy} the equivalence between sliding and reduced, slow flow along a critical manifold of the regularization. In \secref{regul}, we then introduce a (double-)regularization of a PWS system near $\Lambda$. We use this system to define sliding and stability of sliding along $\Lambda$ in \secref{definition} and show, in line with \thmref{holy}, that this definition of the sliding vector-field is equivalent to the Filippov one, see \thmref{holy2}. In the following two sections, \secref{para} and \secref{existence} we then present a thorough and novel analysis of the existence and multiplicity of sliding. In \secref{stabsec} we study the stability of sliding. Here we show that if two sliding vector-fields exists on $\Lambda$, then at most one is stable, see \thmref{uniqueStable}. We also provide some conditions on the PWS system for which the stability of sliding vector-field is independent of the details of the regularization, see \propref{stabilityExample} and \corref{stabilityExample2}. Finally, in \secref{bifurcation} we then describe the emergence and disappearance of sliding vectors and, using a blowup approach, study its consequences on the sliding dynamics along the adjacent codimension-1 sliding manifolds. Here we discuss how this approach can be used to obtain a well-defined limit of solutions as the smooth system approaches the PWS one, even in cases where the Fillipov sliding vector-field is non-unique. We conclude the paper in \secref{conclusion}.

\section{Preliminaries}\seclab{preliminaries}
In this section we set up our problem and present our PWS system in a suitable normalized form. We focus on $\mathbb R^3$ here and delay discussions of possible extensions to $\mathbb R^n$ to the conclusion section, \secref{conclusion}.  We therefore suppose that the switching manifold is the union of two $2D$ manifolds $\Pi_f, \Pi_g \subset \mathcal U$ defined by $\Pi_f=f^{-1}(0)$, $\Pi_g=g^{-1}(0)$ where $f(\textbf x)$ and $g(\textbf x)$ are two smooth functions both having $0$ as a regular value. We then suppose that these manifolds intersect transversally along $\Lambda =\Pi_f\transv \Pi_g$. We introduce local coordinates $\textbf x= (x,y,z)$ such that $f(\textbf{x})=y$, $\Pi_f =\{\textbf{x}\in \mathcal U~\vert ~ y=0\}$, $g(\textbf{x})=z$, $\Pi_g =\{\textbf{x}\in \mathcal U~\vert ~z=0\}$, and
\begin{eqnarray*}
\Lambda = \Pi_f \transv \Pi_g  = \{(x,y,z)~\big\vert ~ x\in \mathcal I,y=z=0\}:\quad \mbox{a subset of the $x$-axis},
\end{eqnarray*}
Here $\mathcal U\subset \mathbb R^3$ and $\mathcal I$ is an appropriate interval. We then consider a PWS system on $\mathcal U$ in the following form
\begin{gather}
\begin{gathered}
\dot{\textbf{x}} = X(\textbf x),\quad
X(\textbf x) 
=\left\{ \begin{array}{cc}
X_1(\textbf x)& \text{for}\quad  \textbf x\in \mathcal{Q}_1,\\
X_2(\textbf x)& \text{for}\quad \textbf x\in \mathcal{Q}_2,\\
X_3(\textbf x)& \text{for}\quad  \textbf x\in \mathcal{Q}_3,\\
X_4(\textbf x)& \text{for}\quad \textbf x\in \mathcal{Q}_4,
\end{array}\right.
\end{gathered}\eqlab{X14}
\end{gather}
where in our local coordinates, $\mathcal Q_{1-4}$ correspond to the four ``quadrants'' ($\{y\ge 0,z\ge 0\}$, $\{y\le 0, z\ge0\}$, $\{y\le 0,z\le 0\}$ and $\{y\ge 0,z\le 0\}$, respectively) 
that the $\mathbb{R}^3$ space is divided into by $\Pi_f$ and $\Pi_g$ (see \figref{PWSDiagram}).
We suppose that
\begin{eqnarray}
X_i\lp\textbf{x}\rp =  \lp
\alpha_i(\textbf{x}), \beta_i(\textbf{x}), \gamma_i(\textbf{x})
\rp^T, \quad i = 1,2,3,4,  \eqlab{SFields}
\end{eqnarray}
are smooth vector fields on $\mathcal U$. This holds, for example, if each $X_{i}$ is analytic on $\overline{\mathcal Q}_i$, after possibly restricting the local neighborhood $\mathcal U$ further. 

We further sub-divide $\Pi$ into 
\begin{align*}
 \Pi = \Pi_{1}\cup \Pi_{2}\cup \Pi_{3} \cup \Pi_{4}, 
\end{align*}
where 
\begin{align}
 \Pi_{i} =\overline{\mathcal Q}_i\cap\overline{\mathcal Q}_{i+1},\eqlab{PiiD}
 \end{align}
 See \figref{PWSDiagram}. The subscripts in \eqref{PiiD}
 are considered mod $4$ such that $\overline{\mathcal Q}_{5} =\overline{\mathcal Q}_{1}$. We adopt this convention henceforth. 
 
 Each plane $\Pi_{i}$ is a codimension-1 \textit{switching} manifold. For example, $\Pi_1$ separates the set $\mathcal{Q}_{1} =\{\textbf{x}\in \mathcal U~\vert~ y>0,\,z>0\}$ from the set $\mathcal{Q}_{2}=\{\textbf{x}\in \mathcal U~\vert~ y<0,\,z>0\}$. 
 Each switching manifold $\Pi_i$ is then sub-divided into three types of regions, crossing, sliding and folds. For example, for $\Pi_1$ we have $\Pi_1=\Pi_1^{cr}\cup \Pi_1^{sl}\cup \Pi_1^{f}$ where
\begin{itemize}
 \item $\Pi_1^{cr}\subset \Pi_1$ is the \textit{crossing region} where: 
\begin{eqnarray}
(X_1 f(x,0,z)(X_2 f(x,0,z)) =\beta_1(x,0,z) \beta_2(x,0,z) >0 \eqlab{slide_cond}.
\end{eqnarray}
 \item $\Pi_1^{sl}\subset \Pi_1$ is the \textit{sliding region} where \begin{eqnarray}
 (X_1 f(x,0,z))(X_2 f(x,0,z)) =\beta_1(x,0,z) \beta_2(x,0,z)<0 \eqlab{PiSL}.
 \end{eqnarray}
 \item $\Pi_1^{f}\subset \Pi_1$ is the \textit{fold region} where \begin{eqnarray}
 (X_1 f(x,0,z))(X_2 f(x,0,z)) =\beta_1(x,0,z) \beta_2(x,0,z)=0 \eqlab{Pif}.
 \end{eqnarray}
\end{itemize}
Here $X_{i} f=\nabla f \cdot X_{i}$ denotes the Lie-derivative of $f$ along $X_i$, $i=1,2,3,4$. Since $f(\textbf{x})=y$ in our coordinates we have that $X_i f =\beta_{i}$ by \eqref{SFields}. Similarly, $X_i g=\gamma_i$. We define the subsets of $\Pi_i$, $\Pi_i^{cr}$, $\Pi_i^{sl}$, $\Pi_i^{f}$ analogously for $i=2,3,4$.

\begin{figure}[h!]
	\centering
	\begin{subfigure}[b]{0.49\textwidth}
		\centering
		\includegraphics[scale = 0.1]{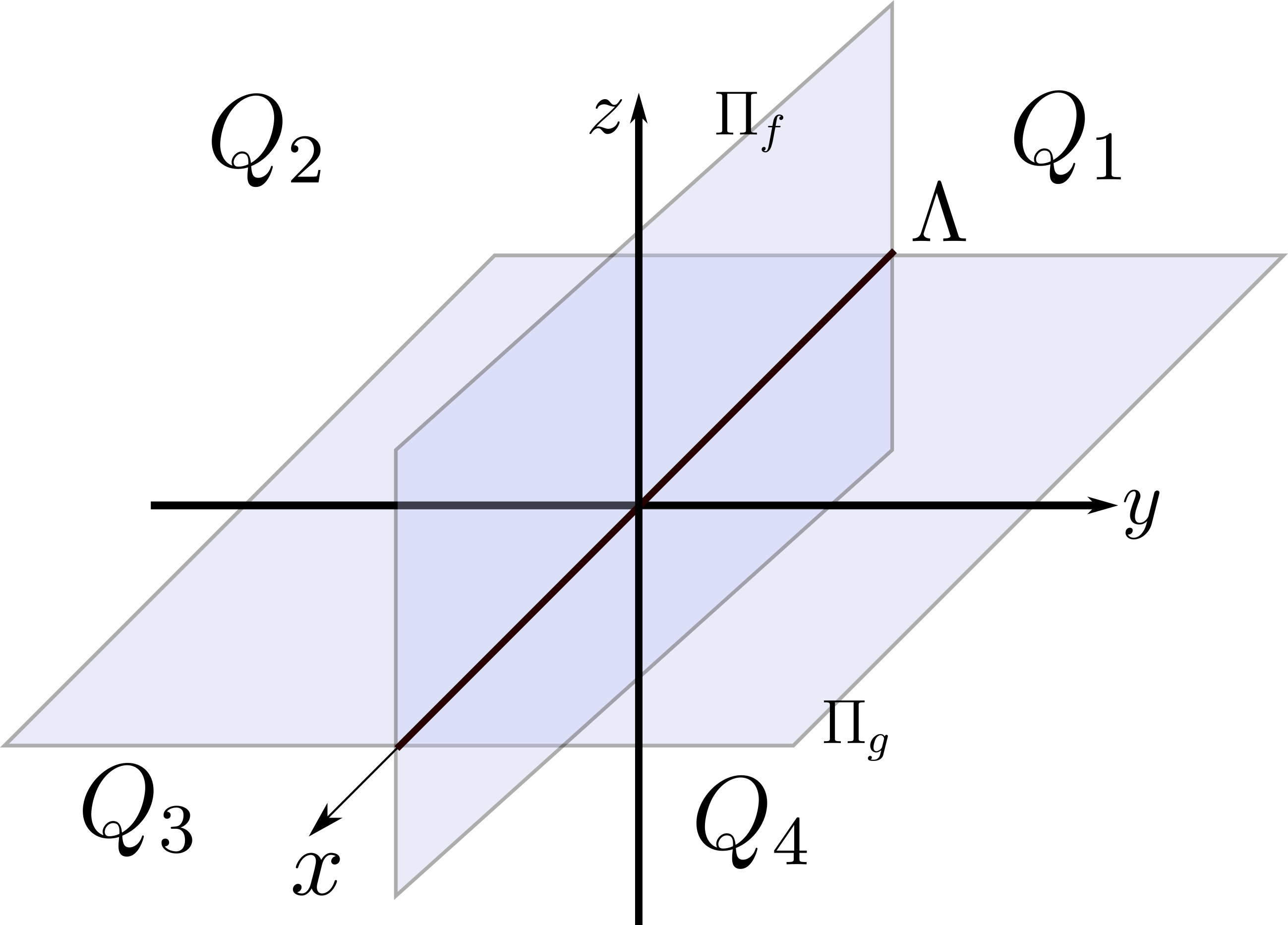}
		\caption{The switching sets in $\mathbb{R}^3$}
	\end{subfigure}
	\begin{subfigure}[b]{0.49\textwidth}
		\centering
		\includegraphics[scale = 0.1]{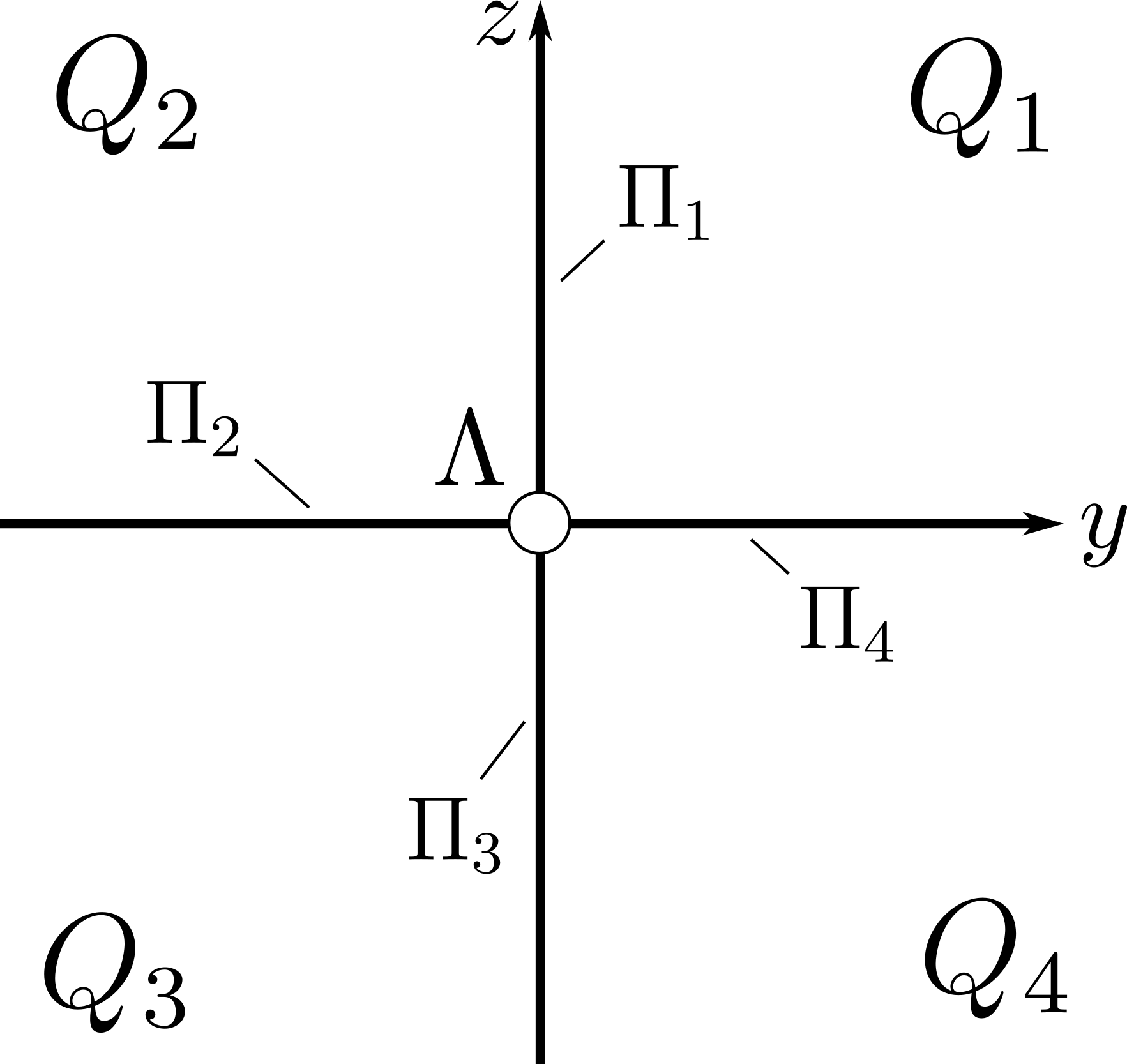}
		\caption{The switching sets in the $(y,z)-$plane}
	\end{subfigure}
	\caption{The switching sets and the four quadrants of the PWS system \eqref{X14} in $\mathbb{R}^3$ and the projection onto the $(y,z)-$plane.}
\figlab{PWSDiagram}
\end{figure}

In the sliding region, the vector fields on either side of $\Pi_i^{sl}$ point either toward or away from $\Pi_i^{sl}$. 
For $i=1$, we define the sliding vector-field by Filippov \protect\cite{filippov1988differential} as follows.
\begin{definition}\defnlab{Xslcod1}
Consider the PWS system $(X_1,X_2)$ on $\mathcal Q_1 \cup \mathcal Q_2$. Then the \textnormal{sliding vector field} $X_1^{sl}$ on $\Pi_1^{sl}$ (where $\beta_1(x,0,z)\beta_2(x,0,z)<0$) is the convex combination of 
$X_1$ and $X_2$ such that $X_1^{sl}(\textbf x)$ is tangent to $\Pi_1^{sl}$. In details,
 \begin{eqnarray}
 X_1^{sl}(\textbf x) =  
                         \sigma_1(\textbf x) X_1(\textbf x)+ (1-\sigma_1(\textbf x)) X_{2}(\textbf x)\in T_{\textbf x} \Pi_1^{sl},~ \textbf x\in \Pi_1^{sl},
 \eqlab{XSliding}
\end{eqnarray}
where $T_{\textbf x} \Pi_1^{sl}$ is the tangent space to $\Pi_1^{sl}$ at $\textbf x$ and $\sigma_1$ satisfies
\begin{eqnarray*}
\sigma_1(\textbf x) = 
\frac{\beta_2(\textbf x)}{\beta_2(\textbf x)-\beta_1(\textbf x)},~ \textbf x\in \Pi_1^{sl}.
\end{eqnarray*}
The flow of $X_1^{sl}$ is called the sliding flow. 
If $\beta_1(\textbf x)<0$ and $\beta_2(\textbf x)>0$, then the sliding flow is said to be stable, while if $\beta_1(\textbf x)>0$ and $\beta_2(\textbf x)<0$, then the sliding flow is said to be unstable.
\end{definition}

We define $X_i^{sl}$ on $\Pi_i^{sl}$ for $i=2,3,4$ analogously as the convex combination of $X_i$ and $X_{i+1}$ ($5\rightarrow 1$ if $i=4$) that is tangent to $\Pi_i^{sl}$. Notice that the sliding vector-field $X_i^{sl}$ on $\Pi_i^{sl}$ is unique.

Forward orbits of either $X_{1-4}$ on $\mathcal Q_{1-4}$, or $X_i^{sl}$ on $\Pi_i^{sl}$ can also reach $\Lambda$ in finite time. To have a well-defined forward or backward flow in our open set $\mathcal U$, we therefore need to define a \textit{sliding vector-field} on $\Lambda$. Traditionally, sliding vector fields on $\Lambda$ have been defined as the convex combinations of $X_{1-4}$ which are tangent to $\Lambda$, see for example \protect\cite{jeffrey2014a,filippov1988differential,guglielmi2015a}:

\begin{definition}
	\textbf{(Extension of the Filippov Convention on $\Lambda$)}
	Consider the PWS system $(X_1,X_2, X_3,X_4)$ on $\mathcal Q_1 \cup \mathcal Q_2\cup \mathcal Q_3 \cup \mathcal Q_4$. A \textnormal{sliding vector field} $X^{sl}$ (if it exists) is then a convex combination of $X_{1-4}$ such that $X^{sl}$ is tangent to $\Lambda$. In details,
	\begin{eqnarray}
	X^{sl}(\textbf x)&=
	\nu_1(x) X_1(\textbf x)+ \nu_2(x) X_2(\textbf x)+\nu_3(x) X_3(\textbf x)+\nu_4(x) X_4(\textbf x)\in T_{\textbf x}\Lambda , \quad \textbf x=(x,0,0)\in \Lambda,
	\eqlab{slide}
	\end{eqnarray}
	where $T_{\textbf x}\Lambda\simeq$ the $x$-axis is the tangent space to $\Lambda$ at $\textbf x$ and
	\begin{gather*}
	\sum_{i=1}^4\nu_i(x) = 1,\quad (x,0,0) \in \Lambda. \nonumber
	\end{gather*}
	\defnlab{slide2Fil}
\end{definition}

However, as stated in these references, there exists no simple criterion yet on determining a priori if and how many sliding vector fields are produced on $\Lambda$ in this way, without performing calculations and investigating whether and how many convex combinations of $X_{1-4}$ that are tangent to $\Lambda$ exist. That is, we are not yet in position to define and describe $\Lambda^{cr}$, $\Lambda^{sl}$ and $\Lambda^f$ using conditions similar to \eqref{slide_cond}, \eqref{PiSL} and \eqref{Pif}. Furthermore, the notion of stability of the sliding flow on $\Lambda$ is not as clear as in the case of codimension-1 discontinuities. In the case of codimension-1 discontinuities, according to \defnref{Xslcod1} and as \figref{co1slide} illustrates, the sliding flow is simply characterized as either stable or unstable depending on the orientation of the smooth vector fields on either side of the discontinuity. In the case of codimension-2 discontinuities, on the other hand,  since we have four smooth vector fields around the discontinuity (see e.g. \figref{proj1}), such a simple characterization is not possible, and further analysis is required. Also, it seems inaccurate to describe \figref{proj1} (b), second row, as just unstable. We see both orbits entering and leaving $\Lambda$, creating a saddle structure with stable-like and unstable-like manifolds, each being $2D$ in the full $3D$ space. 
 
To circumvent these issues we will in this paper simply view the PWS vector-field \eqref{X14} as a singular limit of a regularization of \eqref{X14}. First, we will in the following section describe the connection between the sliding vector-field in \defnref{Xslcod1} and its regularization for the case of a codimension-1 discontinuity set. We will again focus on $\Pi_1$ but $\Pi_i$ can be handled similarly. 
 \begin{figure}[h!]
 	\centering
 	\begin{subfigure}[b]{0.49\textwidth}
 		\centering
 		\includegraphics[scale = 0.091]{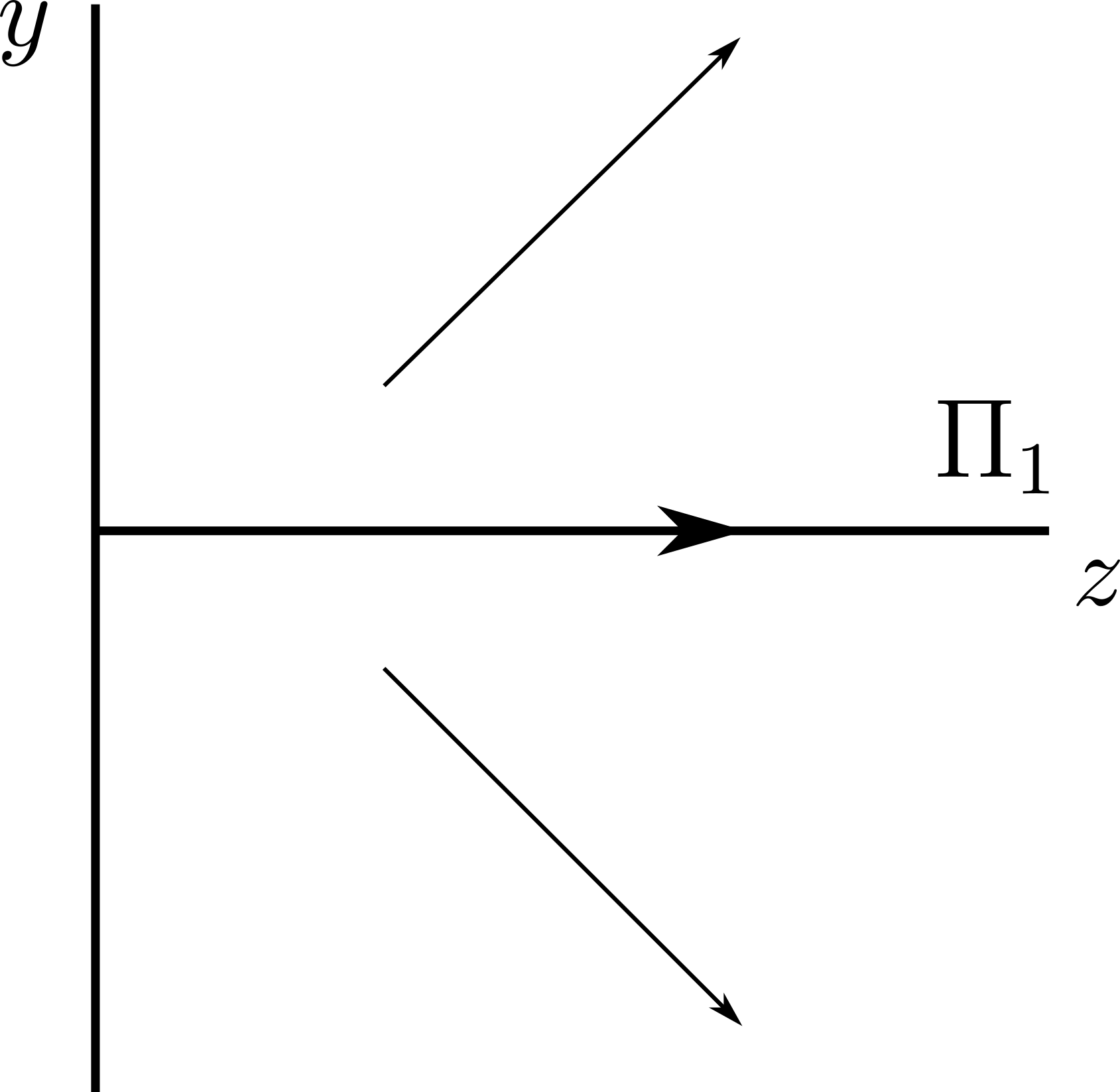}
 		\caption{unstable sliding}
 	\end{subfigure}
 	\begin{subfigure}[b]{0.49\textwidth}
 		\centering
 		\includegraphics[scale = 0.091]{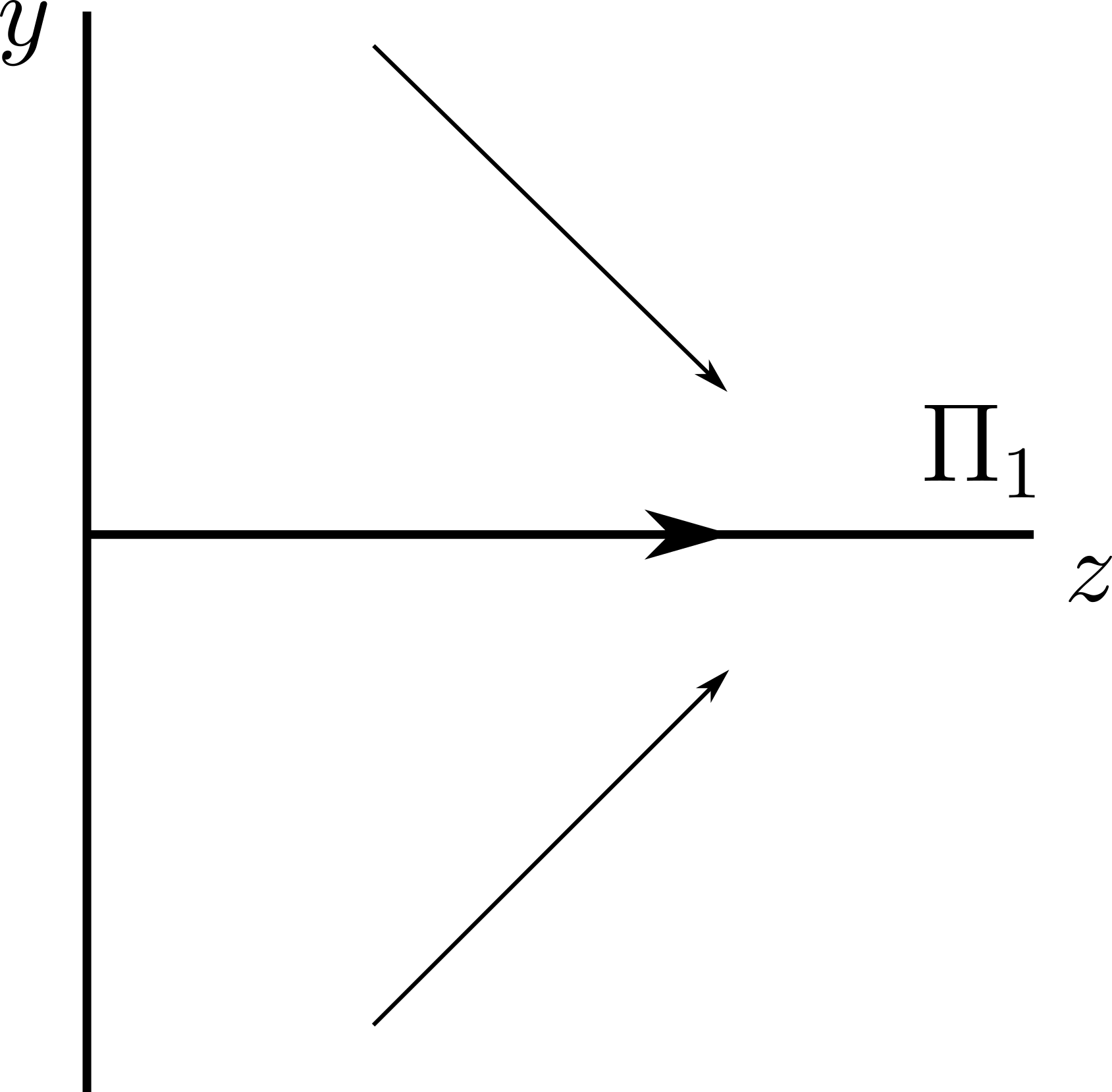}
 		\caption{stable sliding}
 	\end{subfigure}
 	\caption{In the case of sliding on codimension-1 discontinuities, the sliding flow is characterized as simply stable or unstable, depending on the orientation of the smooth vector fields on either side of the discontinuity. If both vector fields point away from the discontinuity (left), then the sliding flow is unstable, while if both vector fields point towards the discontinuity (right), then the sliding flow is stable.}
 	\figlab{co1slide}
 \end{figure}

\subsection{Regularization of the PWS system $(X_1,X_2)$ across the codimension-1 discontinuity set $\Pi_1$}
We define a regularization function as follows.
\begin{definition}\defnlab{phiFuncs}
	A regularization function is a smooth ($C^{k\ge 1}$) function $\phi:\,\mathbb R\rightarrow [-1,1]$ which is strictly increasing $\phi'(s)>0$ for all $s$: $\phi(s)\in (-1,1)$, and asymptotic:
	\begin{align*}
	\phi(s)\rightarrow \pm 1\quad \mbox{for}\quad s\rightarrow \pm \infty.
	\end{align*}
	Moreover, the two functions $\phi_+:[0,\infty)\rightarrow [-1,1]$ and $\phi_-:(-\infty,0] \rightarrow [-1,1]$ defined as
	\begin{align}
	\phi_\pm(r) = \begin{cases}
	\pm 1 ~\mbox{for $r=0$},\\
	\phi(r^{-1}) ~\mbox{for $r\gtrless 0$},\end{cases}
	 \eqlab{phipm}
	\end{align}
	are also smooth functions.
\end{definition}

This class of regularization functions include the non-analytic Sotomayor and Teixeira regularization functions \protect\cite{Sotomayor96} that satisfy
\begin{align*}
\phi(s)=\left\{\begin{array}{cc}
1 & \text{for}\quad s\ge 1,\\
\in (-1,1)& \text{for}\quad s\in (-1,1),\\
-1 & \text{for}\quad s\le -1.\\
\end{array}\right.
\end{align*}
Such functions were used in \protect\cite{KristiansenHogan2015,krihog2,reves_regularization_2014}. But the set of functions in \defnref{phiFuncs} also include more natural regularization functions such as $\lp2/ \pi \rp\text{arctan}(s)$ and $\tanh(s)$. 
\begin{remark}\remlab{newRemark}
{
The condition \eqref{phipm} is a technical one that enables the use of dynamical systems theory to study \eqref{Regul} for $\varepsilon\ll 1$ (using local invariant manifolds).  See e.g. \protect\cite{KristiansenHogan2018,kriBlowup}. In this manuscript, we will use \eqref{phipm} in \secref{bifurcation}, see also \appref{aa} for further details. }
\end{remark}

We then define the following regularization of the PWS system $(X_1,X_2)$ on $\mathcal Q_1\cup \mathcal Q_2$:
\begin{definition} A regularization of the PWS system $(X_1,X_2)$ on $\mathcal Q_1\cup \mathcal Q_2$ is a smooth vector field:
	\begin{eqnarray}
	X_\varepsilon= \frac{\lp 1+\psi(\varepsilon^{-1}y)\rp}{2}X_1+\frac{\lp 1-\psi(\varepsilon^{-1}y)\rp}{2}X_2, \eqlab{Regul}
	\end{eqnarray}
	for $0<\varepsilon\ll 1$, where the function $\psi$ satisfies \defnref{phiFuncs}.
\end{definition}

Using \eqref{SFields}, the regularized system \eqref{Regul} gives the system of differential equations:
\begin{eqnarray}
\begin{aligned}
\dot{x} &= \frac{\lp 1+\psi(\varepsilon^{-1}y)\rp}{2}\alpha_1+\frac{\lp 1-\psi(\varepsilon^{-1}y)\rp}{2}\alpha_2,\\
\dot{y} &= \frac{\lp 1+\psi(\varepsilon^{-1}y)\rp}{2}\beta_1+\frac{\lp 1-\psi(\varepsilon^{-1}y)\rp}{2}\beta_2,\\
\dot{z} &= \frac{\lp 1+\psi(\varepsilon^{-1}y)\rp}{2}\gamma_1+\frac{\lp 1-\psi(\varepsilon^{-1}y)\rp}{2}\gamma_2.
\end{aligned}
\eqlab{Reg1}
\end{eqnarray}
Notice that 
 \begin{align*}
  X_\varepsilon(\textbf x) \rightarrow \begin{cases}
                                        X_1(\textbf x)~\mbox{for $\textbf x\in \mathcal Q_1$},\\
                                        X_2(\textbf x)~\mbox{for $\textbf x\in \mathcal Q_2$},
                                       \end{cases},
 \end{align*}
pointwise for $\varepsilon\rightarrow 0$. However, the system is singular for $y=\varepsilon=0$. It will therefore be useful to work with two separate time scales. We will say that $t$ in \eqref{Reg1} is the \text{slow time} whereas $\tau=t\varepsilon^{-1}$ will be referred to as the \text{fast time}. Furthermore, dynamics is hidden within $y=\mathcal{O}(\varepsilon)$. We therefore introduce $\hat y$ by:
\begin{align}
 \hat y = \varepsilon^{-1} y. \eqlab{yhat}
\end{align}
Inserting \eqref{yhat} into \eqref{Reg1} gives:
\begin{align}
\begin{aligned}
\dot{x} &= \frac{\lp 1+\psi(\hat{y})\rp}{2}\alpha_1+\frac{\lp 1-\psi(\hat{y})\rp}{2}\alpha_2,\\
\varepsilon\dot{\hat{y}} &= \frac{\lp 1+\psi(\hat{y})\rp}{2}\beta_1+\frac{\lp 1-\psi(\hat{y})\rp}{2}\beta_2,\\
\dot{z} &= \frac{\lp 1+\psi(\hat{y})\rp}{2}\gamma_1+\frac{\lp 1-\psi(\hat{y})\rp}{2}\gamma_2.
\end{aligned}
\eqlab{SFS}
\end{align}
This is a slow-fast system \protect\cite{kuehn2015}, in the slow formulation with both $x$ and $z$ being slow variables and $\hat y$ being fast. The fast system with respect to the fast time $\tau$ is 
\begin{eqnarray}
\begin{aligned}
{x}' &= \varepsilon\lp\frac{\lp 1+\psi(\hat{y})\rp}{2}\alpha_1+\frac{\lp 1-\psi(\hat{y})\rp}{2}\alpha_2\rp,\\
{\hat{y}'} &= \frac{\lp 1+\psi(\hat{y})\rp}{2}\beta_1+\frac{\lp 1-\psi(\hat{y})\rp}{2}\beta_2,\\
{z'} &= \varepsilon\lp\frac{\lp 1+\psi(\hat{y})\rp}{2}\gamma_1+\frac{\lp 1-\psi(\hat{y})\rp}{2}\gamma_2\rp.
\end{aligned}\eqlab{SFF}
\end{eqnarray}
The limiting systems \eqref{SFF}$_{\varepsilon=0}$: 
\begin{align}
\begin{aligned}
{x}' &= 0,\\
{\hat{y}'} &= \frac{\lp 1+\psi(\hat{y})\rp}{2}\beta_1+\frac{\lp 1-\psi(\hat{y})\rp}{2}\beta_2,\\
{z'} &= 0,
\end{aligned}\eqlab{layer}
\end{align}
is called the \textit{layer problem}, while \eqref{SFS}$_{\varepsilon=0}$:
\begin{eqnarray}
\begin{aligned}
\dot{x} &= \frac{\lp 1+\psi(\hat{y})\rp}{2}\alpha_1+\frac{\lp 1-\psi(\hat{y})\rp}{2}\alpha_2,\\
0 &= \frac{\lp 1+\psi(\hat{y})\rp}{2}\beta_1+\frac{\lp 1-\psi(\hat{y})\rp}{2}\beta_2,\\
\dot{z} &= \frac{\lp 1+\psi(\hat{y})\rp}{2}\gamma_1+\frac{\lp 1-\psi(\hat{y})\rp}{2}\gamma_2,
\end{aligned}
\eqlab{reduced}
\end{eqnarray}
is called the \textit{reduced problem}. In \eqref{layer} and \eqref{reduced} 
\begin{align}
 \alpha_i = \alpha_i(x,0,z),\,\beta_i=\beta_i(x,0,z),\,\gamma_i=\gamma_i(x,0,z).\eqlab{later}
\end{align}
Notice that $x$ and $z$ are constant in \eqref{layer} whereas $\hat y$ is slaved in \eqref{reduced}.
Let $\psi_* = \psi_* \lp x,z\rp$ be defined as:
\begin{align}
\psi_* = \lp \beta_2+\beta_1\rp / \lp \beta_2-\beta_1\rp,\eqlab{psiStar}
\end{align}
for $\beta_2\ne\beta_1$. Clearly, $\psi_*(x,z)\in (-1,1)$ if and only if $(x,0,z)\in \Pi_{1}^{sl}$ and $\beta_1(x,0,z)\beta_2(x,0,z)<0$. The \textit{critical manifold $C_0$} of the slow-fast system \eqref{layer} is then defined as the following graph over $\Pi_{1}^{sl}$:
\begin{align}
C_0 = \lb \lp x, \hat{y}, z\rp ~\lvert~  ~\hat{y} = \psi^{-1}\lp \psi_*(x,z) \rp,~ (x,0,z)\in \Pi_1^{sl}\rb. \eqlab{C01}
\end{align}
Notice, that $C_0$ is the set of equilibria of \eqref{layer}. 
Now, we have the following important result.
\begin{theorem}\thmlab{holy}
 \protect\cite{Llibre09,KristiansenHogan2015}
Consider a stable (unstable) sliding vector-field $X_1^{sl}$ on $\Pi_1^{sl}$. Then $C_0$ \eqref{C01} is a normally hyperbolic and attracting (repelling, respectively) critical manifold of \eqref{layer}. Furthermore, let $\Phi:C_0\to\Pi_1^{sl}$ be the diffeomorphism defined by $\Phi (x,\hat y,z)=(x,0,z)$. Then the pull-back of $X_1^{sl}$, $\Phi_* X_1^{sl}$, coincides with the reduced vector-field, see \eqref{reduced}, on $C_0$.
\end{theorem}
\begin{proof}
 The proof is straightforward but we include some details here because the result is crucial to the approach of the paper. For the hyperbolicity and the stability we simply linearize \eqref{layer} about a point $(x,\hat y,z)\in C_0$. If $X_1^{sl}$ is stable (unstable) then we find a single non-zero and negative (positive) eigenvalue. To realise that the reduced problem coincides with $X_{sl}$ we define $\sigma_1 = {\lp 1+\psi_*\rp}/{2}$ so that
\begin{eqnarray*}
\sigma_1 = \frac{\beta_2}{\beta_2-\beta_1},
\end{eqnarray*}
and realise from \eqref{reduced} that
\begin{align*}
\dot{x} &= \sigma_1 \alpha_1+\lp 1-\sigma_1\rp \alpha_2,\\
\dot{z} &= \sigma_1 \gamma_1+\lp 1-\sigma_1\rp \gamma_2,
\end{align*}
which coincides with $X_1^{sl}$  \eqref{XSliding}.
\end{proof}

The converse statement is also true, i.e. a reduced vector-field on a critical manifold also gives sliding of the PWS system. Similarly, if $(x,0,z)\in \Pi_1^{cr}$ then $\psi_*$ in \eqref{psiStar} is $\notin (-1,1)$ and therefore there is no equilibrium of \eqref{layer}. Hence $\hat y'\gtrless 0$. 

\section{{Regularization of the PWS system at the intersection of the discontinuities}}\seclab{regul} 
We now define a regularization of the PWS system \eqref{X14} in a neighborhood of $\Lambda$ by {regularizing across both codimension-1 discontinuity sets $\Pi_f$ and $\Pi_g$ at the same time}:
\begin{definition} A regularization of the PWS system \eqref{X14} is a smooth vector field:
\begin{eqnarray}
\begin{aligned}
X_\varepsilon(x,y,z)= &\frac{1}{2}\lp \frac{X_1}{2}{\lp 1+\phi\lp \varepsilon^{-1}z\rp\rp}+\frac{X_4}{2}{\lp 1-\phi\lp \varepsilon^{-1}z\rp\rp}\rp{\lp 1+\psi\lp \varepsilon^{-1}y\rp\rp}\\
&+\frac{1}{2}\lp \frac{X_2}{2}{\lp 1+\phi\lp \varepsilon^{-1}z\rp\rp}+\frac{X_3}{2}{\lp 1-\phi\lp \varepsilon^{-1}z\rp\rp}\rp{\lp 1-\psi\lp \varepsilon^{-1}y\rp\rp},
\end{aligned}\eqlab{reg3a}
\end{eqnarray}
for $0<\varepsilon\ll 1$, where the functions $\phi, \psi$ both belong to the class of functions defined in \defnref{phiFuncs}.
\end{definition}

Notice that 
\begin{align}
X_\varepsilon(\textbf x) \rightarrow \begin{cases}
X_1(\textbf x)~\mbox{for $\textbf x\in \mathcal Q_1$},\\
X_2(\textbf x)~\mbox{for $\textbf x\in \mathcal Q_2$},\\
X_3(\textbf x)~\mbox{for $\textbf x\in \mathcal Q_3$},\\
X_4(\textbf x)~\mbox{for $\textbf x\in \mathcal Q_4$},\\
\end{cases} 
\eqlab{singlim}
\end{align}
pointwise for $\varepsilon\rightarrow 0$. 
For simplicity, we will henceforth assume the following:
\begin{itemize}
 \item[(A)] the coordinate functions $\alpha_i,\,\beta_i$ and $\gamma_i$ only depend on $x$ (and not on $y$ and $z$), and we will generally suppress the dependence on $x$ in our notation.
\end{itemize}
All of our result extend to the more general case but the notation just gets slightly more involved. 

From the right hand side of \eqref{reg3a} we define the function $F_x:(-1,1)^2\to\mathbb{R}^3$ as:
\begin{align}
F_x(\psi,\phi) = &\frac{1}{2}\lp \frac{X_1}{2}{\lp 1+\phi\rp}+\frac{X_4}{2}{\lp 1-\phi\rp}\rp{\lp 1+\psi\rp}+\frac{1}{2}\lp \frac{X_2}{2}{\lp 1+\phi\rp}+\frac{X_3}{2}{\lp 1-\phi\rp}\rp{\lp 1-\psi\rp}, \eqlab{Fx}
\end{align}
for any $x\in \mathcal I$ such that 
\begin{align*}
X_\varepsilon(x,y,z) = F_x\lp\psi\lp\varepsilon^{-1}y\rp,\phi\lp\varepsilon^{-1}z\rp\rp,
\end{align*}
using that $X_i$ only depends upon $x$ by assumption (A).
Furthermore, we will use the function $\tilde{F}_x:{(-1,1)}^2\to \mathbb{R}^2$ in order to refer to the $yz-$components of $F_x$:
\begin{align}
\tilde{F}_x(\psi,\phi) = &\frac{1}{2}\lp \frac{\tilde{X}_1}{2}{\lp 1+\phi\rp}+\frac{\tilde{X}_4}{2}{\lp 1-\phi\rp}\rp{\lp 1+\psi\rp}+\frac{1}{2}\lp \frac{\tilde{X}_2}{2}{\lp 1+\phi\rp}+\frac{\tilde{X}_3}{2}{\lp 1-\phi\rp}\rp{\lp 1-\psi\rp}, \eqlab{FxTil}
\end{align}
where $\tilde{X}_i$ are the projections of the smooth vector fields $X_i$ onto the $yz-$plane:
\begin{eqnarray}
\tilde{X}_i =  \begin{pmatrix}
\beta_i\\ \gamma_i
\end{pmatrix}, i = 1,2,3,4  \eqlab{PrFields}.
\end{eqnarray} 

Using \eqref{SFields} the regularized system is written as:
\begin{eqnarray}
\begin{aligned}
\dot{x}= &\frac{1}{2}\lp \frac{\alpha_1}{2}{\lp 1+\phi\lp \varepsilon^{-1}z\rp\rp}+\frac{\alpha_4}{2}{\lp 1-\phi\lp \varepsilon^{-1}z\rp\rp}\rp{\lp 1+\psi\lp \varepsilon^{-1}y\rp\rp}\\
&+\frac{1}{2}\lp \frac{\alpha_2}{2}{\lp 1+\phi\lp \varepsilon^{-1}z\rp\rp}+\frac{\alpha_3}{2}{\lp 1-\phi\lp \varepsilon^{-1}z\rp\rp}\rp{\lp 1-\psi\lp \varepsilon^{-1}y\rp\rp},\\
\dot{y}= &\frac{1}{2}\lp \frac{\beta_1}{2}{\lp 1+\phi\lp \varepsilon^{-1}z\rp\rp}+\frac{\beta_4}{2}{\lp 1-\phi\lp \varepsilon^{-1}z\rp\rp}\rp{\lp 1+\psi\lp \varepsilon^{-1}y\rp\rp}\\
&+\frac{1}{2}\lp \frac{\beta_2}{2}{\lp 1+\phi\lp \varepsilon^{-1}z\rp\rp}+\frac{\beta_3}{2}{\lp 1-\phi\lp \varepsilon^{-1}z\rp\rp}\rp{\lp 1-\psi\lp \varepsilon^{-1}y\rp\rp},\\
\dot{z}= &\frac{1}{2}\lp \frac{\gamma_1}{2}{\lp 1+\phi\lp \varepsilon^{-1}z\rp\rp}+\frac{\gamma_4}{2}{\lp 1-\phi\lp \varepsilon^{-1}z\rp\rp}\rp{\lp 1+\psi\lp \varepsilon^{-1}y\rp\rp}\\
&+\frac{1}{2}\lp \frac{\gamma_2}{2}{\lp 1+\phi\lp \varepsilon^{-1}z\rp\rp}+\frac{\gamma_3}{2}{\lp 1-\phi\lp \varepsilon^{-1}z\rp\rp}\rp{\lp 1-\psi\lp \varepsilon^{-1}y\rp\rp}.
\end{aligned}
\eqlab{RegSysS}
\end{eqnarray}
The above system is singular for $y=\varepsilon=0$ or $z=\varepsilon=0$. As \eqref{Reg1}, it will therefore again be useful to work with two separate time scales. The time $t$ in \eqref{RegSysS} is the \text{slow time} whereas $\tau=t\varepsilon^{-1}$ will be referred to as the \text{fast time}. We then introduce the variables:
\begin{align}
\hat y = \varepsilon^{-1} y  , \quad
\hat z = \varepsilon^{-1} z. 
\eqlab{varhat}
\end{align}
Inserting  equations \eqref{varhat} into \eqref{RegSysS} gives:
 \begin{eqnarray}
 \begin{aligned}
 \dot {x}= &\frac{1}{2}\lp \frac{\alpha_1}{2}{\lp 1+\phi\lp \hat z\rp\rp}+\frac{\alpha_4}{2}{\lp 1-\phi\lp \hat z\rp\rp}\rp{\lp 1+\psi\lp \hat y\rp\rp}\\
 &+\frac{1}{2}\lp \frac{\alpha_2}{2}{\lp 1+\phi\lp \hat z\rp\rp}+\frac{\alpha_3}{2}{\lp 1-\phi\lp \hat z\rp\rp}\rp{\lp 1-\psi\lp \hat y\rp\rp},\\
 \varepsilon \dot{\hat {y}}= &\frac{1}{2}\lp \frac{\beta_1}{2}{\lp 1+\phi\lp \hat z\rp\rp}+\frac{\beta_4}{2}{\lp 1-\phi\lp \hat z\rp\rp}\rp{\lp 1+\psi\lp \hat y\rp\rp}\\
 &+\frac{1}{2}\lp \frac{\beta_2}{2}{\lp 1+\phi\lp \hat z\rp\rp}+\frac{\beta_3}{2}{\lp 1-\phi\lp \hat z\rp\rp}\rp{\lp 1-\psi\lp \hat y\rp\rp},\\
 \varepsilon \dot{\hat {z}}= &\frac{1}{2}\lp \frac{\gamma_1}{2}{\lp 1+\phi\lp \hat z\rp\rp}+\frac{\gamma_4}{2}{\lp 1-\phi\lp \hat z\rp\rp}\rp{\lp 1+\psi\lp \hat y\rp\rp}\\
 &+\frac{1}{2}\lp \frac{\gamma_2}{2}{\lp 1+\phi\lp \hat z\rp\rp}+\frac{\gamma_3}{2}{\lp 1-\phi\lp \hat z\rp\rp}\rp{\lp 1-\psi\lp \hat y\rp\rp},
 \end{aligned}\eqlab{SFSNew}
 \end{eqnarray}
which is a slow-fast system, in the slow formulation with $x$ being the slow variable and $\hat y$ and $\hat{z}$ being fast. The fast system with respect to the fast time $\tau$ is: 
\begin{eqnarray}
\begin{aligned}
{x}'= &\varepsilon\bigg(\frac{1}{2}\lp \frac{\alpha_1}{2}{\lp 1+\phi\lp \hat z\rp\rp}+\frac{\alpha_4}{2}{\lp 1-\phi\lp \hat z\rp\rp}\rp{\lp 1+\psi\lp \hat y\rp\rp}\\
&+\frac{1}{2}\lp \frac{\alpha_2}{2}{\lp 1+\phi\lp \hat z\rp\rp}+\frac{\alpha_3}{2}{\lp 1-\phi\lp \hat z\rp\rp}\rp{\lp 1-\psi\lp \hat y\rp\rp}\bigg),\\
\hat {y}'= &\frac{1}{2}\lp \frac{\beta_1}{2}{\lp 1+\phi\lp \hat z\rp\rp}+\frac{\beta_4}{2}{\lp 1-\phi\lp \hat z\rp\rp}\rp{\lp 1+\psi\lp \hat y\rp\rp}\\
&+\frac{1}{2}\lp \frac{\beta_2}{2}{\lp 1+\phi\lp \hat z\rp\rp}+\frac{\beta_3}{2}{\lp 1-\phi\lp \hat z\rp\rp}\rp{\lp 1-\psi\lp \hat y\rp\rp},\\
\hat {z}'= &\frac{1}{2}\lp \frac{\gamma_1}{2}{\lp 1+\phi\lp \hat z\rp\rp}+\frac{\gamma_4}{2}{\lp 1-\phi\lp \hat z\rp\rp}\rp{\lp 1+\psi\lp \hat y\rp\rp}\\
&+\frac{1}{2}\lp \frac{\gamma_2}{2}{\lp 1+\phi\lp \hat z\rp\rp}+\frac{\gamma_3}{2}{\lp 1-\phi\lp \hat z\rp\rp}\rp{\lp 1-\psi\lp \hat y\rp\rp}.
\end{aligned}\eqlab{SFFNew}
\end{eqnarray}

\begin{remark} \remlab{dummy}
We note that $(\hat y,\hat z)={\lambda}$ in the dummy system in \protect\cite[Definition 4.1]{jeffrey2014a} for $m=2$ when $\phi=\psi =1$.  
\end{remark}

Setting $\varepsilon=0$ in \eqref{SFSNew} gives the reduced problem:
\begin{eqnarray}
\begin{aligned}
\dot {x}= &\frac{1}{2}\lp \frac{\alpha_1}{2}{\lp 1+\phi\lp \hat z\rp\rp}+\frac{\alpha_4}{2}{\lp 1-\phi\lp \hat z\rp\rp}\rp{\lp 1+\psi\lp \hat y\rp\rp}\\
&+\frac{1}{2}\lp \frac{\alpha_2}{2}{\lp 1+\phi\lp \hat z\rp\rp}+\frac{\alpha_3}{2}{\lp 1-\phi\lp \hat z\rp\rp}\rp{\lp 1-\psi\lp \hat y\rp\rp},\\
0= &\frac{1}{2}\lp \frac{\beta_1}{2}{\lp 1+\phi\lp \hat z\rp\rp}+\frac{\beta_4}{2}{\lp 1-\phi\lp \hat z\rp\rp}\rp{\lp 1+\psi\lp \hat y\rp\rp}\\
&+\frac{1}{2}\lp \frac{\beta_2}{2}{\lp 1+\phi\lp \hat z\rp\rp}+\frac{\beta_3}{2}{\lp 1-\phi\lp \hat z\rp\rp}\rp{\lp 1-\psi\lp \hat y\rp\rp},\\
0= &\frac{1}{2}\lp \frac{\gamma_1}{2}{\lp 1+\phi\lp \hat z\rp\rp}+\frac{\gamma_4}{2}{\lp 1-\phi\lp \hat z\rp\rp}\rp{\lp 1+\psi\lp \hat y\rp\rp}\\
&+\frac{1}{2}\lp \frac{\gamma_2}{2}{\lp 1+\phi\lp \hat z\rp\rp}+\frac{\gamma_3}{2}{\lp 1-\phi\lp \hat z\rp\rp}\rp{\lp 1-\psi\lp \hat y\rp\rp},
\end{aligned}\eqlab{redprob}
\end{eqnarray}
and setting $\varepsilon=0$ in \eqref{SFFNew} gives the layer problem:
\begin{eqnarray}
\begin{aligned}
x'=&0,\\
{\hat{y}}'=&\frac{1}{2}\lp \frac{\beta_1}{2}{\lp 1+\phi\lp\hat{z}\rp\rp}+\frac{\beta_4}{2}{\lp 1-\phi\lp \hat{z}\rp\rp}\rp{\lp 1+\psi\lp \hat{y}\rp\rp}+\frac{1}{2}\lp \frac{\beta_2}{2}{\lp 1+\phi\lp \hat{z}\rp\rp}+\frac{\beta_3}{2}{\lp 1-\phi\lp \hat{z}\rp\rp}\rp{\lp 1-\psi\lp \hat{y}\rp\rp},\\
{\hat{z}}'=&\frac{1}{2}\lp \frac{\gamma_1}{2}{\lp 1+\phi\lp\hat{z}\rp\rp}+\frac{\gamma_4}{2}{\lp 1-\phi\lp \hat{z}\rp\rp}\rp{\lp 1+\psi\lp \hat{y}\rp\rp}+\frac{1}{2}\lp \frac{\gamma_2}{2}{\lp 1+\phi\lp \hat{z}\rp\rp}+\frac{\gamma_3}{2}{\lp 1-\phi\lp \hat{z}\rp\rp}\rp{\lp 1-\psi\lp \hat{y}\rp\rp}.\\
\end{aligned} 
\eqlab{layprob}
\end{eqnarray}
By assumption (A) all $\alpha_i=\alpha_i(x)$, $\beta_i=\beta_i(x)$ and $\gamma_i=\gamma_i(x)$ (as opposed to $\alpha_i(x,0,0),\,\beta_i(x,0,0),\,\gamma_i(x,0,0)$, recall \eqref{later}). Notice that the above layer problem can be written as:
\begin{align}
\begin{pmatrix}
\hat{y}' \\ \hat{z}'
\end{pmatrix} &= \tilde{F}_x(\psi\lp \hat{y}\rp,\phi\lp \hat{z}\rp),
\eqlab{genprob}
\end{align}
and $x'=0$ 
using \eqref{FxTil}.
The critical manifold $C_0$, as the set of equilibria of \eqref{layprob}, can therefore be written in the following form
\begin{align}
C_0 =\lb(x,\hat y,\hat z)~\big\vert~ \tilde{F}_x\lp\psi\lp \hat{y}\rp, \phi\lp \hat{z}\rp\rp = \lp  0,0\rp\rb.
\eqlab{C0}
\end{align}
Generically, $C_0$ is $1D$. Furthermore, it is normally hyperbolic if the eigenvalues of the Jacobian matrix $\textbf{J} = \textbf{J}(\textbf{x})$ of the fast subsystem have non-zero real part. It is attracting (repelling) if both real parts are negative (at least one real part is positive). Finally, it is of saddle type if the eigenvalues are non-zero and of opposite sign.  

For $(x,\hat{y}_*,\hat{z}_*)\in C_0$ the Jacobian matrix of the fast subsystem \eqref{genprob} can be expressed as:
\begin{eqnarray}
{\bf J} = {\textbf{D}{\tilde{F}_x}}{\bf P},
\eqlab{Jac}
\end{eqnarray}
where ${\textbf{D}{\tilde{F}_x}}$ is the Jacobian matrix of $\tilde{F}_x\lp\psi,\phi\rp$, evaluated at $\psi_* = \psi(\hat y_*), \phi_*=\phi(\hat z_*)$,
and where $\textbf{P} = \tl{diag}\lp \psi'_*, \phi'_*\rp$, with $\psi'_*= \psi'(\hat{y}_*)$, $\phi'_*= \phi'(\hat{z}_*)$.
\section{Definition of the sliding flow on $\Lambda$: Extending the Filippov Convention}\seclab{definition}
Analogously to the correspondence between the sliding vector field on a codimension-1 discontinuity set and the reduced problem on a normally hyperbolic critical manifold, recall \thmref{holy}, we will use the reduced problem \eqref{redprob} on $C_0$, obtained from \eqref{SFSNew}$_{\varepsilon=0}$,
to define the sliding vector-field on $\Lambda$. 
For this, let
\begin{align}
\begin{aligned}
\sigma_\psi(x) = \frac{ 1+\psi_*(x)}{2}, \quad
\sigma_\phi(x) = \frac{ 1+\phi_*(x)}{2},
\end{aligned}\eqlab{sigmaphis}
\end{align} 
where $(\psi_*(x), \phi_*(x))\in (-1,1)^2$ are such that $\tilde{F}_x\lp \psi_*(x), \phi_*(x) \rp = \lp 0,0\rp$. 
Then, from \eqref{redprob} follows that the dynamics on $C_0$ is:
\begin{eqnarray}
\dot{x}&=
\lp\sigma_\psi{\sigma_\phi}\rp\alpha_1 + \lp \lp 1-{\sigma_\psi}\rp\sigma_\phi\rp\alpha_2+\lp \lp 1-{\sigma_\psi}\rp\lp 1-{\sigma_\phi}\rp\rp \alpha_3+\lp {\sigma_\psi}\lp 1-{\sigma_\phi}\rp\rp\alpha_4, \eqlab{c0dyn}
\end{eqnarray}
and the coefficients $\lp\sigma_\phi,\sigma_\psi\rp\in\lp 0,1\rp^2$ can be calculated explicitly based on ${\tilde{X}}_{1-4}$. 
\begin{proposition}
Consider \begin{gather}
\begin{gathered}
\textnormal{A} = \det \lp \tilde{X}_1 ~ \tilde{X}_2\rp
~,~\textnormal{B} = \det \lp \tilde{X}_4 ~ \tilde{X}_2 \rp+\det \lp \tilde{X}_1 ~ \tilde{X}_3 \rp
~,~{\Gamma} = \det \lp \tilde{X}_4 ~ \tilde{X}_3 \rp,	\\
\Delta = \textnormal{B}^2 -4\textnormal{A}\Gamma,
\end{gathered}\eqlab{delta}
\end{gather}
and let $\lp\sigma_{\psi},\sigma_{\phi}\rp\in (0,1)^2$ be so that $\psi_*\in (-1,1)$ and $\phi_*\in (-1,1)$ in \eqref{sigmaphis} satisfy $\tilde F_x(\psi_*,\phi_*) = (0,0)$. Then $(x,\hat y_*,\hat z_*)\in C_0$ where $\hat y_*=\psi^{-1}(\psi_*),\,\hat z_*=\phi^{-1}(\phi_*)$. Furthermore, if $\Delta\ge 0$ and $A+\Gamma-B\ne 0$, then the pair $\lp \sigma_\psi,\sigma_\phi \rp$ is given by either of the following expressions
\begin{align*}
 \lp\sigma_{\psi}^+,\sigma_{\phi}^+\rp,\quad \lp\sigma_{\psi}^-,\sigma_{\phi}^-\rp,
\end{align*}	
where
\begin{align}
\sigma_{\phi}^{(\pm)} = \frac{2\Gamma - \textnormal{B}\pm\sqrt\Delta}{2\lp \textnormal{A}+\Gamma -\textnormal{B}\rp},\qquad \sigma_\psi^{(\pm)} = \frac{\beta_2\sigma_\phi^{(\pm)}+\beta_3\lp 1-\sigma_\phi^{(\pm)}\rp}{\lp \beta_2-\beta_1\rp\sigma_\phi^{(\pm)}+\lp \beta_3-\beta_4\rp\lp 1-\sigma_\phi^{(\pm)}\rp}.
\eqlab{sigmas}
\end{align}
If $\textnormal{A}+\Gamma-\textnormal{B} = 0$, then
\begin{align}
\sigma_\phi = \frac{{\Gamma}}{2\Gamma-\textnormal{B}},\qquad \sigma_\psi= \frac{\beta_2\sigma_\phi+\beta_3\lp 1-\sigma_\phi\rp}{\lp \beta_2-\beta_1\rp\sigma_\phi+\lp \beta_3-\beta_4\rp\lp 1-\sigma_\phi\rp}. \eqlab{sigmassame}
\end{align}
\proplab{sigmas}
\end{proposition}

\begin{proof}
	For $\sigma_\psi = \lp 1+\psi_*\rp/2$, $\sigma_\phi = \lp 1+\phi_*\rp/2$,	where $(\psi_*,\phi_*)$ is a solution of the algebraic equations of \eqref{redprob}, the  algebraic equations  of \eqref{redprob} are written as
	\begin{align}
	\begin{aligned}
	\lp {\beta_1}\sigma_\phi+{\beta_4}{\lp 1-\sigma_\phi \rp}\rp{\sigma_\psi}+\lp {\beta_2}\sigma_\phi+{\beta_3}{\lp 1-\sigma_\phi \rp}\rp{\lp 1-\sigma_\psi\rp}&=0,\\
	\lp {\gamma_1}\sigma_\phi+{\gamma_4}{\lp 1-\sigma_\phi \rp}\rp{\sigma_\psi}+\lp {\gamma_2}\sigma_\phi+{\gamma_3}{\lp 1-\sigma_\phi \rp}\rp{\lp 1-\sigma_\psi\rp}&=0,
	\end{aligned}
	\eqlab{finds}
	\end{align}
	from which follows that
	\begin{align}
	-\frac{\sigma_\psi}{\lp 1-\sigma_\psi\rp}&=
	\frac{\lp {\beta_2}\sigma_\phi+{\beta_3}{\lp 1-\sigma_\phi \rp}\rp}{\lp {\beta_1}\sigma_\phi+{\beta_4}{\lp 1-\sigma_\phi \rp}\rp}
	= 
	\frac{\lp {\gamma_2}\sigma_\phi+{\gamma_3}{\lp 1-\sigma_\phi \rp}\rp}{\lp {\gamma_1}\sigma_\phi+{\gamma_4}{\lp 1-\sigma_\phi \rp}\rp} . \eqlab{pairing}
	\end{align}
	We can therefore eliminate $\sigma_\psi$ and obtain the following quadratic equation for $\sigma_\phi$
	\begin{gather*}
	\lp \beta_1\gamma_2-\beta_2\gamma_1\rp\sigma_\phi^2+\lp \beta_4\gamma_2-\beta_2\gamma_4+\beta_1\gamma_3-\beta_3\gamma_1\rp\lp 1-\sigma_\phi\rp\sigma_\phi+\lp \beta_4\gamma_3-\beta_3\gamma_4\rp\lp 1-\sigma_\phi\rp^2=0,
	\end{gather*}
	and can be therefore written as (see \eqref{PrFields})
	\begin{align*}
	\underbrace{\det\lp \tilde{X}_1~ \tilde{X}_2\rp}_{\tl{A}}\sigma_\phi^2+\underbrace{\lb \det\lp \tilde{X}_4~ \tilde{X}_2\rp +\det\lp \tilde{X}_1~ \tilde{X}_3\rp\rb}_{\tl{B}}\lp 1-\sigma_\phi\rp\sigma_\phi+\underbrace{\det\lp \tilde{X}_4~  \tilde{X}_3\rp}_{\Gamma}\lp 1-\sigma_\phi\rp^2=0,
	\end{align*}
	where $\tilde{X}_i$ are column vectors. We then obtain the quadratic equation
	\begin{align}
	\lp \tl{A}+{\Gamma}-\tl{B} \rp\sigma_\phi^2+\lp \tl{B}-2\Gamma\rp\sigma_\phi+\Gamma=0.
	\eqlab{quadratic}
	\end{align}
	Expressions \eqref{sigmas} and \eqref{sigmassame} follow from equations \eqref{pairing} and \eqref{quadratic}. 
\end{proof}

The quantities $\tl{A}+\Gamma-\tl{B}$ and $\tl{B}-2\Gamma$ are sums of oriented areas of parallelograms that are formed by the vectors $\tilde{X}_i$. \figref{proj1} illustrates some examples of PWS vector fields for which $\tl{A}+\Gamma-\tl{B} = 0$.

We now propose the following alternative definition of a sliding vector-field.
\begin{definition}
	\textbf{(Sliding flow as the dynamics on the critical manifold of the regularized system)}
	Consider the PWS system $(X_1,X_2, X_3,X_4)$ on $\mathcal Q_1 \cup \mathcal Q_2\cup \mathcal Q_3 \cup \mathcal Q_4$. The sliding region $\Lambda^{sl}\subset\Lambda$ is then defined as:
	\begin{align*}
	\Lambda^{sl} = \lb (x,0,0)\in \Lambda ~\big\lvert~ \exists ~\lp \sigma_\psi(x), \sigma_\phi(x)\rp\in(0,1)^2 ~\rb.
	\end{align*}
	where $\sigma_\psi(x), \sigma_\phi(x)$ are given by \propref{sigmas}, and the \textnormal{sliding vector field} $X^{sl}$ on $\Lambda^{sl}$ is defined by the reduced vector-field on the critical manifold $C_0$ of the slow-fast regularized system \eqref{reg3a}. In details,
	\begin{gather}
	X_{sl}(\textbf x)=\begin{pmatrix}
	\lp\sigma_\psi{\sigma_\phi}\rp\alpha_1 + \lp \lp 1-{\sigma_\psi}\rp\sigma_\phi\rp\alpha_2+\lp \lp 1-{\sigma_\psi}\rp\lp 1-{\sigma_\phi}\rp\rp \alpha_3+\lp {\sigma_\psi}\lp 1-{\sigma_\phi}\rp\rp\alpha_4 \\
	0\\
	0
	\end{pmatrix},\quad 
\textbf{x}=(x,0,0) \in \Lambda^{sl}.\eqlab{Xsl2}
	\end{gather}
The stability of the sliding flow is defined by the stability of the corresponding equilibrium point $(x,\psi^{-1}(\psi_*(x)),\phi^{-1}(\phi_*(x)))$ (recall \eqref{sigmaphis})  of \eqref{layprob}: Let $\lambda_1$ and $\lambda_2$ be the eigenvalues of the Jacobian $\textbf J$ in \eqref{Jac} and suppose that $\text{Re}\,\lambda_i\ne 0$ for $i=1,2$. Then the sliding flow is said to be attracting (repelling) if $\text{Re}\,\lambda_i <0$ ($\text{Re}\,\lambda_i>0$), for $i=1,2$, and of saddle type $\lambda_1\lambda_2<0$.
	\defnlab{slide2}
\end{definition}

Similarly to the case of codimension-1 discontinuities, we make the following important observation.

\begin{theorem}\thmlab{holy2}
A vector-field is a sliding vector-field of \defnref{slide2Fil} if and only if it is a sliding vector-field of \defnref{slide2}.
\thmlab{slideEquiv}
\end{theorem}
\begin{proof}
We write \eqref{Xsl2} as \eqref{slide} by setting
\begin{align}
	{\nu_1(x)} = \sigma_\psi{\sigma_\phi},\quad
	{\nu_2(x)} = \lp 1-{\sigma_\psi}\rp\sigma_\phi,\quad
	{\nu_3(x)} = \lp 1-{\sigma_\psi}\rp\lp 1-{\sigma_\phi}\rp,\quad
	{\nu_4(x)} = {\sigma_\psi}\lp 1-{\sigma_\phi}\rp,
	\eqlab{nus}
	\end{align}
	where $(\sigma_\psi,\sigma_\phi)$ are given in \propref{sigmas}.
	Clearly $\sum_{i}\nu_i(x) = 1$ and $X^{sl}$ is tangent to $\Lambda$ by construction. We can similarly write \eqref{slide} as \eqref{Xsl2} going the other way using \propref{sigmas}. 
\end{proof}

Although \defnref{slide2Fil} and \defnref{slide2} are equivalent in terms the sliding vector field, the concept of its stability is only defined in \defnref{slide2}. 

The sliding flow is expressed in the form \eqref{Xsl2} \textit{under the assumption} that $\nu_{1-4}\in(0,1)$ given by \eqref{nus} exist, see also \protect\cite[Theorem 1]{guglielmi2017a}. In \propref{sigmas}, we give the expressions of the coefficients $(\sigma_\psi,\sigma_\phi)$ of the sliding vector fields in closed form, however we are not yet in a position to know a priory if and how many sliding vector fields exist on $\Lambda$, without calculating  $\sigma_\psi^{(\pm)}$ and $\sigma_\phi^{(\pm)}$. In principle, (see also \protect\cite{jeffrey2014a}) there could exist zero, one or two pairs of coefficients $(\sigma_\psi,\sigma_\phi)$ defining the critical manifold $C_0$, and therefore zero, one or two sliding vector fields defined on $\Lambda$. 

In the following, we will apply geometric approach, using the canopy in \protect\cite{jeffrey2014a}, to derive simple criteria that determine the existence and multiplicity of the sliding flow on $\Lambda$, based only on the smooth vector fields $\tilde{X}_{1-4}$.

\section{The parametric surface induced by the regularization}\seclab{para} The parametrization $F_x:(-1,1)^2\to\mathbb{R}^3$ that is given by \eqref{Fx} and that is induced by the regularized system \eqref{reg3a} defines a surface $\mathcal{S} = F_x(\lp-1,1\rp^2)\subset \mathbb R^3$, as illustrated in 	\figref{projes}.  The boundaries of the smooth surface $\mathcal{S}$ are the straight segments that connect the endpoints of consequent $X_i$, i.e. $\psi\mapsto F(\psi,1)$ is a straight line connecting $X_1$ to $X_2$, $\phi\mapsto F(-1,\phi)$ is a straight line connecting $X_2$ to $X_3$, $\psi\mapsto F(\psi,-1)$ is a straight line connecting $X_3$ to $X_4$, $\phi\mapsto F(1,\phi)$ is a straight line connecting $X_4$ to $X_1$ (see \figref{projes} (a)). Recall that this parametrization is related to the regularized system by:
\begin{eqnarray*}
{X}_\varepsilon\lp x,\varepsilon\hat{y}, \varepsilon\hat{z}\rp = F_x\lp \psi(\hat{y}), \phi(\hat{z})\rp
\end{eqnarray*}
and the subscript $x$ is to denote that every point on $\Lambda$ defines a different surface, since all $X_i$ depend on $x$. Since $F_x$ is bilinear, the surface $\mathcal{S}$ is a doubly ruled surface, and in  case it is a regular surface, it corresponds to a bounded hyperbolic paraboloid. This surface is called canopy in \protect\cite{jeffrey2014a}. 

\begin{figure}[h!]
	~
	\centering
	\begin{subfigure}[b]{0.4\textwidth}
		\centering
		\includegraphics[scale = 0.19]{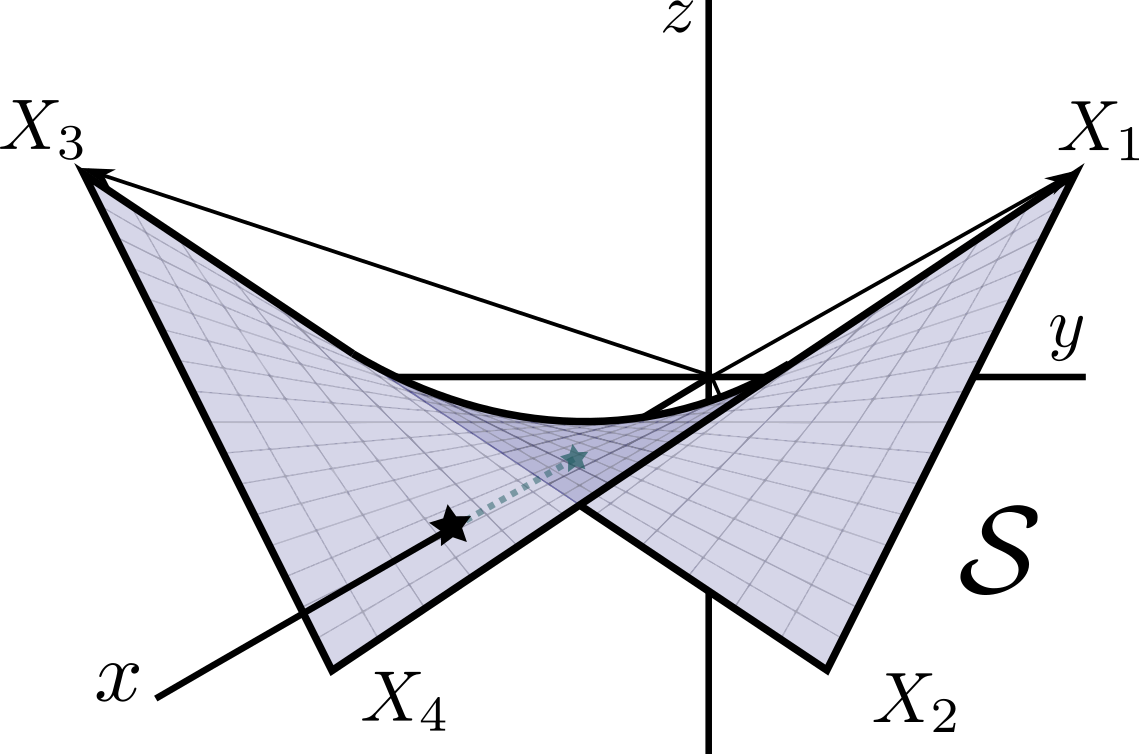}
		\caption{ The parametric surface $\mathcal{S}$ in $\mathbb{R}^3$}
	\end{subfigure}
	~
	\begin{subfigure}[b]{0.5\textwidth}
		\centering
		\includegraphics[scale = 0.19]{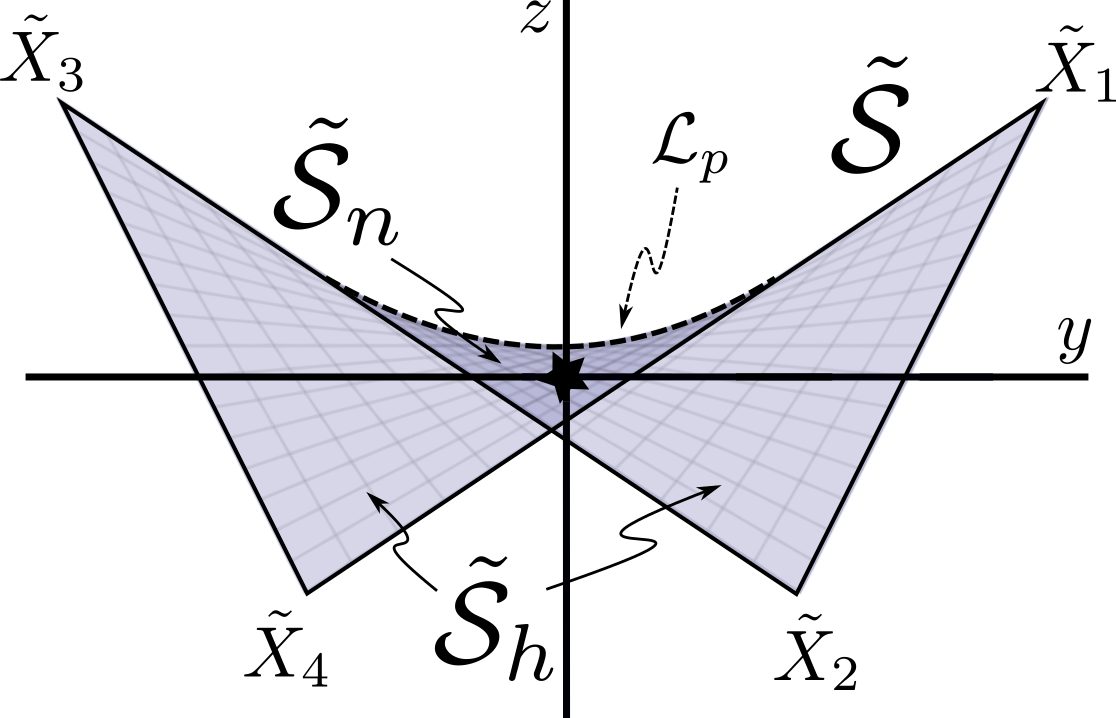}
		\caption{The projection of $\mathcal{S}$ onto the $yz-$plane.}
	\end{subfigure}
	\caption{The parametric surface $\mathcal{S}$ intersects with $\Lambda$ if and only if the origin of the $yz-$plane is contained in $\tilde{\mathcal{S}}$. As illustrated, if the origin is located inside the dark gray region (i.e. the nonhomeomorphic region $\tilde{\mathcal{S}}_n$), then $\mathcal{S}$ intersects with $\Lambda$ at exactly two points. On the other hand, if the origin were located inside the light gray region (i.e. the homeomorphic region $\tilde{\mathcal{S}}_h$), then $\mathcal{S}$ would intersect with $\Lambda$ at only one point. Notice that the edges connecting the end-points $X_i$ in (a) do not belong to the surface $\mathcal S$; they correspond to either $\psi=\pm 1$ or $\phi=\pm 1$. However, when projecting $\mathcal S$ onto the $yz$-plane as seen in (b), the edges connecting $\tilde X_1$ with $\tilde X_4$ and $\tilde X_2$ with $\tilde X_3$ each become divided into two parts; one part, which is the boundary of the subset $\tilde{\mathcal{S}}_n$, see \defnref{tildeSnDefn}, and therefore belongs to $\tilde{\mathcal{S}}$, and another part which is not part of $\tilde{\mathcal{S}}$. }
	\figlab{projes}	
\end{figure}

\begin{proposition}\proplab{jeffrey2014aProp}
	(See also \protect\cite{jeffrey2014a}) A sliding vector field exists at $\textbf{x}=(x,0,0)\in \Lambda$ if and only if $T_\textbf{x}\Lambda\simeq$ the $x$-axis intersects $\mathcal{S}$ at some $F_x(\lp \psi_*,\phi_*\rp)$. The magnitude of the sliding vector is equal to the $x-$component of $F_x\lp \psi_*,\phi_*\rp$.
\end{proposition}

\begin{proof}
		Consider $\lp \psi_*,\phi_*\rp$ for which $\mathcal{S}$ intersects with $\Lambda$ at $F_x\lp \psi_*,\phi_*\rp$. Then $\lp \psi_*,\phi_*\rp$ are such that $\tilde{F}_x\lp \psi_*,\phi_*\rp=(0,0)$, and we have:
	\begin{align*}
	F_x(\psi_*,\phi_*) = \begin{pmatrix}
	\nu_1\alpha_1+\nu_2\alpha_2+\nu_3\alpha_3+\nu_4\alpha_4\\
	0\\
	0
	\end{pmatrix},
	\end{align*}
	 where the $x-$component of $F_x(\psi_*,\phi_*)$ gives the sliding vector field on $\Lambda$ (see \defnref{slide2} and \thmref{slideEquiv}). On the other hand, consider $\lp \psi_*,\phi_*\rp\in (-1,1)^2$ that via \eqref{sigmaphis} and \defnref{slide2} give a sliding vector field on $\Lambda$. Then again we have $\tilde{F}_x\lp \psi_*,\phi\rp=(0,0)$, and therefore the point ${F}_x\lp \psi_*,\phi_*\rp$ in $\mathcal{S}$ lies on $\Lambda$.
\end{proof}

\subsection{The projection of $\mathcal{S}$: Existence and multiplicity of the sliding flow} 
Let $\tilde{\mathcal{S}}$ be the projection of ${\mathcal{S}}$ onto the $yz$-plane. This region is given $\tilde{F}_x((-1,1)^2)$ using the parametrization $\tilde{F}_x$ in \eqref{FxTil}. Then by \propref{jeffrey2014aProp}, ${\mathcal{S}}$ intersects with $\Lambda$ if and only if the origin of the $yz$-plane is contained within $\tilde {\mathcal{S}}$, and the multiplicity of the sliding vector field depends on where the origin is located in $\tilde {\mathcal{S}}$.

For example, \figref{projes} (a) illustrates a case where $\mathcal{S}$ intersects $\Lambda$ twice. These two points project to the same point in $\tilde{\mathcal{S}}$, as shown in \figref{projes} (b). The collection of all such points in $\mathcal{S}$, i.e. where the projection $\mathcal{S}\to\tilde{\mathcal{S}}$ is two-to-one, makes out the dark shaded region in \figref{projes} (b). Following the definition below we will refer to this region as the nonhomeomorphic region. In contrast, the light gray areas in \figref{projes} (b) all lift (by the preimage of $\mathcal S\rightarrow \tilde {\mathcal S}$) to single points on the set ${\mathcal{S}}$. This set will be called the homeomorphic region. 

\begin{definition}\defnlab{tildeSnDefn}
	The subset $\tilde{\mathcal{S}}_n\subseteq\tilde{\mathcal{S}}$ of points $(y,z)\in \tilde{\mathcal S}$ for which the cardinality $\#F^{-1}(y,z)$ is $2$, will be called the \textnormal{nonhomeomorphic} region of $\tilde{\mathcal S}$ (dark gray area of $\tilde{\mathcal{S}}$ in \figref{projes}). In case $\tilde{\mathcal{S}}_n\neq\emptyset$, the curved line $\mathcal{L}_p$ that is both a boundary of $\tilde{\mathcal{S}}_n$ and of  $\mathcal{\tilde{S}}$ will be called the \textnormal{parabolic line}. The subset $\tilde{\mathcal{S}}_h\subseteq\tilde{\mathcal{S}}$ of points $(y,z)\in \tilde{\mathcal S}\backslash \mathcal L_p$ for which the cardinality $\#F^{-1}(y,z)$ is $1$, will be called the \textnormal{homeomorphic} region of $\tilde{\mathcal{S}}$ (light gray area of $\tilde{\mathcal{S}}$ in \figref{projes}). 
%
%
%
%
\end{definition}

Notice that the restriction $\tilde{F}_{x,n}=\tilde{F}_{x}\vert_{\tilde{F}_x^{-1}\lp \tilde{S}_n\rp}:\tilde F_x^{-1}\lp \tilde{S}_n\rp\to\tilde{S}_n$ is two-to-one whereas $\tilde{F}_{x,h}=\tilde{F}_{x}\vert_{\tilde{F}_x^{-1}\lp \tilde{S}_h\rp}:\tilde{F}_x^{-1}\lp \tilde{S}_h\rp\to\tilde{S}_h$ is a homeomorphism. Notice also that $\tilde{\mathcal S} = \tilde{\mathcal S}_n \sqcup \mathcal L_p\sqcup \tilde{\mathcal S}_h$ and that only $\tilde S_n$ is open in general. 
Based on the above, we state the following corollary concerning the existence and multiplicity of the sliding flow on $\Lambda$.

\begin{corollary} \textnormal{(Existence and multiplicity of the sliding flow on $\Lambda$)}
	If the origin of the $yz$-plane is contained in $\tilde{\mathcal{S}}_h\sqcup \mathcal L_p$, then there exists a unique sliding vector field on $\Lambda$. If the origin of the $yz$-plane is contained in $\tilde{\mathcal{S}}_n$, then there exists a pair of sliding vector fields. If the origin of the $yz$-plane is not contained in $\tilde{\mathcal{S}}$, then there exists no sliding vector field on $\Lambda$.
	\corlab{cor1}
\end{corollary}

In order to be able to derive geometric criteria on the existence and multiplicity of the sliding flow, it is essential to distinguish among the various possible shapes of $\tilde{\mathcal{S}}$.

\subsection{Distinguishing the projections}
Since $\tilde{\mathcal{S}}$ is the projection of a bounded hyperbolic paraboloid, the three possible shapes of $\tilde{\mathcal{S}}$, depending on the rotation of $\mathcal{S}$, are the ones illustrated in the first row of \figref{int0}. The second row contains the quadrilaterals that we obtain by connecting the endpoints of subsequent $\tilde{X}_{i}$ and that are associated with the above projections. Using these quadrilaterals, we will be able to distinguish among these projections based on $\tilde{X}_{1-4}$.

\begin{figure}[h!]
	\centering
	\begin{subfigure}[b]{0.3\textwidth}
		\centering
		\includegraphics[scale = 0.16]{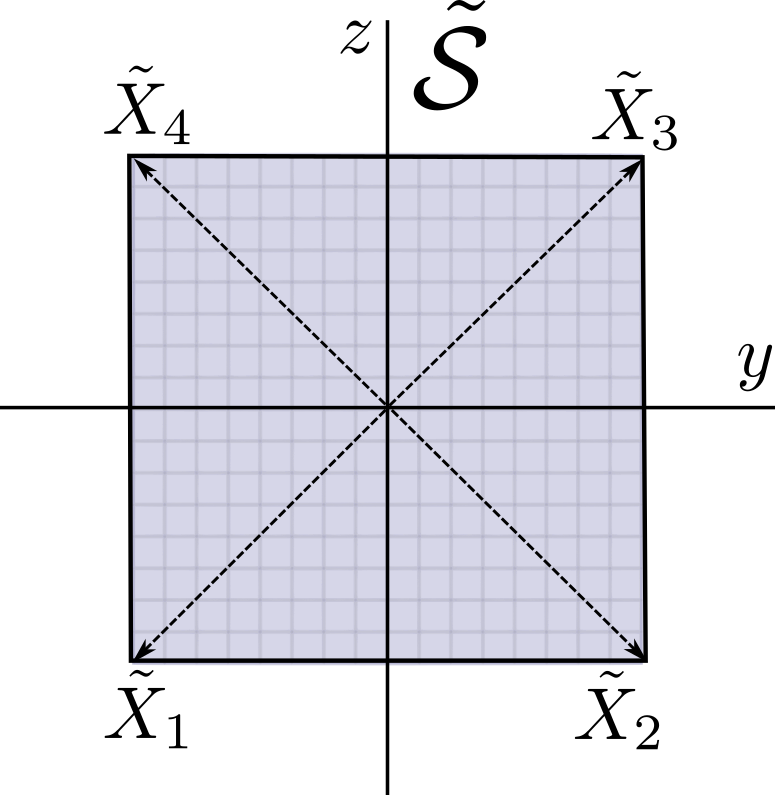}
	\end{subfigure}
	~
	\begin{subfigure}[b]{0.3\textwidth}
		\centering
		\includegraphics[scale = 0.16]{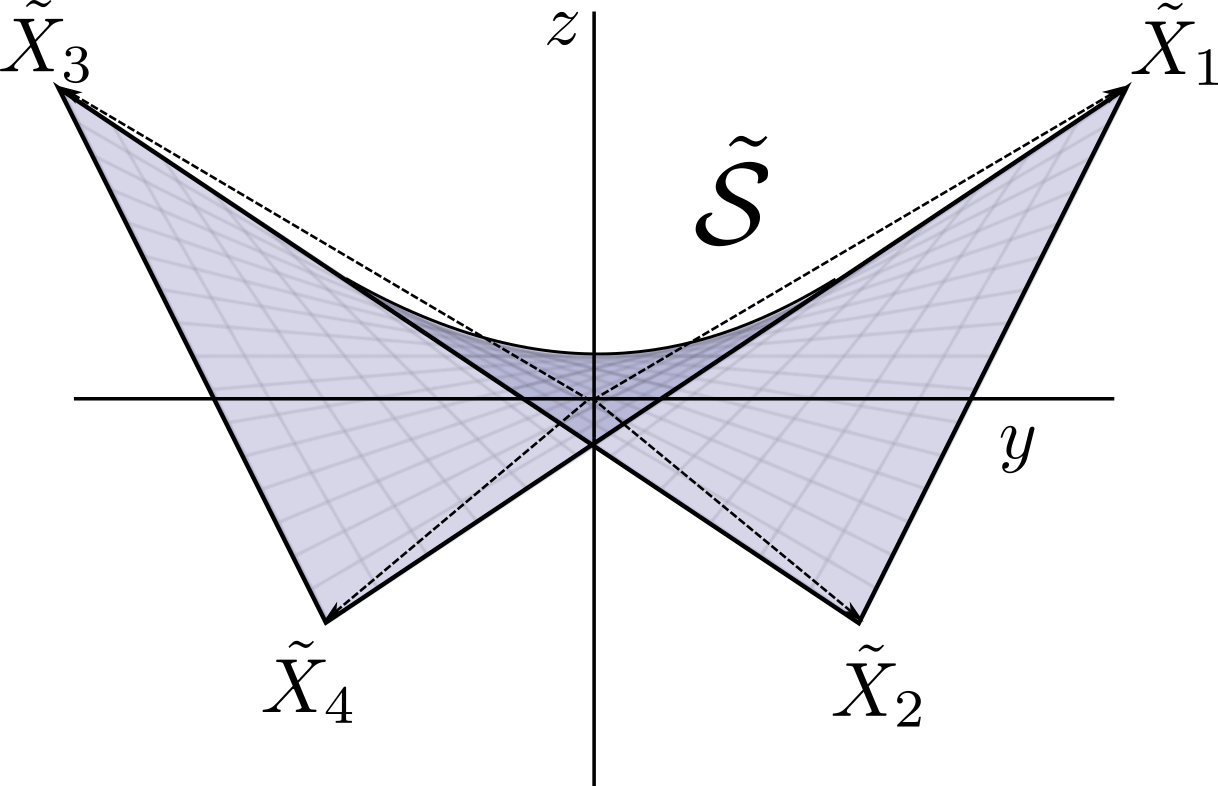}
	\end{subfigure}
	~
	\begin{subfigure}[b]{0.3\textwidth}
		\centering
		\includegraphics[scale = 0.16]{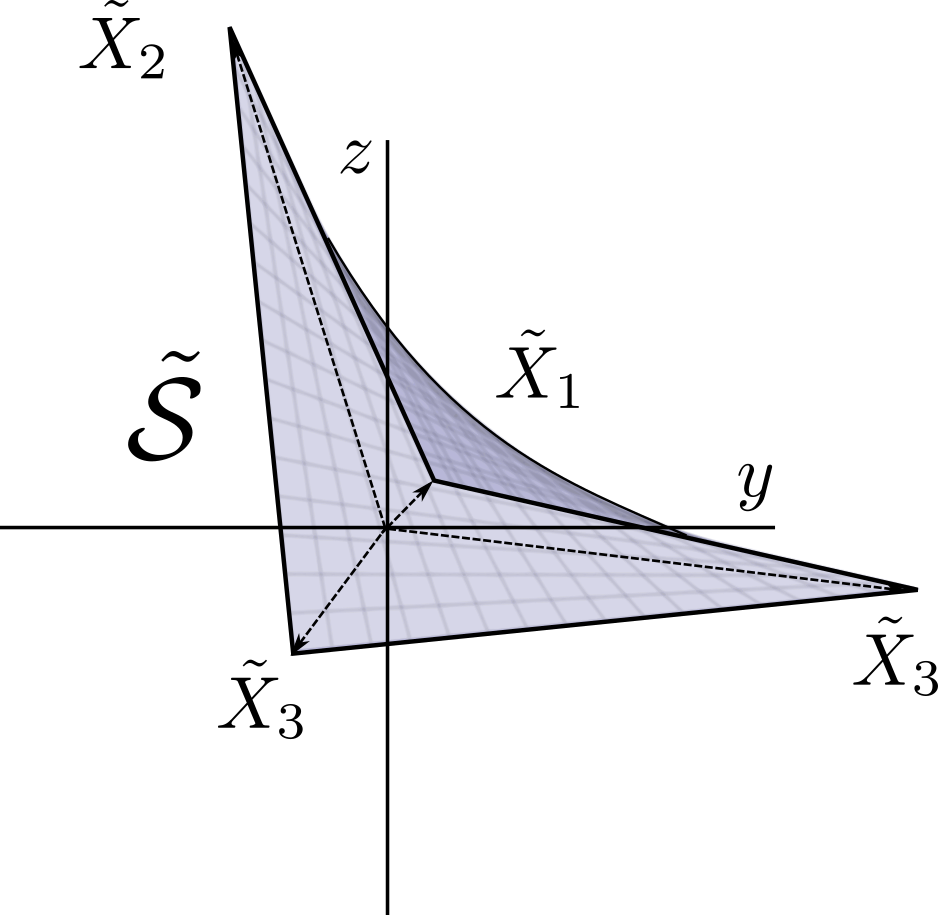}
	\end{subfigure}\\
	\begin{subfigure}[b]{0.3\textwidth}
		\centering
			\includegraphics[scale = 0.16]{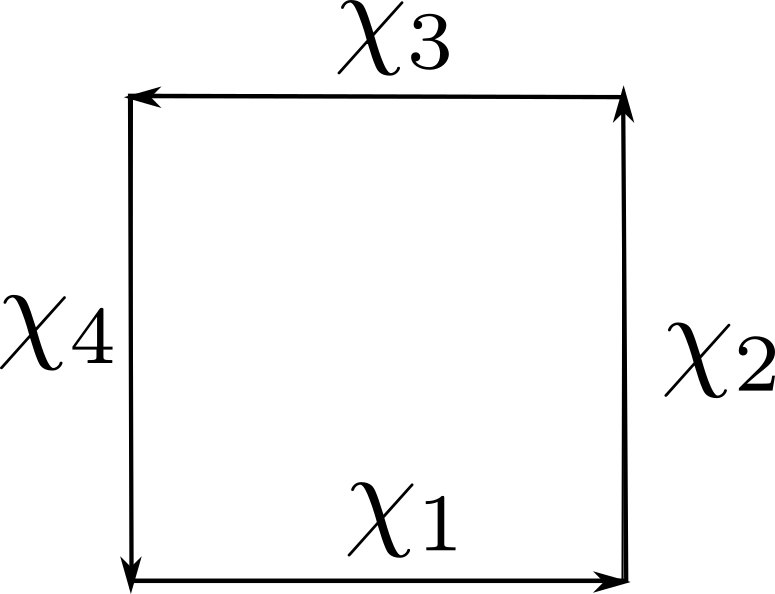}
		\caption{Convex}
	\end{subfigure}
	~
	\begin{subfigure}[b]{0.3\textwidth}
		\centering
		\includegraphics[scale = 0.16]{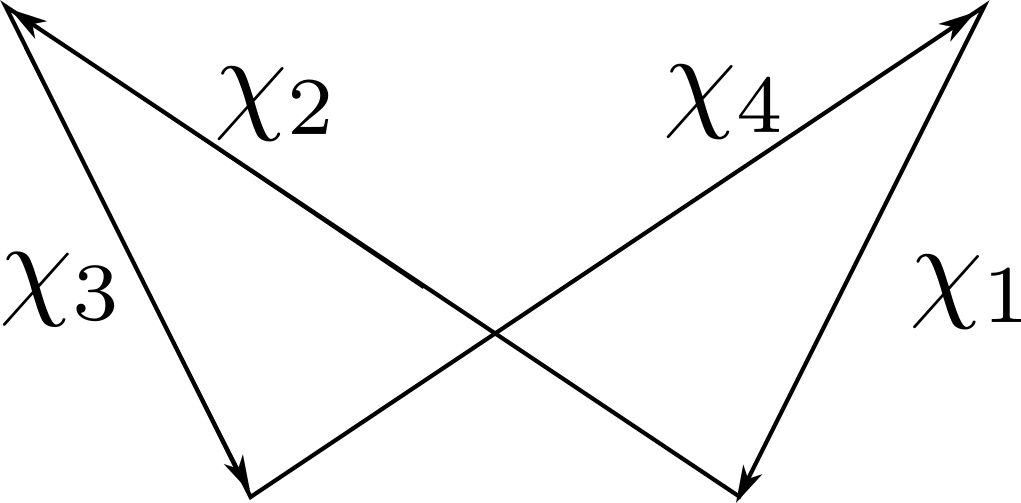}
		\caption{Crossed}
	\end{subfigure}
	~
	\begin{subfigure}[b]{0.3\textwidth}
		\centering
		\includegraphics[scale = 0.16]{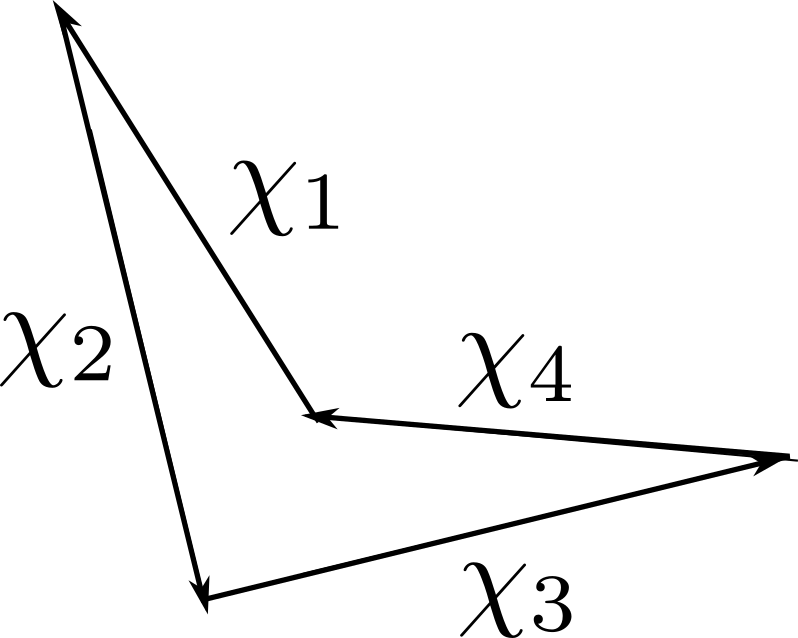}
		\caption{Concave}
	\end{subfigure}
	\caption{The three possible shapes of the projection of $\mathcal{S}$ onto the $yz-$plane (first row) and the three possible quadrilaterals that are obtained by connecting the endpoints of subsequent $\tilde{X}_{1-4}$ according to 	\defnref{DistCas} (second row).}
	\figlab{int0}
\end{figure}

\begin{definition}
	Define the \textnormal{difference vectors} $\chi_i$ and the \textnormal{difference determinants} $\delta_i$ as:
	\begin{eqnarray*}
	\chi_i = 
	\tilde{X}_{i+1}-\tilde{X}_i
	,\quad
	\delta_i = 
	\det\lp \chi_i~\chi_{i+1}\rp, ~~i = 1,2,3,4.
	\end{eqnarray*}
The geometric shape formed by connecting the endpoints of  consequent $\tilde{X}_i$ (i.e. $\tilde{X}_1$ with $\tilde{X}_2$, $\tilde{X}_2$ with $\tilde{X}_3$, $\tilde{X}_3$ with $\tilde{X}_4$, $\tilde{X}_4$ with $\tilde{X}_1$) will be called the \textnormal{projected quadrilateral}. 
	\begin{enumerate}
		\item If the difference determinants $\delta_{1-4}$ are all of the same sign, then the projected quadrilateral will be called a \textnormal{convex quadrilateral} and $\tilde{\mathcal{S}}$ will be called a \textnormal{convex projection}. 
		\item If two of the difference determinants  $\delta_{1-4}$ are positive and the other two are negative, then the projected quadrilateral will be called a \textnormal{crossed quadrilateral} and $\tilde{\mathcal{S}}$ will be called a \textnormal{crossed projection}.
		\item If three of the difference determinants  $\delta_{1-4}$ are of the same sign and the remaining one is of opposite sign, then the projected quadrilateral will be called a \textnormal{concave quadrilateral} and $\tilde{\mathcal{S}}$ will be called a \textnormal{concave projection}.
	\end{enumerate}
	\defnlab{DistCas}
\end{definition}
{The observation that the origin of the $yz$-plane must be contained in such a quadrilateral in order for a sliding vector field to exist was also made in \protect\cite{dieci2017moments}, for the case of ``generally attracting'' intersection of switching manifolds.}
\section{Criteria on the Existence and Multiplicity of the Sliding Vector Field on $\Lambda$}\seclab{existence}
Here we will describe geometrically inspired conditions on the existence and multiplicity of the sliding flow, for the different cases of the quadrilateral projections described in \defnref{DistCas}. An important conclusion of this section is that the existence and multiplicity of the sliding vector field depend only on the shape of the projection, {i.e. only on $\tilde{X}_{1-4}$ and not on the choice of regularization,} and for any fixed projection the same conditions hold for all symmetric transformations (rotation, reflexion, time reversal).
\subsection{The Convex Cases}
For the convex cases, it always holds that $\tilde{\mathcal{S}}_n = \emptyset$ and $\tilde{\mathcal{S}} = \tilde{\mathcal{S}}_h$, hence there could exist either zero or one sliding vector field on $\Lambda$.	 
\begin{proposition}
	Assume that $\tilde{\mathcal{S}}$ is a convex projection, according to \defnref{DistCas}. A unique sliding vector field is defined on the codimension-2 discontinuity $\Lambda$ of the PWS system \eqref{X14} if and only if:
	\begin{eqnarray*}
	\tl{Condition 1: }\det\lp \tilde{X}_1 ~ \tilde{X}_2\rp\det\lp \tilde{X}_3 ~ \tilde{X}_4\rp > 0,\\
	\tl{Condition 2: } \det\lp \tilde{X}_2 ~ \tilde{X}_3\rp\det\lp \tilde{X}_4 ~ \tilde{X}_1\rp > 0.
	\end{eqnarray*}
	\proplab{propdef}
\end{proposition}

\begin{figure}[h!]
	\centering
	\includegraphics[scale=0.13]{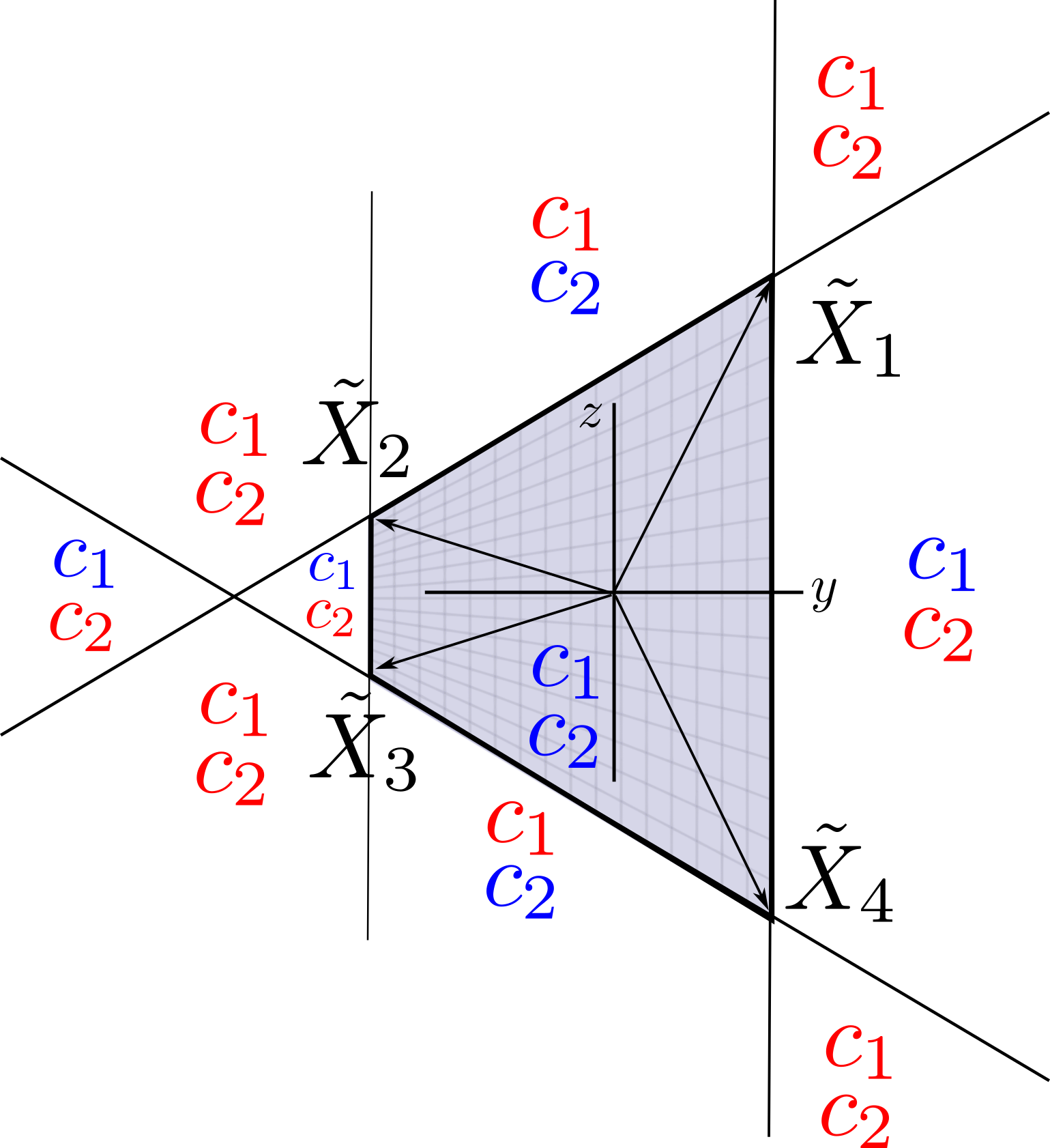}
	\caption{Conditions 1 \& 2 of \propref{propdef} hold only in the case for which the origin is located inside the convex quadrilateral. If the origin were located in any other region outside $\tilde{\mathcal{S}}$, then at least one of the two conditions would be violated, and the potentially violated conditions are indicated by red color.} 
	\figlab{convexfig}
\end{figure}

\begin{proof}
	The proof is based on geometric arguments referring to \figref{convexfig}. The lines connecting subsequent endpoints of $\tilde{X}_i$ separate the $yz-$plane into distinct regions (\figref{convexfig}). If the origin is contained inside the convex quadrilateral obtained by connecting subsequent endpoints $\tilde{X}_i$ with straight segments, then Conditions 1 and 2 are satisfied, as can be easily verified using the right-hand rule. If the origin were ``moved'' to another region (with the quadrilateral shape being fixed), it would have to cross one of the lines connecting subsequent endpoints of $\tilde{X}_i$, thus one of the determinants $\det\lp \tilde{X}_i~\tilde{X}_{i+1}\rp$  would change its sign and one of the two conditions would be violated. In  \figref{convexfig}, ``$c_1$'' is used to denote Condition 1 and ``$c_2$'' is used to denote Condition 2. In every region, the blue font-color is used to indicate that the  respective condition is satisfied and the red font-color is used to indicate that the respective condition is violated, in case the origin is contained in that region.
\end{proof}

\subsection{The Crossed Cases}	
Essentially, the crossed projections reduce to the cases where $\chi_1$ is either an edge or a diagonal (\figref{cross-proof}). All cases are then obtained by symmetry (rotation, reflection, time reversal).

Generically $\tilde{\mathcal{S}}_h\neq \emptyset$, $\tilde{\mathcal{S}}_n\neq \emptyset$,  and a unique sliding vector field exists on $\Lambda$ if the origin of the $yz-$plane is contained in $\tilde{\mathcal{S}}_h$. Two different sets of conditions describe the cases where the vector $\chi_1$ is either an edge or a diagonal of the quadrilateral.
\begin{proposition}
	Assume that $\tilde{\mathcal{S}}$ is a crossed projection, according to \defnref{DistCas}. In the case where $\chi_1$ is an edge, if:
	\begin{eqnarray*}
	\tl{($\chi_1$-edge) } 
	\begin{cases} \tl{Condition 1: }\det\lp \tilde{X}_1 ~ \tilde{X}_2\rp\det\lp \tilde{X}_3 ~ \tilde{X}_4\rp < 0,\\
	\tl{Condition 2: } \det\lp \tilde{X}_2 ~ \tilde{X}_3\rp\det\lp \tilde{X}_4 ~ \tilde{X}_1\rp > 0,
	\end{cases}
	\end{eqnarray*}
	then a unique sliding vector field is defined on the codimension-2 discontinuity $\Lambda$ of the PWS system \eqref{X14}. In the case where $\chi_1$ is a diagonal, if:
	\begin{eqnarray*}
	\tl{($\chi_1$-diagonal)} 
	\begin{cases} \tl{Condition 3: }\det\lp \tilde{X}_1 ~ \tilde{X}_2\rp\det\lp \tilde{X}_3 ~ \tilde{X}_4\rp > 0,\\
	\tl{Condition 4: } \det\lp \tilde{X}_2 ~ \tilde{X}_3\rp\det\lp \tilde{X}_4 ~ \tilde{X}_1\rp < 0,
	\end{cases}
	\end{eqnarray*}
	then a unique sliding vector field is defined on the codimension-2 discontinuity $\Lambda$ of the PWS system \eqref{X14}. 
	\proplab{crossed}
\end{proposition}

\begin{figure}[h!]
	~
	\centering
	\begin{subfigure}[b]{0.4\textwidth}
		\centering
		\includegraphics[scale = 0.18]{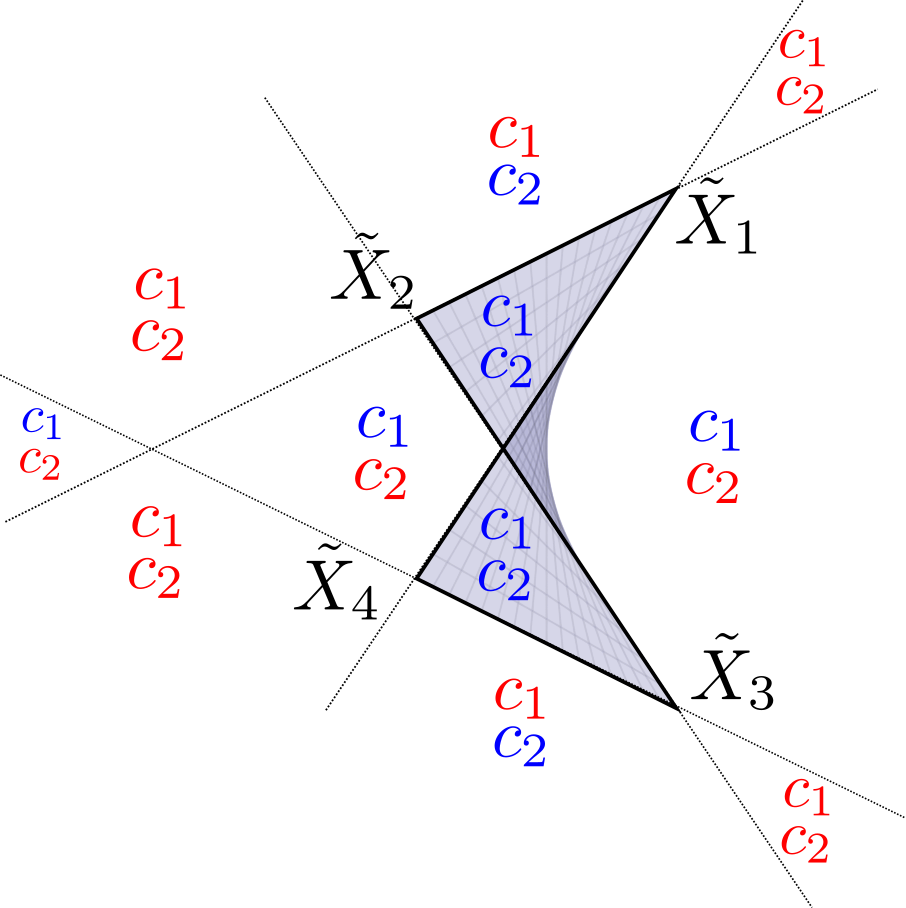}
		\caption{$\chi_1$-edge}
	\end{subfigure}
	~
	\begin{subfigure}[b]{0.5\textwidth}
		\centering
		\includegraphics[scale = 0.18]{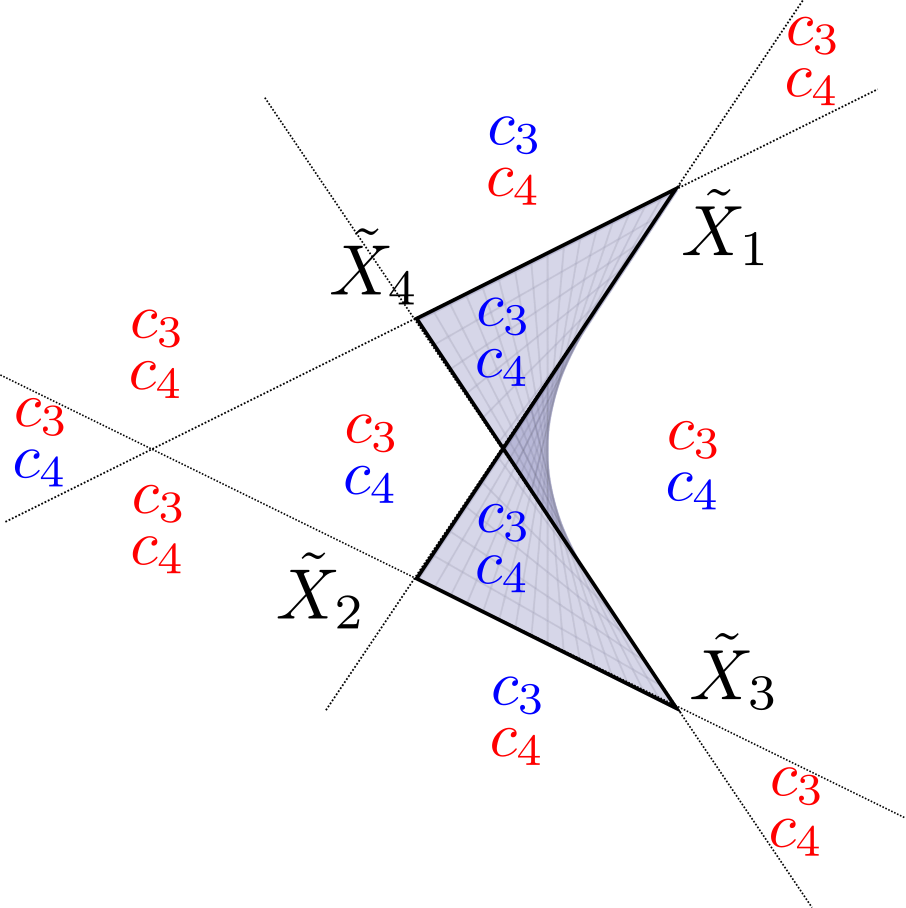}
		\caption{$\chi_1$-diagonal}
	\end{subfigure}
	\caption{The conditions of \propref{crossed}, for the respective cases, hold only when the origin is located inside the homeomorphic region of the crossed projection. When the origin is located in any other region, at least one of these conditions is violated, and the corresponding violated conditions are indicated by red.}
	\figlab{cross-proof}
\end{figure}

\begin{proof}
	The proof is similar to the proof of \propref{propdef} (see \figref{cross-proof}).
\end{proof}

A pair of sliding vector fields exists on $\Lambda$ if and only if the origin of the $yz-$plane is contained in $\tilde{\mathcal{S}}_n$, and this area is bounded by $\mathcal{L}_p$. 
As seen in \figref{cross-proof} (and as follows from simple geometry), $\mathcal{L}_p$ is formed between the endpoints:
\begin{gather*}
\tilde{X}_1 \tl{ and } \tilde{X}_3, \tl{ if } \nl \tilde{X}_1-\tilde{X}_3\nr > \nl \tilde{X}_4-\tilde{X}_2\nr,
\end{gather*}
or:
\begin{gather*}
\tilde{X}_4 \tl{ and } \tilde{X}_2, \tl{ if } \nl \tilde{X}_1-\tilde{X}_3\nr < \nl \tilde{X}_4-\tilde{X}_2\nr,
\end{gather*}
where $\nl \cdot\nr$ denotes the Euclidean norm. If ${\nl \tilde{X}_1-\tilde{X}_3\nr = \nl \tilde{X}_4-\tilde{X}_2\nr}$, then $\tilde{\mathcal{S}}_n=\emptyset,~\mathcal{L}_p=\emptyset$. 

We will present the criteria for the existence of a pair of sliding vector fields for the case where $\chi_1$ is an edge, and the case where $\chi_1$ is a diagonal can be studied similarly.

\begin{proposition}
	Assume that $\tilde{\mathcal{S}}$ is a crossed projection, according to \defnref{DistCas}, with $\chi_1$ being an edge, and define $\kappa$ as:
	\begin{align*}
	\kappa = \begin{cases} 
	1, \qquad  \tl{if} \quad 
	{|| \tilde{X}_1-\tilde{X}_3|| > || \tilde{X}_4-\tilde{X}_2||}, \\
	2, \qquad  \tl{if} \quad 
	{|| \tilde{X}_1-\tilde{X}_3|| < || \tilde{X}_4-\tilde{X}_2||},
	\end{cases}
	\end{align*}
	A pair of sliding vector fields is defined on the codimension-2 discontinuity $\Lambda$ of the PWS system \eqref{X14} if and only if:
	\begin{gather*}
	\tl{Condition 1: } \det\lp \tilde{X}_\kappa~\tilde{X}_{\kappa+2}\rp \det \lp  \tilde{X}_2~\tilde{X}_3 \rp<0,\\
	\tl{Condition 2: }\det\lp \tilde{X}_\kappa~\tilde{X}_{\kappa+2}\rp \det \lp  \tilde{X}_4~\tilde{X}_1 \rp>0,\\
	\tl{Condition 3: }\lp \det\lp \tilde{X}_4~ \tilde{X}_2\rp +\det\lp \tilde{X}_1~ \tilde{X}_3\rp\rp^2-4 \det\lp \tilde{X}_1~ \tilde{X}_2\rp\det\lp \tilde{X}_4~ \tilde{X}_3\rp  >0.
	\end{gather*}
	\proplab{cross-double}
\end{proposition}
	\begin{figure}[H]
	~
	\centering
	\begin{subfigure}[b]{0.4\textwidth}
		\centering
		\includegraphics[scale = 0.19]{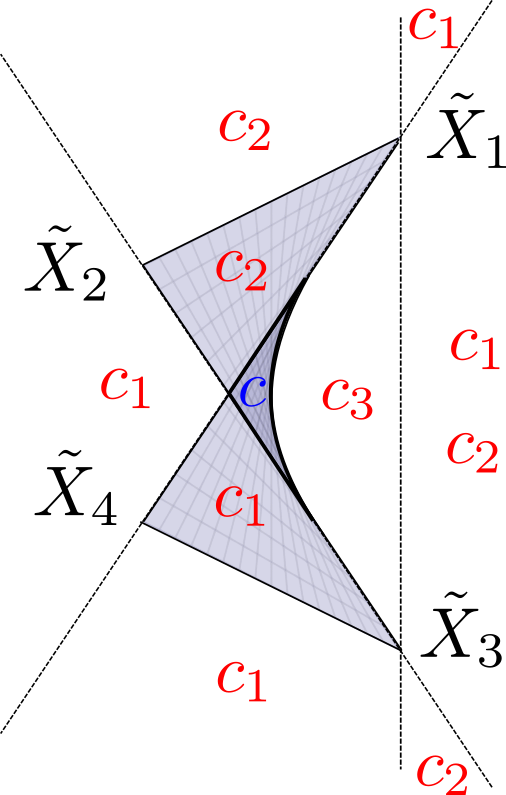}
		\caption{$\mathcal{L}_p$ between $\tilde{X}_1$ and $\tilde{X}_3$}
	\end{subfigure}
	~
	\begin{subfigure}[b]{0.5\textwidth}
		\centering
		\includegraphics[scale = 0.19]{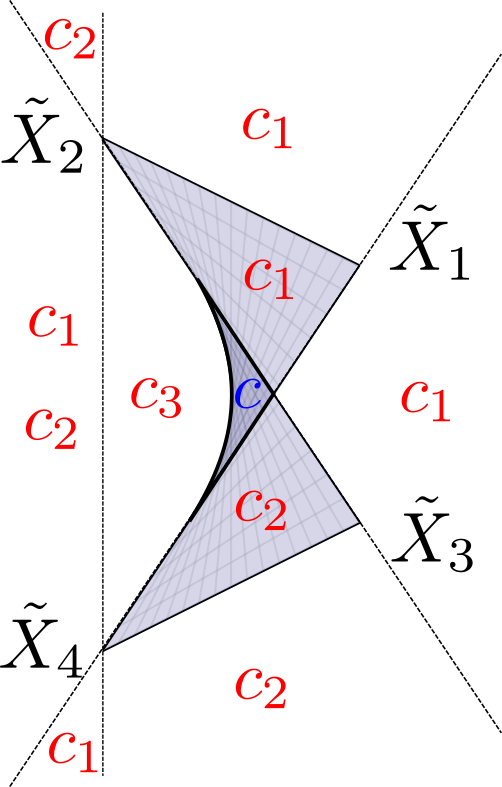}
		\caption{$\mathcal{L}_p$ between $\tilde{X}_2$ and $\tilde{X}_4$ }
	\end{subfigure}
	
	\caption{Conditions 1-3 of \propref{cross-double} hold only in the case for which the origin is located inside the nonhomeomorphic region of the crossed projetion. When the origin is located in any other region, at least one of the conditions is violated, and the corresponding violated conditions are indicated by red colour.}
	\figlab{cros-double-proof}
\end{figure}

\begin{proof}
	The proof is similar to the proof of \propref{propdef} (see \figref{cros-double-proof}). Condition 3 is obtained by requiring $\Delta>0$ in \propref{sigmas} (see \eqref{delta}). This is a sufficient and necessary condition for two real solutions $\lp \sigma_\psi^{(\pm)}, \sigma_\phi^{(\pm)}\rp$ given by \eqref{sigmas} to exist, and it therefore guarantees that the origin of the $yz-$plane lies on the same side as $\tilde{ \mathcal S}_n$ with respect to $\mathcal{L}_p$; if the origin were lying on $\mathcal{L}_p$, then $\Delta=0$ would hold.  
\end{proof}

The case where $\chi_1$ is a diagonal is described by interchanging the indices 2 and 4 {in Conditions 1-2} of \propref{cross-double}.

\subsection{The Concave Cases}
Four possible concave cases are illustrated in Figure \figref{concaves}, where each case is characterized by the vertex corresponding to the endpoint of $\tilde{X}_1$ (the other concave cases are related to the illustrated ones by reflection). We will demonstrate the results for the case shown in \figref{concaves} (a) and results for the other cases can be derived similarly.

\begin{figure}[ht]
	\centering
	\begin{subfigure}[b]{0.23\textwidth}
		\centering
		\includegraphics[scale=0.16]{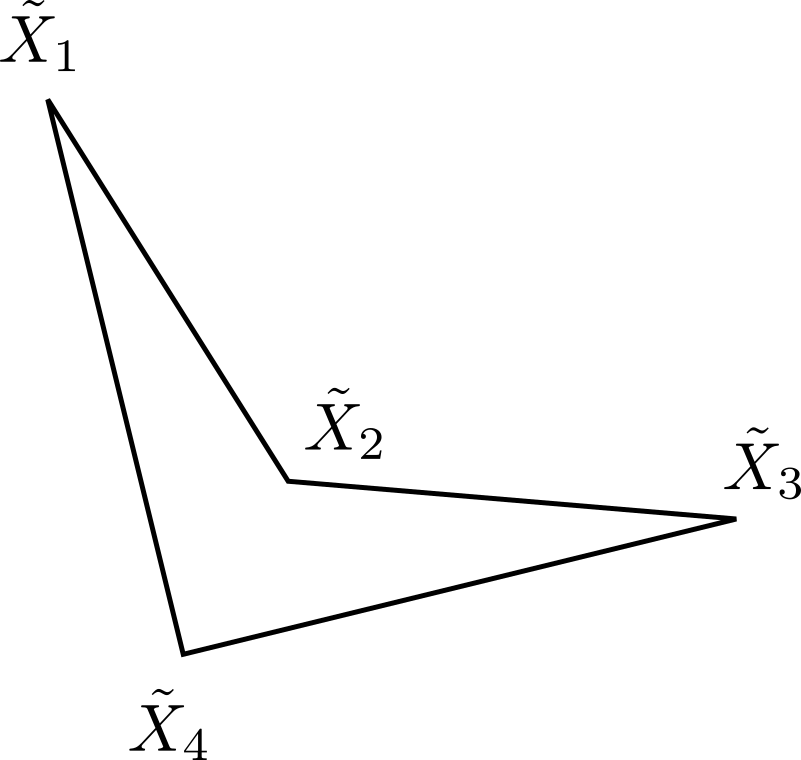}
		\caption{$\tilde{X}_1$-exterior}
	\end{subfigure}
	~
	\begin{subfigure}[b]{0.23\textwidth}
		\centering
		
		\includegraphics[scale=0.16]{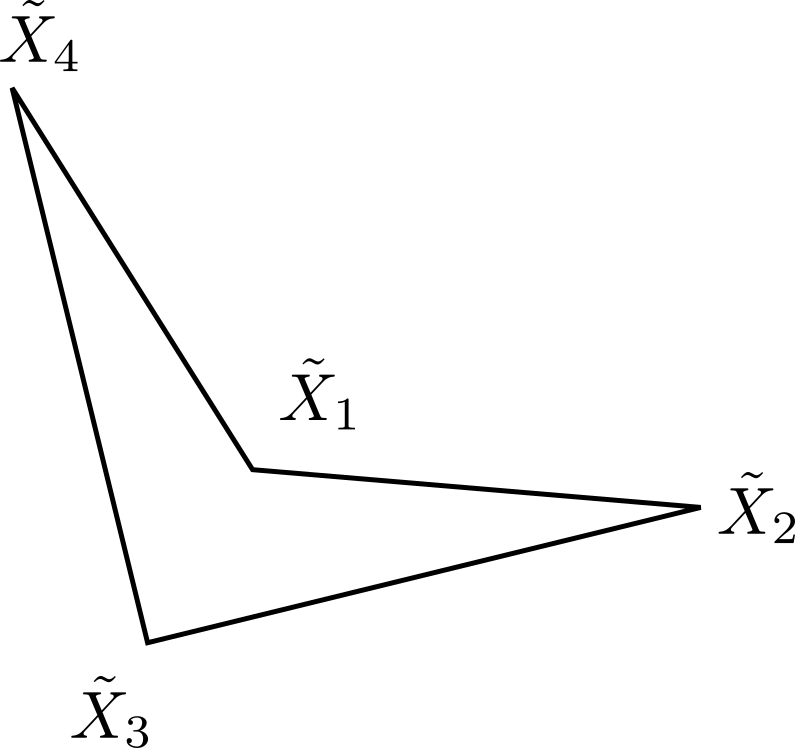}
		\caption{$\tilde{X}_1$-interior}
	\end{subfigure}
	~
	\begin{subfigure}[b]{0.23\textwidth}
		\centering
		
		\includegraphics[scale=0.16]{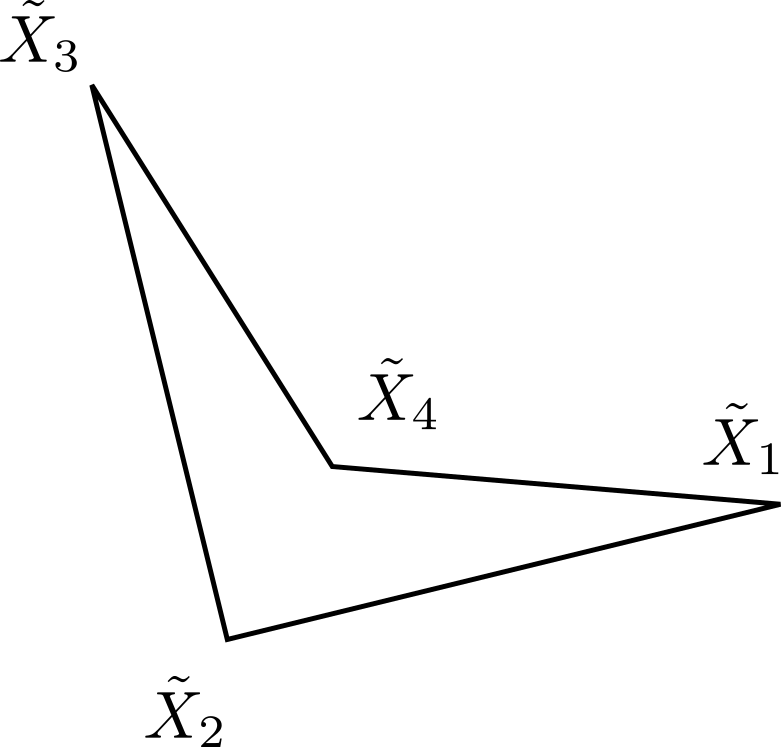}
		\caption{$\tilde{X}_1$-exterior}
	\end{subfigure}
	~
	\begin{subfigure}[b]{0.23\textwidth}
		
		\includegraphics[scale=0.16]{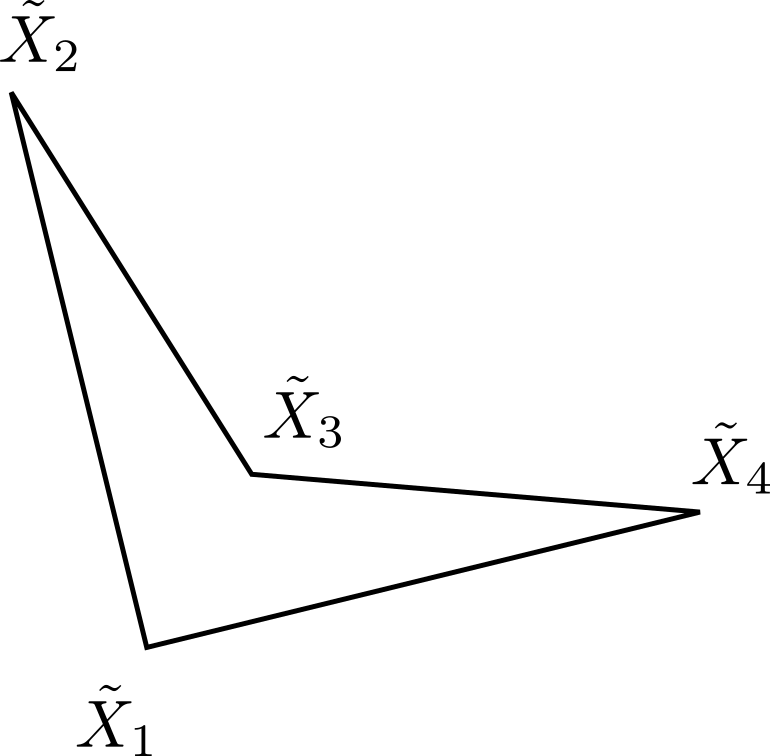}
		\caption{$\tilde{X}_1$-tip}
	\end{subfigure}
	\caption{Four possible ways to ``distribute'' the endpoints of the vector fields $\tilde{X}_{1-4}$ to the corners of a concave quadrilateral. All other concave cases are related to the illustrated ones by reflection.}
	\figlab{concaves}
\end{figure}

A unique sliding vector field exists if the origin is contained in the homeomorphic region of the concave projection. In order to investigate this, we need to divide the homeomorphic region to a convex and a crossed subregion, as shown in 	\figref{conv-double-proof} (a), in each of which a different set of conditions applies. 
\begin{proposition}
	Assume that $\tilde{\mathcal{S}}$ is a concave projection, according to \defnref{DistCas}, where the difference determinant $\delta_1$ is of different sign than $\delta_{2-4}$ (\figref{concaves} (a)). If:  
	\begin{gather*}
	\tl{(crossed subregion) } 
	\begin{cases}
	\tl{Condition 1: } \det\lp \tilde{X}_1~\tilde{X}_2\rp \det \lp  \tilde{X}_2~\tilde{X}_3 \rp< 0,\\
	\tl{Condition 2: }  \det\lp \tilde{X}_3~\tilde{X}_4\rp \det \lp  \tilde{X}_4~\tilde{X}_1 \rp >0,
	\end{cases}
	\end{gather*}
or:
	\begin{gather*}
	\tl{(convex subregion) } 
	\begin{cases}
	\tl{Condition 3: }  \det \lp  \tilde{X}_2~\tilde{X}_3 \rp \det\lp \tilde{X}_3~\tilde{X}_4\rp>0,\\
	\tl{Condition 4: }\det\lp \tilde{X}_4~\tilde{X}_1\rp \det \lp  \tilde{X}_1~\tilde{X}_2 \rp>0,
	\end{cases}
	\end{gather*}
	then a unique sliding vector field is defined on the codimension-2 discontinuity $\Lambda$ of the PWS system \eqref{X14}.
	\proplab{propconc}
\end{proposition} 
\begin{proof}
	The proof is similar to the proof of  \propref{propdef} (see  \figref{conv-double-proof} (a)).
\end{proof}

\begin{figure}[ht]
	~
	\centering
	\begin{subfigure}[b]{0.4\textwidth}
		\centering
		\includegraphics[scale = 0.21]{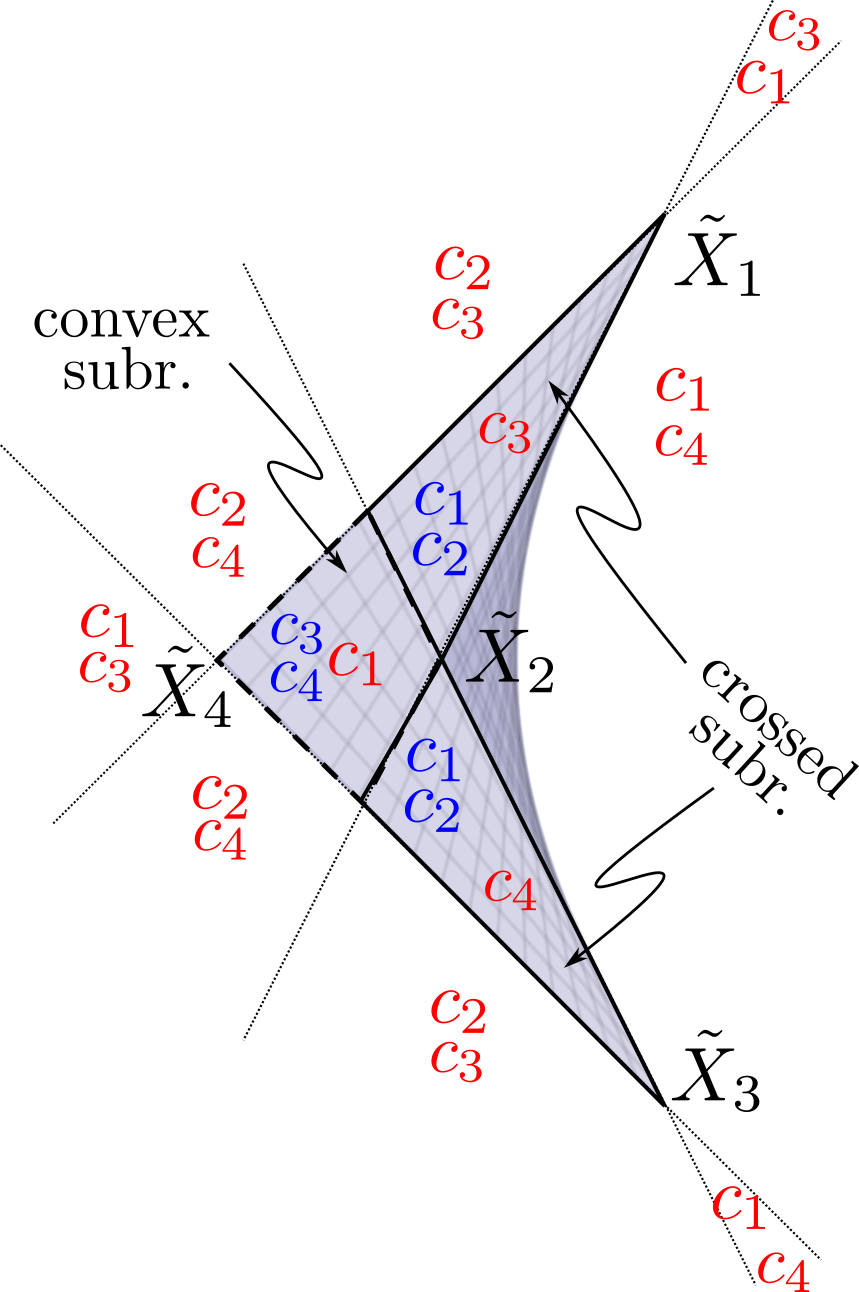}
		\caption{homeomorphic region}
	\end{subfigure}
	~
	\begin{subfigure}[b]{0.5\textwidth}
		\centering
		\includegraphics[scale = 0.21]{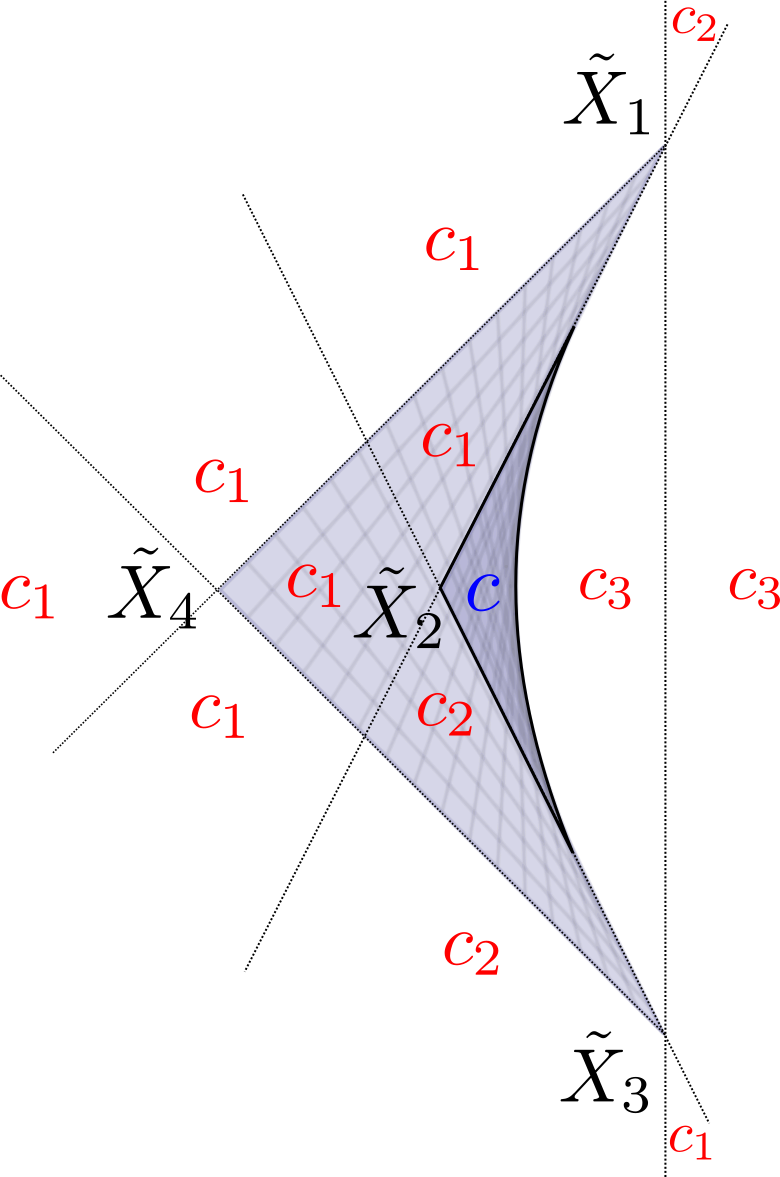}
		\caption{non-homeomorphic region}
	\end{subfigure}
	\caption{The conditions of \propref{propconc} correspond to case illustrated in (a), where the convex and crossed subregions of the concave projections are illustrated. The conditions of \propref{conv-double} correspond to the case illustrated in (b).}
	\figlab{conv-double-proof}
\end{figure}
A pair of sliding vector fields exists if the origin is contained in the nonhomeomorphic region of the concave projection. 
\begin{proposition} 
	Assume that $\tilde{\mathcal{S}}$ is a concave projection, according to \defnref{DistCas}, where the difference determinant $\delta_1$ is of different sign than $\delta_{2-4}$ (\figref{concaves} (a)). A pair of sliding vector fields is defined on the codimension-2 discontinuity $\Lambda$ of the PWS system \eqref{X14} if and only if: 
	\begin{gather*}
	\tl{Condition 1: } \det\lp \tilde{X}_1~\tilde{X}_2\rp \det \lp  \tilde{X}_{1}~\tilde{X}_{3} \rp<0,\\
	\tl{Condition 2: }\det\lp \tilde{X}_2~\tilde{X}_3\rp\det \lp  \tilde{X}_{1}~\tilde{X}_{3} \rp<0,\\
	\tl{Condition 3: }\lp \det\lp \tilde{X}_4~ \tilde{X}_2\rp +\det\lp \tilde{X}_1~ \tilde{X}_3\rp\rp^2-4 \det\lp \tilde{X}_1~ \tilde{X}_2\rp\det\lp \tilde{X}_4~ \tilde{X}_3\rp  >0.
	\end{gather*}
	\proplab{conv-double}
\end{proposition}

\begin{proof}
	The proof is similar to the proof of  	\propref{cross-double} (see \figref{conv-double-proof} (b)).
\end{proof}

\tabref{switches} demonstrates how the indices in \propref{propconc} and  \propref{conv-double} should be modified in order to obtain criteria for the existence of sliding vector fields in case $\tilde{\mathcal{S}}$ corresponds to a convex projection illustrated in \figref{concaves} (b), (c) or (d).

\begin{table}[ht]
	\begin{center}
		\caption{Modification of the indices in Conditions 1-{4} of \propref{propconc} and in Conditions 1-{2} of \propref{conv-double} in order to describe the cases (b), (c) and (d) of  \figref{concaves}.}
		\begin{tabular}{|c|c|}
			\hline
			\textbf{\figref{concaves} (b)} & $1\leftarrow4$, $2\leftarrow1$, $3\leftarrow2$, $4\leftarrow3$\\
			\hline
			\textbf{\figref{concaves} (c)} & $1\leftrightarrow3$, $2\leftrightarrow4$\\
			\hline
			\textbf{\figref{concaves} (d)} & $1\leftarrow2$, $2\leftarrow3$, $3\leftarrow4$, $4\leftarrow1$\\
			\hline
			
			\hline
		\end{tabular}
		\vskip 10pt
		\tablab{switches}
	\end{center}
\end{table}

\section{Stability of the sliding flow}
\seclab{stabsec}
According to 	\defnref{slide2}, the stability of the sliding flow is determined by the determinant and the trace of the Jacobian matrix $\textbf{J}$ given by \eqref{Jac}.
We have that:
\begin{align*}
\det\lp \textbf{J}\rp = \det\lp \textbf{D}{\tilde{F}_x}\rp\det\lp \textbf{P}\rp,
\end{align*}
and since $\det\lp\textbf{P}\rp >0$ we have:
\begin{align}
\tl{sgn}\lp\det\lp \textbf{J}\rp\rp = \tl{sgn}\lp \det\lp \textbf{D}{\tilde{F}_x}\rp\rp. \eqlab{signs}
\end{align} 
Therefore, if $\det\lp \textbf{D}{\tilde{F}_x}\rp<0$ then the stability is of saddle type and if $\det\lp \textbf{D}{\tilde{F}_x}\rp>0$ then the stability is of focus/node/center type. In the latter case, whether the sliding flow is attracting, repelling or of center type is determined by the sign of $\tl{tr}\lp\textbf{J}\rp$.

The parametrization $\tilde{F}_x$ maps the open square $\lp -1,1\rp^2$ of the $\psi\phi$-plane to $\tilde{\mathcal{S}}$ in the $yz-$plane, as shown in \figref{map}. For the Jacobian matrix $\textbf{D}{\tilde{F}_x}$ of the parametrization $\tilde{F}_{x}$ we have:
\begin{align}
 \textbf{D}{\tilde{F}_x}\lp \psi_*,\phi_*\rp = \lp \partial_\psi\tilde{F}_x\lp \psi_*,\phi_*\rp\quad \partial_\phi\tilde{F}_x\lp \psi_*,\phi_*\rp\rp,
 \eqlab{diffEquiv}
\end{align}
where the two tangent vectors $\partial_\psi\tilde{F}_x$ and  $\partial_\phi\tilde{F}_x$ are given by:
\begin{align*}
\partial_\psi\tilde{F}_x = \frac{1}{4}\begin{pmatrix}
\lp \beta_1-\beta_2\rp\lp 1+\phi\rp + \lp \beta_4-\beta_3\rp\lp 1-\phi\rp\\
\lp \gamma_1-\gamma_2\rp\lp 1+\phi\rp + \lp \gamma_4-\gamma_3\rp\lp 1-\phi\rp
\end{pmatrix}, \quad
\partial_\phi\tilde{F}_x = \frac{1}{4}\begin{pmatrix}
\lp \beta_1-\beta_4\rp\lp 1+\psi\rp + \lp \beta_2-\beta_3\rp\lp 1-\psi\rp\\
\lp \gamma_1-\gamma_4\rp\lp 1+\psi\rp + \lp \gamma_2-\gamma_3\rp\lp 1-\psi\rp
\end{pmatrix}.
\end{align*}

\begin{figure}[h!]
	\centering
	\includegraphics[scale=0.4]{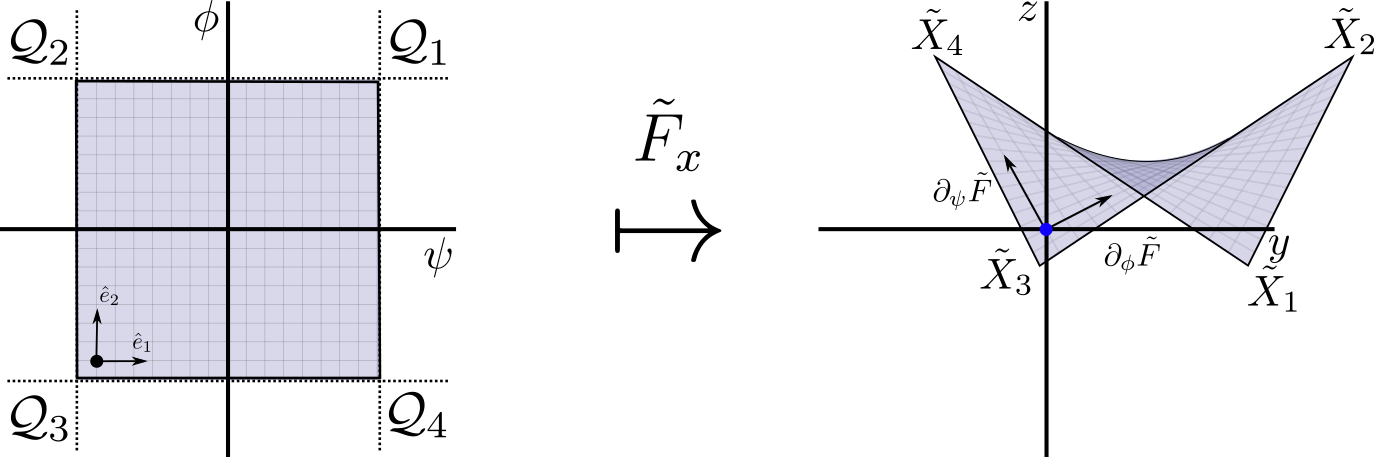}
	\caption{The area $\tilde{\mathcal{S}}$ is the image of the unit square in the $\psi\phi-$plane under $\tilde{F}_x$. For the case of a crossed projection, the map $\tilde{F}_x$ stretches the unit square and folds it back to the $\psi\phi-$plane, and the unit normal vectors are mapped to the tangent vectors of $\tilde{\mathcal{S}}$. The Jacobian matrix of $\tilde{F}_x$ is formed by these two vectors as columns.}
	\figlab{map}
\end{figure}

We are therefore able to relate the Jacobian matrix of the fast subsystem with the matrix formed by the two tangent vectors as columns. If the origin of the $yz-$plane is contained in a subregion of $\tilde{\mathcal{S}}$ where the orientation was preserved under $\tilde{F}_x$ (i.e. $\det\lp \textbf{D}{\tilde{F}_x}\rp>0$), then the stability of the sliding is of node, focus or center type. On the other hand, if the origin of the $yz-$plane is contained in a subregion of $\tilde{\mathcal{S}}$ where the orientation was reversed under $\tilde{F}_x$ (i.e. if $\det\lp \textbf{D}{\tilde{F}_x}\rp<0$) then the stability of the sliding flow is of saddle type. We are able to distinguish between these two cases by only looking at the projected smooth vector fields $\tilde{X}_i$.

\begin{proposition}
	Assume that the origin of the $yz$-plane is contained in a subregion of $\tilde{S}_h$. Then there exists $k\in \{1,2,3,4\}$ such that $\tilde X_k$ and $\tilde X_{k+1}$ are two consequent vectors whose endpoints are corners of this subregion and of $\mathcal{S}$. Furthermore, if:
	\begin{eqnarray*}
	\det\lp \tilde{X}_{k} ~ \tilde{X}_{k+1}\rp > 0,
	\end{eqnarray*}
	then the stability of the sliding vector field is of node, focus or center type.
	On the other hand, if:
	\begin{eqnarray*}
	\det\lp \tilde{X}_{k} ~ \tilde{X}_{k+1}\rp < 0,
	\end{eqnarray*}
	then the stability of the sliding vector field is of saddle type.
	\proplab{orient}
\end{proposition}
\begin{proof}
	For $\tilde{\mathcal{S}}_h\neq\emptyset$, there is at least one edge that is boundary of $\tilde{\mathcal{S}}_h$ and of $\tilde{\mathcal{S}}$ and that is formed by connecting the endpoints of two subsequent projections $\tilde{X}_k$, $\tilde{X}_{k+1}$. Therefore there always exists at least one $k\in\lb 1,2,3,4\rb$ such that the endpoints of $\tilde{X}_k$ and of $\tilde{X}_{k+1}$ are corners of both $\tilde{\mathcal{S}}_h$ and of $\tilde{\mathcal{S}}$.
	
The homeomorphic region is foliated by straight lines, which are given by fixing one of the two parameters $\psi$ or $\phi$ in $\tilde{F}_x\lp \psi, \phi\rp$ and varying the other. The tangent vector $\partial_\psi\tilde{F}$ is directed towards the edge connecting the endpoints of $\tilde{X}_1$ and $\tilde{X}_4$, and the tangent vector $\partial_\phi\tilde{F}$ is directed towards the edge connecting the endpoints of $\tilde{X}_1$ and $\tilde{X}_2$, as illustrated in \figref{proj1}. It follows that in the cases where $\det\lp \tilde{X}_{k} ~ \tilde{X}_{k+1}\rp > 0$  we have that $\det\lp \partial_\psi\tilde{F}_x\lp \psi_*,\phi_*\rp\quad \partial_\phi\tilde{F}_x\lp \psi_*,\phi_*\rp\rp >0$ and therefore from 
\eqref{diffEquiv} and \defnref{slide2} we conclude that the stability of the sliding flow is of node, focus or center type. On the other hand, in the cases where $\det\lp \tilde{X}_{k} ~ \tilde{X}_{k+1}\rp < 0$  we have that $\det\lp \partial_\psi\tilde{F}_x\lp \psi_*,\phi_*\rp\quad \partial_\phi\tilde{F}_x\lp \psi_*,\phi_*\rp\rp <0$ and therefore from 
\eqref{diffEquiv} and \defnref{slide2} we conclude that the stability of the sliding flow is of saddle type. The above hold for all shapes of $\tilde{\mathcal{S}}$ (see e.g. \figref{cr-un}).
\end{proof}

\begin{figure}[]
	\begin{subfigure}[b]{0.23\textwidth}
		\centering
		\includegraphics[scale=0.21]{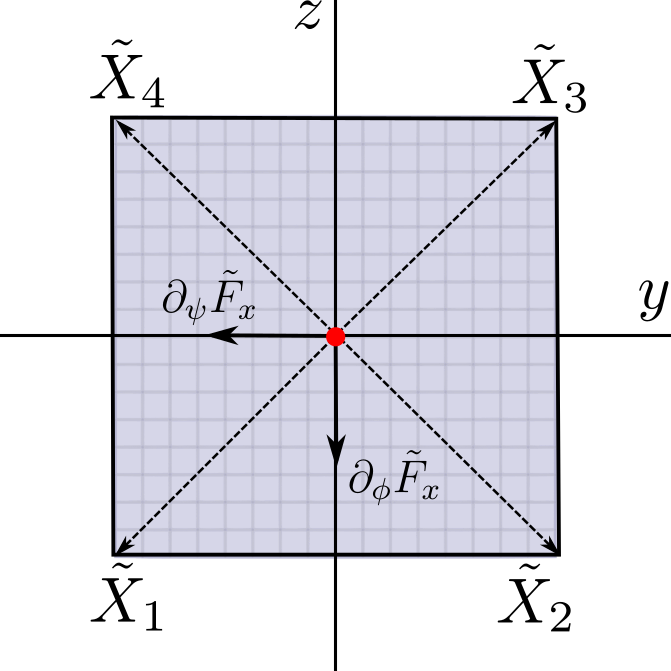}
	\end{subfigure}
	~
	\begin{subfigure}[b]{0.23\textwidth}
		
		\includegraphics[scale=0.21]{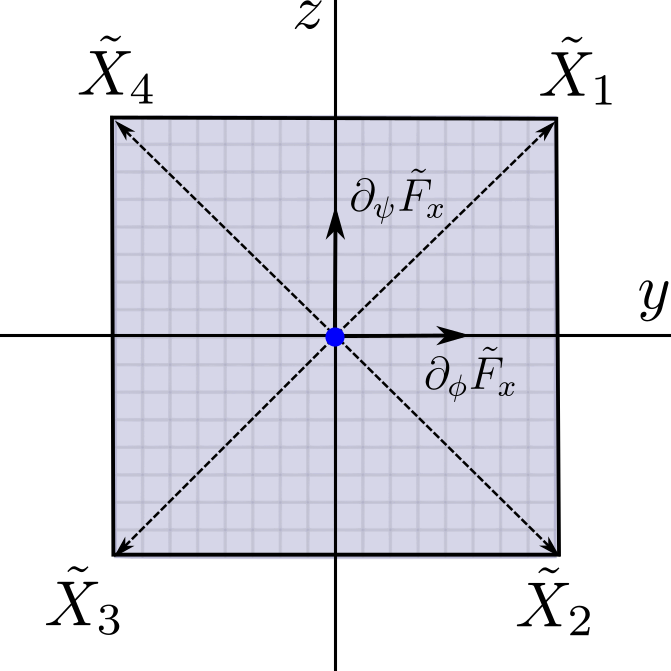}
	\end{subfigure}
	~
	\begin{subfigure}[b]{0.23\textwidth}
		\centering
		
		\includegraphics[scale=0.21]{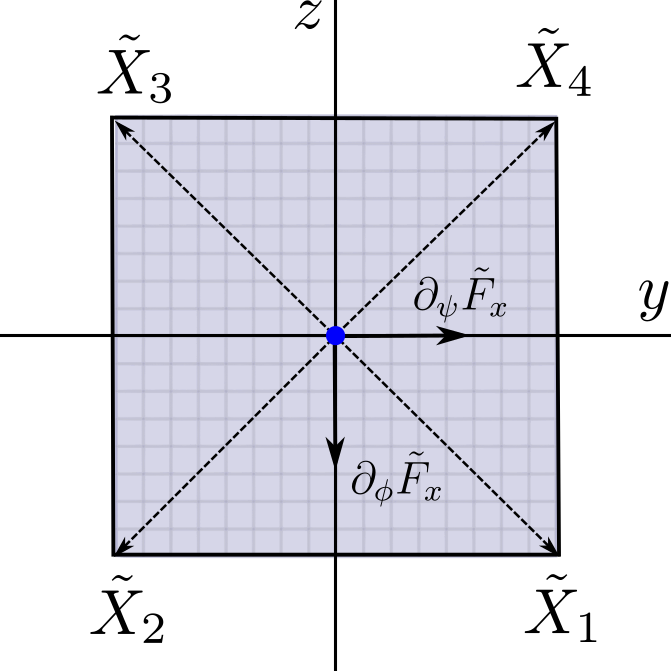}
	\end{subfigure}
	~
	\begin{subfigure}[b]{0.23\textwidth}
		
		\includegraphics[scale=0.21]{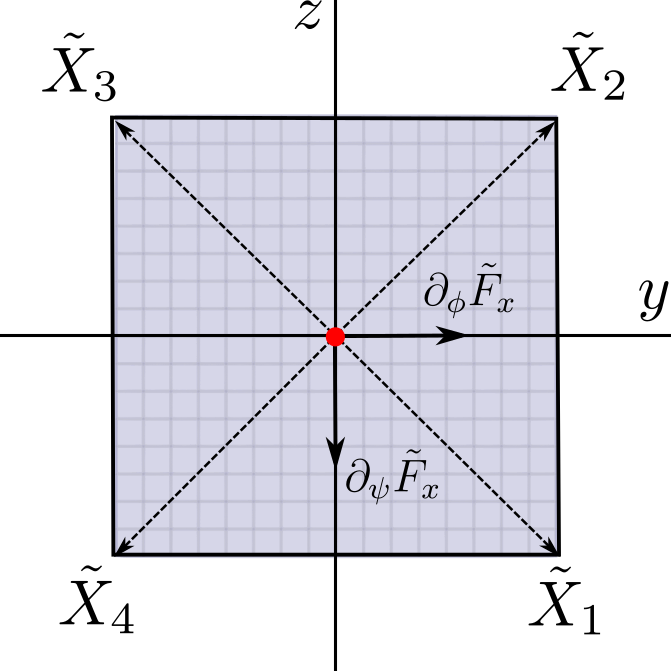}
	\end{subfigure}
	\\
	\vskip1pt
	\begin{subfigure}[b]{0.23\textwidth}
		\centering
		\includegraphics[scale=0.21]{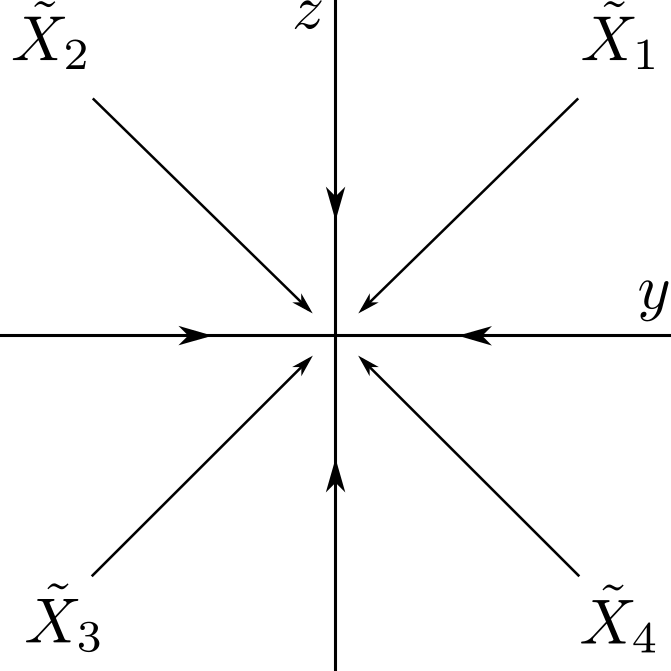}
				\caption{node/focus}
	\end{subfigure}
	~
	\begin{subfigure}[b]{0.23\textwidth}
		
		\includegraphics[scale=0.21]{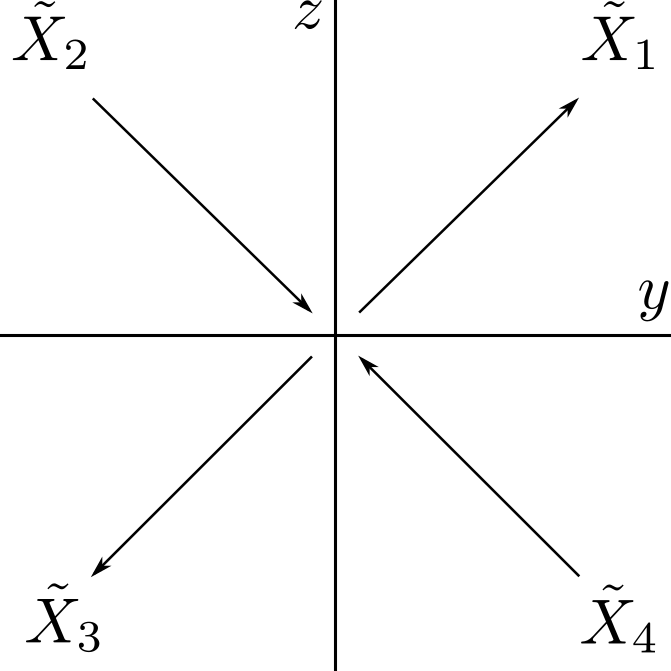}
				\caption{saddle}
	\end{subfigure}
~
	\begin{subfigure}[b]{0.23\textwidth}
		\centering
		
		\includegraphics[scale=0.21]{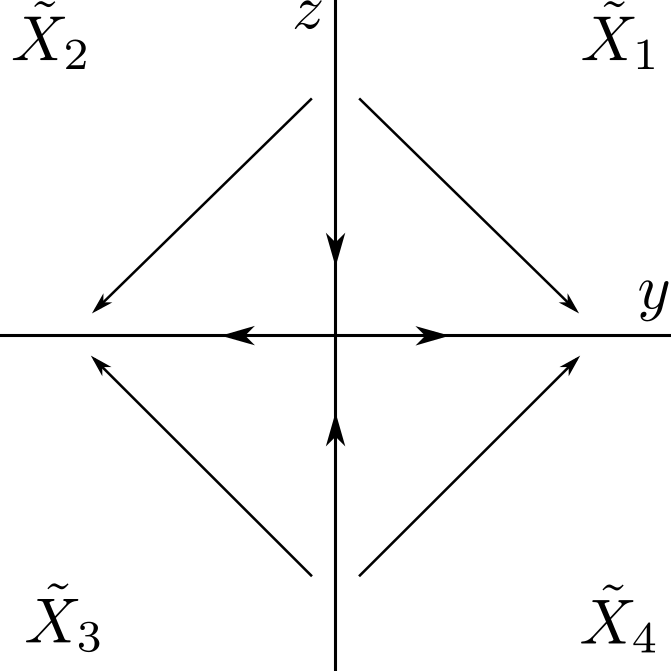}
				\caption{saddle}
	\end{subfigure}
	~
	\begin{subfigure}[b]{0.23\textwidth}
		
		\includegraphics[scale=0.21]{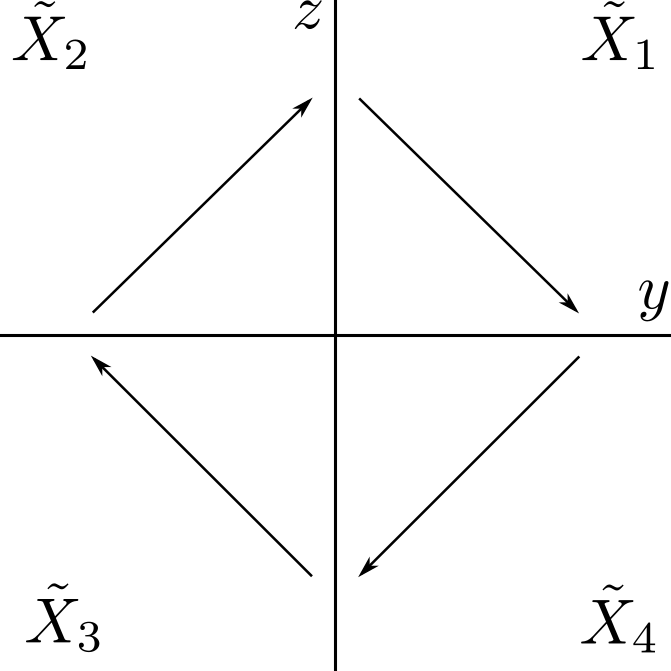}
		\caption{center/focus}
	\end{subfigure}
	\caption{The four convex projection cases. Case (a) corresponds to node/focus-type sliding, cases (b) and (c) correspond to saddle type sliding, case corresponds to center/focus-type sliding. The figure also illustrates the sliding flow on $\Pi_i$, in case it exists.}
	\figlab{proj1}
\end{figure}

An important observation from this result is that whether the sliding flow is of node/focus/center or saddle type does not depend on the choice of the regularization functions. 
In the following, we will use \propref{orient} in order to study the stability in $\tilde{\mathcal{S}}_h$ and $\tilde{\mathcal{S}}_n$ for all three possible projections. As a consequence of this analysis, it will follow that
\begin{theorem}\thmlab{uniqueStable}
 There exists at most one stable sliding vector-field.
\end{theorem}

{On the other hand, whether a focus/node-type sliding vector field is attracting or repelling does depend upon the choice of regularization function. Indeed, the trace of the Jacobian \eqref{Jac} is given by:
\begin{align*}
\begin{aligned}
\tl{tr}\lp \textbf{J}\rp &= \frac{\psi'_*}{4}\lp \lp \beta_1-\beta_2\rp\lp1+\phi_*\rp+\lp \beta_4-\beta_3\rp \lp 1-\phi_*\rp\rp \\
&\quad+\frac{\phi'_*}{4}\lp \lp \gamma_1-\gamma_4\rp\lp1+\psi_*\rp+\lp \gamma_2-\gamma_3\rp \lp 1-\psi_*\rp\rp,
\end{aligned}
\end{align*}
or:
\begin{align}
\begin{aligned} 
\tl{tr}\lp \textbf{J}\rp& = \frac{\psi'_*}{2}\lp \lp \beta_1-\beta_2\rp\sigma_\phi+\lp \beta_4-\beta_3\rp \lp 1-\sigma_\phi\rp\rp \\
&\quad+\frac{\phi'_*}{2}\lp \lp \gamma_1-\gamma_4\rp\sigma_\psi+\lp \gamma_2-\gamma_3\rp \lp 1-\sigma_\psi\rp\rp,
\end{aligned}
\eqlab{trace}
\end{align}
in terms of $(\sigma_\psi,\sigma_\phi)$. Now, although the values $(\psi_*,\phi_*)$, and therefore also $(\sigma_\psi,\sigma_\phi)$, are independent of the regularization functions, the quantities $\psi'_*$ and $\phi'_*$ do in fact depend upon $\psi$ and $\psi$. The sign of $\tl{tr}$ could therefore vary (at least when the differences $\lp\beta_i-\beta_j\rp$, $\lp\gamma_i-\gamma_j\rp$ are not all of the same sign) for different choices of regularization functions. See also \protect\cite{krihog2} where similar issues occur in the regularization of the two-fold. However, there are special cases where the nodal/focal-stability is independent of the choice of regularization functions. 
	\begin{proposition}\proplab{stabilityExample}
		Assume that the PWS system \eqref{X14} admits a sliding vector field on $\Lambda$ in the sense of \defnref{slide2Fil} which is of node/focus type (i.e. $\text{det}~F_x(\psi_*,\phi_*)>0$) according to \propref{orient}, and assume further that $s_1$, $s_2$ given by:
\begin{align*}
s_1 &= \lp \beta_1-\beta_2\rp\sigma_\phi+\lp \beta_4-\beta_3\rp \lp 1-\sigma_\phi\rp,\\
s_2 &=\lp \gamma_1-\gamma_4\rp\sigma_\psi+\lp \gamma_2-\gamma_3\rp \lp 1-\sigma_\psi\rp,
\end{align*}
are of the same sign. If $s_{1-2}$ are positive, then the sliding vector field is repelling, while if $s_{1-2}$ are negative, then the sliding vector field is attracting.
	\end{proposition}
\begin{proof}
	We brake \eqref{trace} into the two terms:
	\begin{align*}
\frac{\psi'_*}{2}\lp \lp \beta_1-\beta_2\rp\sigma_\phi+\lp \beta_4-\beta_3\rp \lp 1-\sigma_\phi\rp\rp=\frac{\psi'_*}{2}s_1,
	\end{align*}
	and:
	\begin{align*}
\frac{\phi'_*}{2}\lp \lp \gamma_1-\gamma_4\rp\sigma_\psi+\lp \gamma_2-\gamma_3\rp \lp 1-\sigma_\psi\rp\rp=\frac{\phi'_*}{2}s_2.	
	\end{align*}
	By \defnref{phiFuncs} and the fact that $\phi'_*, \psi'_*>0$ the signs of these two terms coincide with the sign of $s_{1}$ and $s_2$, respectively. The result therefore follows.
\end{proof}
} 

{A more restricted class of PWS systems of the form \eqref{X14} for which the stability of the nodal/focal sliding flow on $\Lambda$ does not depend on the regularization is described in the following corollary.
\begin{corollary}\corlab{stabilityExample2}
Assume that the PWS system \eqref{X14} admits a sliding vector field in the sense of \defnref{slide2Fil} which is of node/focus type according to \propref{orient}, and assume further that the differences:
\begin{align*}
\lp \beta_1-\beta_2\rp, \quad \lp \beta_4-\beta_3 \rp,\quad  \lp \gamma_1-\gamma_4\rp,\quad  \lp \gamma_2-\gamma_3\rp,
\end{align*}
are of the same sign. If the above differences are positive, then the sliding vector field is repelling, while if the above differences are negative, then the sliding vector field is attracting. 
\end{corollary}
\begin{proof}
	In this case $s_{1-2}$ in \propref{stabilityExample} have the same sign. The result is therefore follows from the conclusions in \propref{stabilityExample}.
\end{proof}}

{The references \protect\cite{guglielmi2015a} and \protect\cite{guglielmi2017a} provide examples of systems where the stability of the sliding flow changes with different regularization functions.}
\subsection{Convex projections}
This is the simplest of the three possible projections. Four representative examples of $\tilde{\mathcal{S}}$ are illustrated in \figref{proj1}
\begin{corollary}
	Assume that $\tilde{\mathcal{S}}$ is a convex projection, according to \defnref{DistCas}. If:
	\begin{eqnarray*}
	\det\lp \tilde{X}_{i} ~ \tilde{X}_{i+1}\rp < 0, \quad i = 1,2,3,4,
	\end{eqnarray*}
	then the stability of the sliding vector field is of saddle type.
 On the other hand, if:
	\begin{eqnarray*}
	\det\lp \tilde{X}_{i} ~ \tilde{X}_{i+1}\rp > 0,  \quad i = 1,2,3,4,
	\end{eqnarray*}
	then the stability of the sliding vector field is of node, focus or center type.
\end{corollary}

\begin{proof}
	Follows from \propref{orient}.
\end{proof}

We now investigate the dependence on the choice of the regularization function for the individual cases:
\begin{enumerate}
	\item Saddle type: In case $\det\lp \tilde{X}_{i} ~ \tilde{X}_{i+1}\rp < 0$, the stability of the sliding is of saddle type and does not depend on the choice of the regularization functions $\psi(\cdot)$ and $\phi(\cdot)$; see \figref{proj1} (b) and (c).
	\item Node/focus/center type: In case $\det\lp \tilde{X}_{i} ~ \tilde{X}_{i+1}\rp > 0$, then whether the sliding flow is attracting or repelling depends on the sign of the trace of $\textbf{J}$ {as given by \eqref{trace} and it generally depends on the choice of the regularization functions $\psi(\cdot)$ and $\phi(\cdot)$.} However, in \figref{proj1} (a) the sliding flow is always attracting, as $\tl{tr}\lp \textbf{J}\rp<0$ for any values of $\psi_{\hat{y}}^*$ and $\psi_{\hat{z}}^*$ (since all differences between $\beta_{1-4}$ and $\gamma_{1-4}$ in \eqref{trace} are negative). Recall also \corref{stabilityExample2}. On the other hand, for cases similar to the one illustrated in \figref{proj1} (d), the sign of $\tl{tr}\lp \textbf{J}\rp$ for fixed $\beta_{1-4}$ and $\gamma_{1-4}$ can vary for different values of $\phi_{\hat{y}}^*$ and $\psi_{\hat{z}}^*$, i.e it generally depends on the choice of the regularization functions $\psi(\cdot)$ and $\phi(\cdot)$;  i.e. there exist two separate pairs of regularization functions $\psi_1\lp\cdot\rp,\phi_1\lp\cdot\rp$ and $\psi_2\lp\cdot\rp,\phi_2\lp\cdot\rp$ such that for $\psi_1\lp\cdot\rp,\phi_1\lp\cdot\rp$ the sliding flow is of stable focus type, while for $\psi_2\lp\cdot\rp,\phi_2\lp\cdot\rp$ the sliding flow is of unstable focus type. Similar observations were made in \protect\cite{guglielmi2015a} and \protect\cite{guglielmi2017a}.
\end{enumerate}

\subsection{Crossed projection}
The two disjoint subregions of  $\tilde{\mathcal{S}}_h$ (see \figref{cr-un}) in the crossed projection are characterized by different kinds of stabilities. 
\begin{corollary}\corlab{this}
	Assume that $\tilde{\mathcal{S}}$ is a crossed projection (according to \defnref{DistCas}). Then one of the two disjoint subregions of $\tilde{\mathcal{S}}_h$ corresponds to stability of saddle type and the other corresponds to stability of focus/node/center type.
\end{corollary}
\begin{proof}
	Follows from \propref{orient}, as for in one of the two distinct subregions of $\tilde{\mathcal{S}}_h$ we have $\det\lp \tilde{X}_{k} ~ \tilde{X}_{k+1}\rp < 0$, while  the other distinct subregion of $\tilde{\mathcal{S}}_h$ we have $\det\lp \tilde{X}_{k} ~ \tilde{X}_{k+1}\rp > 0$ (see \figref{cr-un}).
\end{proof}

\begin{figure}[h!]
	\centering
	\centering
	\begin{subfigure}[b]{0.3\textwidth}
		\centering
		\includegraphics[scale = 0.2]{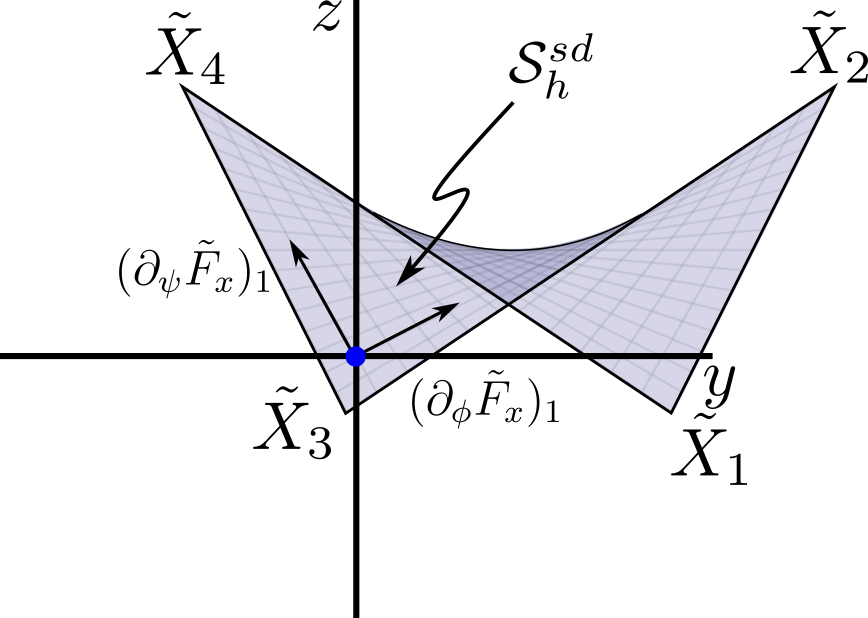}
		\caption{}
	\end{subfigure}
	~
	\begin{subfigure}[b]{0.3\textwidth}
		\centering
		\includegraphics[scale = 0.2]{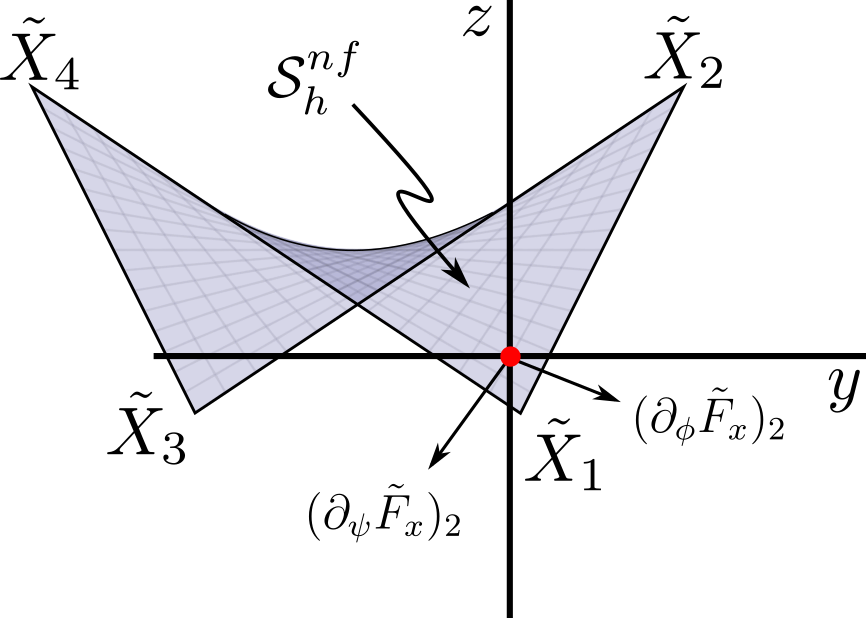}
		\caption{}
	\end{subfigure}
	~
	\begin{subfigure}[b]{0.3\textwidth}
		\centering
		\includegraphics[scale = 0.2]{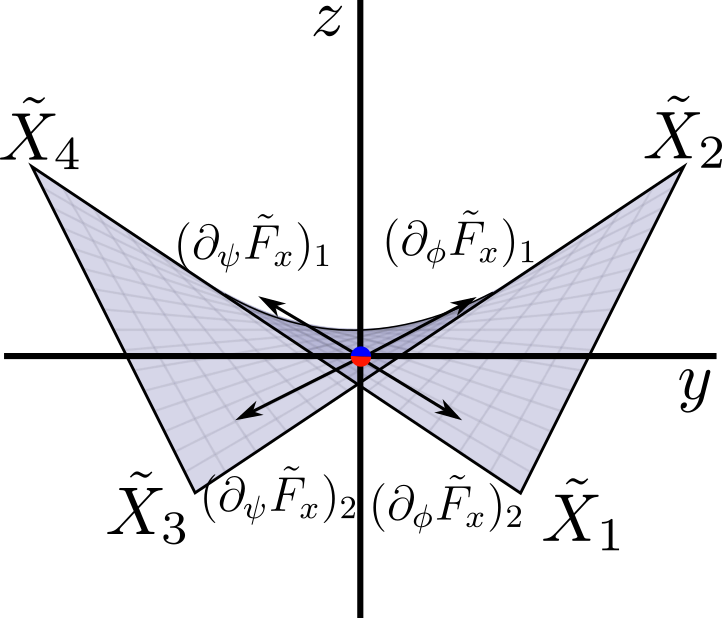}
		\caption{}		
	\end{subfigure}
	\caption{The tangent vectors in each of the distinct subregions of $\tilde{\mathcal{S}}$ for the case of a crossed projection are illustrated. (a) When the origin $(y,z) = (0,0)$ is contained in $\tilde{\mathcal{S}}_h^{sd}$ (see text for definition), then we have $\det\lp \textbf{D}{\tilde{F}_x\lp \psi_*, \phi_*\rp}\rp<0$ and the stability of the sliding flow is of saddle type. (b) When the origin $(y,z) = (0,0)$ is contained in $\tilde{\mathcal{S}}_h^{nf}$, then we have $\det\lp \textbf{D}{\tilde{F}_x\lp \psi_*, \phi_*\rp}\rp>0$ and the stability of the sliding flow is of node/focus/center type. (c) When the origin $(y,z) = (0,0)$ is contained in $\tilde{\mathcal{S}}_n$, then there exist two pairs $\lp \psi_*, \phi_*\rp_1$ and $\lp \psi_*, \phi_*\rp_2$ 
	for which $\tilde{\mathcal{S}}$ intersects with $\Lambda$, and therefore two sliding vector fields exist. In addition, we have $\det\lp \textbf{D}{\tilde{F}_x\lp \psi_*, \phi_*\rp_1}\rp<0$ and $\det\lp\textbf{D}{\tilde{F}_x\lp \psi_*, \phi_*\rp_2}\rp>0$, hence the stability of one sliding vector field is of saddle type, while the stability of the other is of node/focus/center type.}
	\figlab{cr-un}
\end{figure}

In the following, we will use $\tilde{\mathcal{S}}_h^{nf}$ to refer to the subregion of $\tilde{\mathcal{S}}_h$ for which the orientation of the unit box has been preserved under $\tilde{F}_x$. Similarly,  we will use $\tilde{\mathcal{S}}_h^{sd}$ to refer to the subregion of $\tilde{\mathcal{S}}_h$ for which the orientation of the unit box has been reversed under $\tilde{F}_x$. By \corref{this} it therefore follows that if the origin of the $yz-$plane is contained inside $\tilde{\mathcal{S}}_h^{nf}$ ($\tilde{\mathcal{S}}_h^{sd}$), then the stability of the sliding flow on $\Lambda$ is of node/focus/center type (saddle type, respectively). 

Recall that when the origin of the $yz-$plane is contained in $\tilde{\mathcal{S}}_n$, then a pair of sliding vectors exist. The results concerning the stability of these vector fields are general and do not depend on the shape of $\tilde{\mathcal{S}}$.

\begin{proposition}
	Assume that a pair of sliding vector fields exists. Then,  the stability of one of the sliding vector fields is of saddle type and the stability of the other sliding vector field is of node/focus type.
	\proplab{pairstab}
\end{proposition}

\begin{proof}
	Every point on $\tilde{\mathcal{S}_n}$ has two pre-images under $\tilde{F}_x$, and we can view $\tilde{\mathcal{S}_n}$ as two overlapping ``sheets'' with different orientations (see \figref{map} and \figref{cr-un} (c)). Therefore, one sheet corresponds to an area of the unit box where the orientation has been preserved under $\tilde{F}_x$ and one sheet corresponds to an area of the unit box that the orientation has been reversed under $\tilde{F}_x$.
\end{proof}

The important conclusion of \propref{pairstab}, stated as \thmref{uniqueStable}, is that even in cases where two sliding vector fields are defined on $\Lambda$, at most one of them is stable. 

\subsection{Concave cases} 
In the concave cases, we have the following
\begin{corollary}
	Assume that $\tilde{\mathcal{S}}$ is a concave projection (according to  \defnref{DistCas}) with the endpoint of $\tilde{X}_k$ being the tip. If:
	\begin{eqnarray*}
	\det\lp \tilde{X}_k ~ \tilde{X}_{k+1}\rp< 0 \quad \tl{ and }\quad 	\det\lp \tilde{X}_k ~ \tilde{X}_{k-1}\rp< 0,
	\end{eqnarray*}
	then the stability of the sliding vector field is of saddle type. On the other hand, if:
	\begin{eqnarray*}
	\det\lp \tilde{X}_k ~ \tilde{X}_{k+1}\rp> 0 \quad \tl{ and }\quad 	\det\lp \tilde{X}_k ~ \tilde{X}_{k-1}\rp> 0,
	\end{eqnarray*}
	then the stability of the sliding vector field is of node/focus/center type.
\end{corollary}
\begin{proof}
	Follows from \propref{orient}.
\end{proof}


\section{Bifurcations of sliding vector fields using blowup}\seclab{bifurcation}
In the following, we will describe how sliding can appear or disappear along $\Lambda$ and what consequences this has on the dynamics near $\Lambda$. In line with PWS theory, we will call such bifurcations \textit{sliding bifurcations}. To illustrate the findings we will promote the following blowup approach, also used by Peter Szmolyan in \protect\cite{peterTalk} to study a gene regulatory network in $\mathbb R^2$. Consider the following transformation
\begin{align}
 (x,r,(\bar y,\bar z,\bar \varepsilon))\mapsto (x,y,z,\varepsilon) = (x,r \bar y,r\bar z,r\bar \varepsilon),\eqlab{blowup1}
\end{align}
for $r\ge 0$ and $(\bar y,\bar z,\bar \varepsilon)\in S^2 = \{\bar y^2+\bar z^2+\bar \varepsilon^2=1\}$. Clearly this mapping just corresponds to introducing spherical coordinates in the $(y,z,\varepsilon)$-space. Therefore it is one-to-one for $r>0$ but $\{r=0\}$ is mapped onto $\Lambda \times \{0\}$. In this sense, the inverse process of \eqref{blowup1} blows up $\Lambda\times \{0\}$ to a cylinder $\mathcal I\times S^2$. See \figref{LambdaBlowup}.
\begin{figure}[ht] 
	\begin{center}
		\centering
		{\includegraphics[scale=0.1]{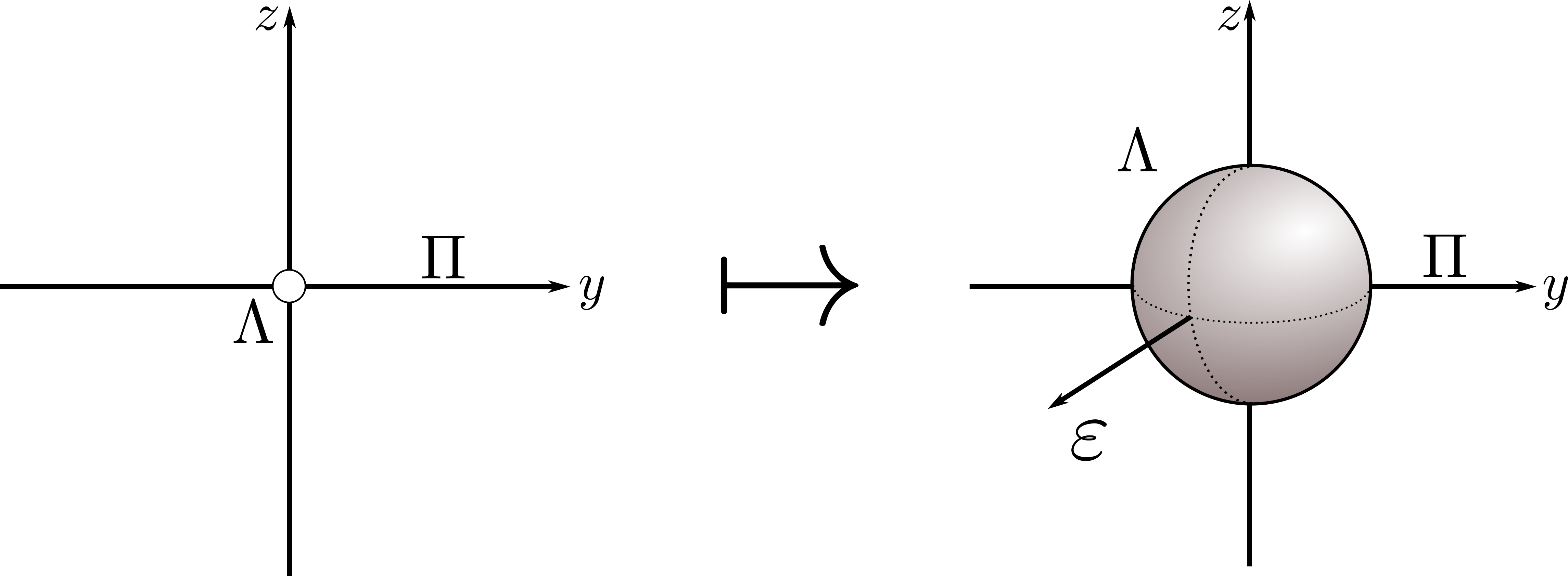}}
	\end{center}
	\caption{ The blowup \eqref{blowup1}.}
	\figlab{LambdaBlowup}
\end{figure}
Notice that by \eqref{varhat} and \eqref{blowup1} we have that
\begin{eqnarray}
\begin{aligned}
 \hat y =& \epsilon^{-1} y = \frac{\bar y}{\bar \epsilon},\\
 \hat z =& \epsilon^{-1} z = \frac{\bar z}{\bar \epsilon},
\end{aligned}
\eqlab{hatyhatz}
\end{eqnarray}
We can therefore think of $(\hat y,\hat z)$ as coordinates for the chart obtained by the central projection from the sphere $(\bar y,\bar z,\bar \varepsilon)\in S^2$ onto the plane $\bar \epsilon=1$.  In this way, \eqref{blowup1} is identical to a Poincare compactification \protect\cite{carmen} of the $(x,\hat y,\hat z)$-system with $r=\epsilon=0$. By using central projections onto the other relevant planes $\bar z=\pm 1$, $\bar y=\pm 1$, we can therefore connect the $(x,\hat y,\hat z)$-dynamics with the PWS system outside $r>0$. 

However, the system on $r\ge 0$, $(\bar y,\bar z,\bar \varepsilon)\in S^2$ is still singular. $\psi(\varepsilon^{-1}y)$, for example, becomes 
$\psi(\bar \varepsilon^{-1} \bar y)$, using \eqref{blowup1} which is not defined along $r\ge 0,(\bar y,\bar z,\bar \varepsilon) = (0, \pm 1,0)$. Consider first $r\ge 0,(\bar y,\bar z,\bar \varepsilon) = (0, 1,0)$. Notice that this set, under \eqref{blowup1}, gets mapped to $\Pi_1$. We therefore blowup $(\bar y,\bar z,\bar \varepsilon)=(0,1,\bar \varepsilon)$ by applying the transformation:
\begin{align}
(\rho,(\bar{\bar y},\bar{\bar \epsilon}))\mapsto (\bar z^{-1}\bar y,\bar z^{-1}\bar \varepsilon) =\rho ( \bar{\bar y},\bar{\bar \varepsilon}),\eqlab{blowup2}
\end{align}
for $\rho \ge 0$ and $(\bar{\bar y},\bar{\bar \varepsilon})\in S^1$. In this way, under the image process of \eqref{blowup2} and \eqref{blowup1} we have blown up $\Pi_1$ to a cylinder $r\ge 0,x\in \mathcal I, (\bar{\bar y},\bar{\bar \varepsilon})\in S^1$. We proceed in a similar way for $(\bar y,\bar z,\bar \varepsilon)=(-1,0,0), (\bar y,\bar z,\bar \varepsilon)=(0,-1,0)$ and $(\bar y,\bar z,\bar \varepsilon)=(1,0,0)$. Notice that these points are mapped to $\Pi_2$, $\Pi_3$ and $\Pi_4$ under \eqref{blowup1} for $r\ge 0$, respectively, and these objects are therefore also under the inverse process blown up to cylinders. This produces the final diagram in \figref{LambdaBlowup2}.

\begin{figure}[h!] 
	\begin{center}
		\centering
		{\includegraphics[scale=0.1]{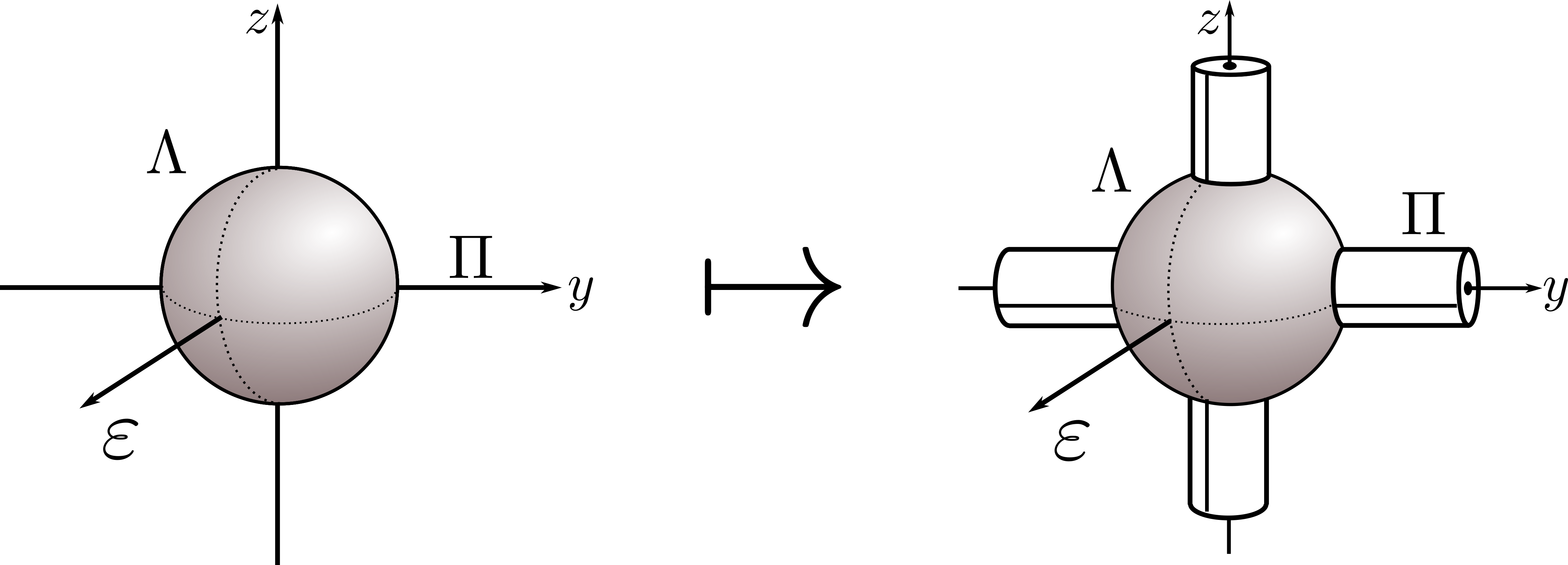}}
	\end{center}
	\caption{The blowup \eqref{blowup2}.}
	\figlab{LambdaBlowup2}
\end{figure}

Now, to illustrate our blowup approach, we will consider the case where $\tilde{\mathcal{S}}$ is a crossed projection for all $x\in \mathcal I$. We suppose  that the location of $\tilde{\mathcal{S}}$ in the $yz-$plane depends upon $x$ in a translational fashion. We consider two different examples in the following sections below. For simplicity, we leave out all the necessary calculations (the interested reader can consult \appref{aa} for a short description and \protect\cite{krupa_extending_2001,KristiansenHogan2018,kriBlowup} for similar computations in other settings) and just present the results in diagrams. 
\subsection{Entering (leaving) $\tilde{\mathcal{S}}_h$ through one of the straight segments}
The first example is seen in \figref{sh-out}. The top row shows the different regions $\tilde{\mathcal S}$ for different $x$: $x<x_b$, $x=x_b$ and $x>x_b$. In (a) where $x<x_b$, for example, $\tilde{\mathcal S}_h$ intersects $(y,z)=(0,0)$, meaning that $\mathcal S$ intersects $\Lambda$ at one single point. In this case, the critical manifold (and hence the sliding) is of saddle type and we illustrate the dynamics of the layer problem, see \eqref{layprob} in the $(\hat y,\hat z)$-coordinates, on the sphere, using the central projection, below $\tilde{\mathcal S}$. Recall that $\dot x=0$ for this layer problem so it is actually $2D$. We only have dynamics on $x$ on the reduced problem on $C_0$ (see \eqref{c0dyn}).  In the four quadrants around the sphere, we illustrate the projections $\tilde X_i$ of the four vector-fields $X_i$, appearing as corners of the region $\tilde{\mathcal S}$. Along the blown up $\Pi_1$-cylinder, emanating from $(\bar y,\bar z,\bar \varepsilon)=(1,0,0)$ on the sphere, we then illustrate the dynamics of the layer problem and the direction (in this $yz$-projection) of the corresponding reduced flow (or equivalently sliding flow),  see \eqref{layer} and \eqref{reduced} in the $(z,\hat y)$-coordinates. 

In (b) we present the same diagram for a different $x$-value $x=x_b$ where we suppose that $\tilde{\mathcal S}$ now intersects the origin in the $yz$-plane along the edge $\chi_3$. Then $\tilde X_4$ and $\tilde X_3$ are anti-parallel vectors and hence the sliding vector-field along $\Pi_4 \cap\{x=x_b\}$ vanishes.\footnote{Notice that this is a consequence of our assumption (A). General nonlinear unfoldings of $X_{3}$ and $X_4$, will produce a locally unique pseudo-equilibrium of the sliding vector-field along $\Pi_4$ at $x=x_b$, $y=z=0$ in \figref{sh-out} (b).} Going from (a) to (b) we see that $C_0$ intersects this line of equilibria at precisely $x=x_b$. In (c) where $x>x_b$, $\tilde{\mathcal S}$ does not intersect $(y,z)=(0,0)$ and therefore sliding along $\Lambda$ has seized to exist. Also, as a consequence, we see that the direction of the sliding vector-field along $\Pi_4$ has changed direction. This example demonstrates how the blowup approach can be used, not only as a computational and dynamical method, but also as an informative, illustrative approach to present the consequences of the sliding bifurcations.  

It is also possible to study the case where $\tilde{\mathcal S}$ crosses $(y,z)=(0,0)$ along the edges $\chi_1$ and $\chi_3$ in such a way that the fast subsystem always has an equilibrium. 
We illustrate the bifurcations in \figref{sn-sh} focussing on the generic cases. In the top row we see that $\tilde{\mathcal S}$ transverse the $yz$-plane in such a way that it always intersects $(y,z)=(0,0)$. \figref{sn-sh} (a) is identical to \figref{sh-out} (a). The details are similar to the case illustrated in \figref{sh-out}, the main difference being that the dynamics on the sphere has two equilibria inside $\tilde{\mathcal S}_n$. 
%
%
%
%
From (a) to (b), $\tilde{\mathcal S}$ has crossed $(y,z)=(0,0)$ along $\chi_1$. As a result a sliding bifurcation occurs along $\Pi_1$ and, in (b), bottom row, a stable node appears. Notice the resulting change of direction of the sliding flow along the corresponding cylinder. Similarly, from (b) to (c), bottom row, the saddle has disappeared due to a collision with the blowup of $\Pi_4$. This collision is due to the sliding bifurcation that occurs along $\Pi_4$ when $\tilde{\mathcal S}$ intersects $(y,z)=(0,0)$ along $\chi_3$. As a result, in agreement with \figref{sh-out}, we see that the sliding flow along $\Pi_4$ changed direction from (b) to (c). We emphasize, that in case (b), we can only prove that the dynamics on the sphere is as illustrated in \figref{sn-sh}(b), when the origin is sufficiently close to $\chi_1$. Further away, Hopf bifurcations could occur, producing global limit cycles that we cannot study by our local methods. The diagram in (c) is therefore also just a potential phase portrait (which we can reproduce numerically for specific values).

\begin{figure}[h!]
	\centering
	\begin{subfigure}[b]{0.3\textwidth}
		\centering
		\includegraphics[scale = 0.45]{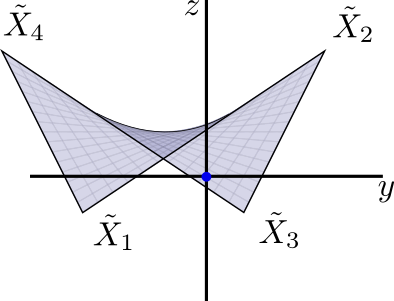}
	\end{subfigure}
	~
	\begin{subfigure}[b]{0.3\textwidth}
		\centering
		\includegraphics[scale = 0.45]{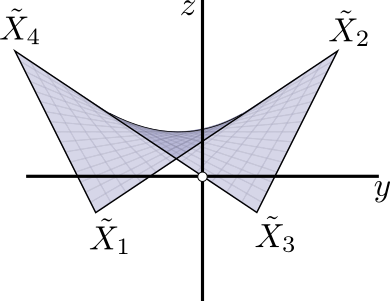}
	\end{subfigure}
	~
	\begin{subfigure}[b]{0.3\textwidth}
		\centering
		\includegraphics[scale = 0.45]{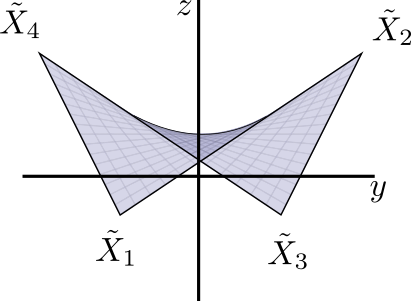}
	\end{subfigure}\\
	\vskip2pt
	\begin{subfigure}[b]{0.3\textwidth}
		\centering
		\includegraphics[scale = 0.41]{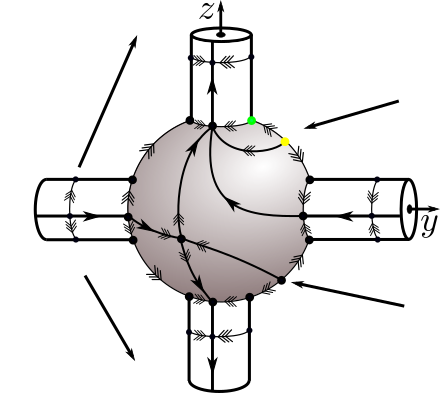}
		\caption{}
	\end{subfigure}
	~
	\begin{subfigure}[b]{0.3\textwidth}
		\centering
		\includegraphics[scale = 0.41]{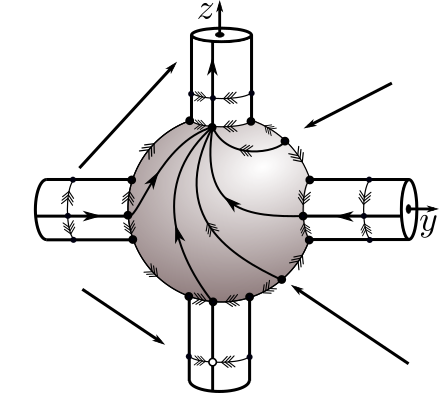}
		\caption{}
	\end{subfigure}
	~
	\begin{subfigure}[b]{0.3\textwidth}
		\centering
		\includegraphics[scale = 0.41]{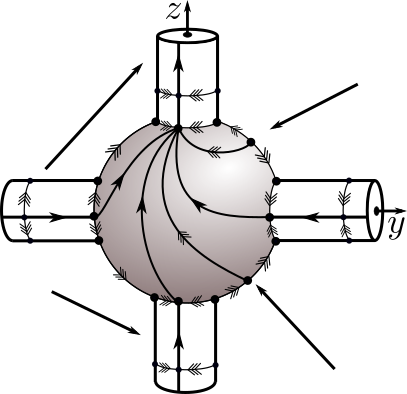}
		\caption{}
	\end{subfigure}
	~
	\caption{Disappearance of a unique sliding vector field by escaping $\tilde{\mathcal{S}}_h$. (a) For $x<x_b$,  the origin $(y,z) = (0,0)$ lies inside $\tilde{\mathcal{S}_h}$ and a sliding vector fields exists. (b) For $x=x_b$, $(y,z) = (0,0)$ lies on $\chi_3$ and at the unique equilibrium point of the fast subsystem disappears. (c) For $x>x_b$, the origin $(y,z) = (0,0)$ lies outside $\tilde{\mathcal{S}}_n$ and the fast subsystem has no equilibrium point, therefore no sliding vector field exists. Triple-headed arrows are used to indicate hyperbolic directions. Single-headed arrows, on the other hand, are representing center directions. All the analysis can be done in directional charts, see \appref{aa} and \protect\cite{krupa_extending_2001,KristiansenHogan2018,kriBlowup}. }
	\figlab{sh-out}
\end{figure}

\begin{figure}[h!]
	\centering
	\begin{subfigure}[b]{0.3\textwidth}
		\centering
		\includegraphics[scale = 0.45]{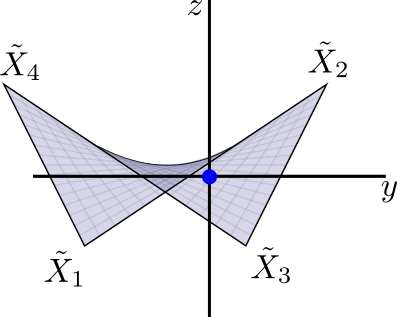}
	\end{subfigure}
	~
	\begin{subfigure}[b]{0.3\textwidth}
		\centering
		\includegraphics[scale = 0.45]{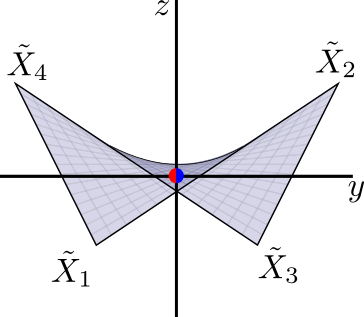}
	\end{subfigure}
	~
	\begin{subfigure}[b]{0.3\textwidth}
		\centering
		\includegraphics[scale = 0.45]{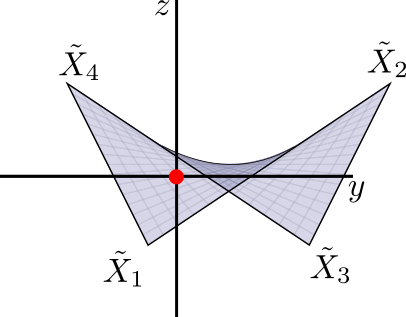}
	\end{subfigure}\\
	\vskip2pt
	\centering
	\begin{subfigure}[b]{0.3\textwidth}
		\centering
		\includegraphics[scale = 0.41]{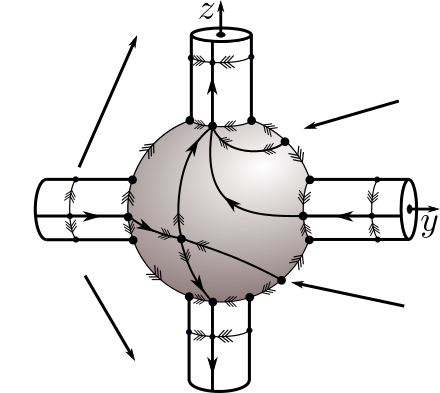}
		\caption{}
	\end{subfigure}
	~
	\begin{subfigure}[b]{0.3\textwidth}
		\centering
		\includegraphics[scale = 0.41]{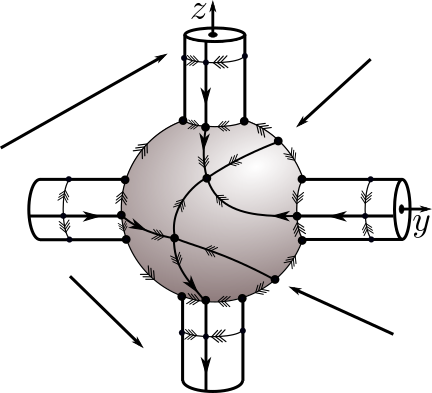}
		\caption{}
	\end{subfigure}
	~
	\begin{subfigure}[b]{0.3\textwidth}
		\centering
		\includegraphics[scale = 0.41]{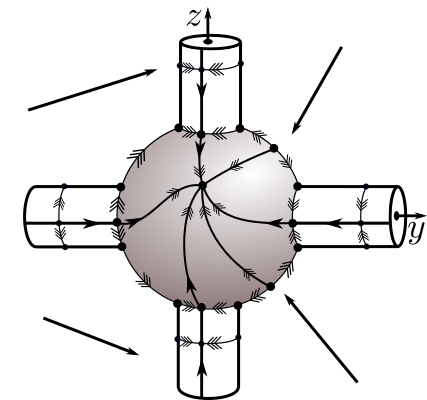}
		\caption{}
	\end{subfigure}
	\caption{Disappearance of one of the two sliding vector fields when entering $\tilde{\mathcal{S}}_h$ from $\tilde{\mathcal{S}}_n$. In (a) and (c), the origin $(y,z) = (0,0)$ lies inside $\tilde{\mathcal{S}}_h$ and a unique sliding vector field exists, of saddle and of node/focus type respectively. In (b), the origin $(y,z) = (0,0)$ lies inside $\tilde{\mathcal{S}}_n$ and a pair of sliding vector fields.}
	\figlab{sn-sh}
\end{figure}

\subsection{Entering (leaving) $\tilde{\mathcal{S}}_n$ through $\mathcal{L}_p$: The saddle-node bifurcation}
In \figref{Lp} we illustrate another sliding bifurcation. In column (a), we see that $\tilde{\mathcal S}$ does not intersect the origin in the $yz$-plane. As a consequence, the dynamics on the sphere does not have any equilibria and the equilibrium that appears along the south pole is the global attractor for the dynamics on the sphere. This creates a mechanism from going from $z>0$ to $z<0$ through stable sliding along $\Pi_4$. In (b) for $x=x_b$, $\tilde{\mathcal S}$ has moved upwards such that $\tilde{\mathcal S}$ now intersects $(y,z)=(0,0)$ along the parabolic line. As a result, there exists a (nonhyperbolic) saddle-node equilibrium of the layer problem \eqref{layprob} which in \figref{Lp}(c) has become a saddle and a stable node for $x>x_b$. Notice, that as opposed to the sliding bifurcation in \figref{sh-out}, this bifurcation does not alter the sliding dynamics along either of the codimension-1 sliding planes $\Pi_i$, $i=1,\ldots, 4$. It is a bifurcation on the sphere, but it has global consequences. In (c), points with $z>0$ cannot get to $z<0$ due to the existence of the unstable manifold of the saddle. All the analysis is based on calculations done in charts, see \appref{aa}, working sufficiently close to the parabolic line. Recall again that the  $(\hat y,\hat z)$-system \eqref{layprob} is a global nonlinear system and hence, further away the bifurcation in (b), limit cycles and homoclinics could appear, which we are not able to study by our predominantly local methods (without imposing additional structure). Interestingly, the bifurcation in \figref{Lp} actually has two generic types depending on the location of the strong stable manifold of the saddle-node and the unique center manifold coming from $\Pi_4$. The two cases are illustrated in \figref{Lp}(b) and \figref{LpNew}(b). The boundary of the two cases is illustrated in \figref{LpNew}(a). There are therefore also further variations of \figref{Lp}(c), the details which we do not present here. 


\begin{figure}[h!]
	\centering
	\begin{subfigure}[b]{0.3\textwidth}
		\centering
		\includegraphics[scale = 0.45]{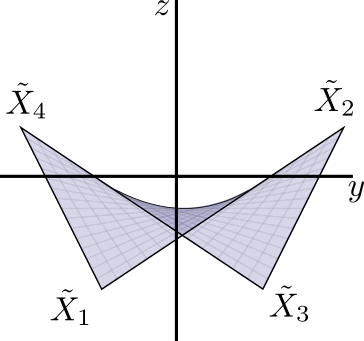}
	\end{subfigure}
	~
	\begin{subfigure}[b]{0.3\textwidth}
		\centering
		\includegraphics[scale = 0.45]{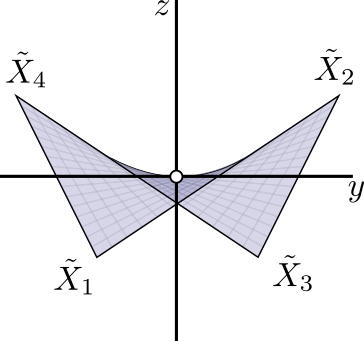}
	\end{subfigure}
	~
	\begin{subfigure}[b]{0.3\textwidth}
		\centering
		\includegraphics[scale = 0.45]{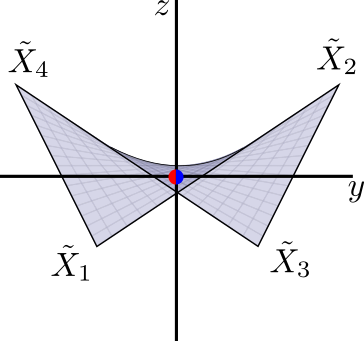}
	\end{subfigure}\\
	\vskip2pt
	\centering
	\begin{subfigure}[b]{0.3\textwidth}
		\centering
		\includegraphics[scale = 0.41]{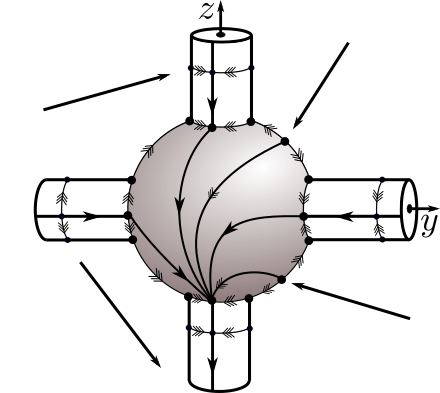}
		\caption{}
	\end{subfigure}
	~
	\begin{subfigure}[b]{0.3\textwidth}
		\centering
		\includegraphics[scale = 0.41]{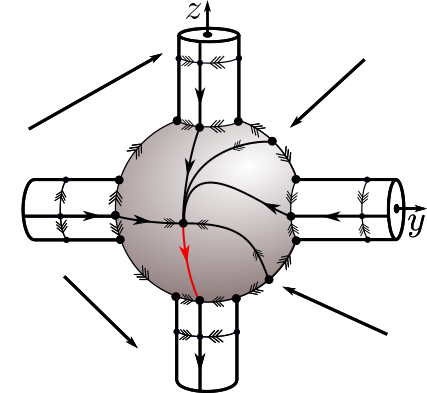}
		\caption{}
	\end{subfigure}
	~
	\begin{subfigure}[b]{0.3\textwidth}
		\centering
		\includegraphics[scale = 0.41]{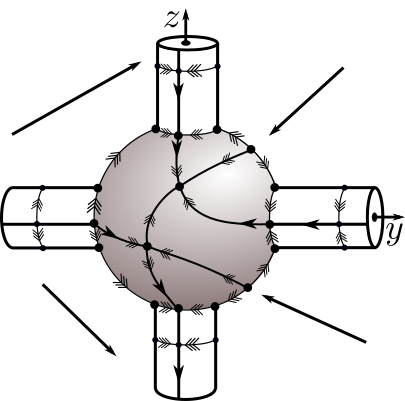}
		\caption{}
	\end{subfigure}
	~
	\caption{Appearance of a pair of sliding vector fields by entering $\tilde{\mathcal{S}}_n$ through $\mathcal{L}_p$. (a) For $x<x_b$, the origin  $(y,z) = (0,0)$ lies outside $\tilde{\mathcal{S}}$ and no sliding vector field exists. (b) For $x=x_b$, the origin $(y,z) = (0,0)$ lies on $\mathcal{L}_p$ and a nonhyperbolic equilibrium point of the fast subsystem emanates, corresponding to a unique sliding vector field. (c) For $x>x_b$, the origin $(y,z) = (0,0)$ lies inside $\tilde{\mathcal{S}}_n$ and the fast subsystem has two hyperbolic equilibrium points, one of saddle and one of node/focus type. Therefore, a pair of sliding vector fields exist, with respective types of stability.}
	\figlab{Lp}
\end{figure}

\begin{figure}[h!]
	\centering
	\begin{subfigure}[b]{0.4\textwidth}
		\centering
		\includegraphics[scale = 0.45]{sn-bif1b.png}
	\end{subfigure}
	~
	\begin{subfigure}[b]{0.48\textwidth}
		\centering
		\includegraphics[scale = 0.45]{sn-bif1b.png}
	\end{subfigure}\\
	\vskip2pt
	\centering
	\begin{subfigure}[b]{0.4\textwidth}
		\centering
		\includegraphics[scale = 0.41]{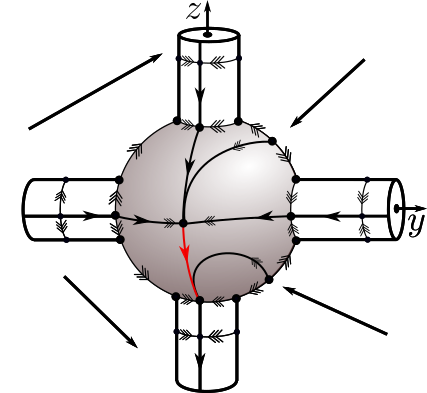}
		\caption{}
	\end{subfigure}
	~
	\begin{subfigure}[b]{0.48\textwidth}
		\centering
		\includegraphics[scale = 0.41]{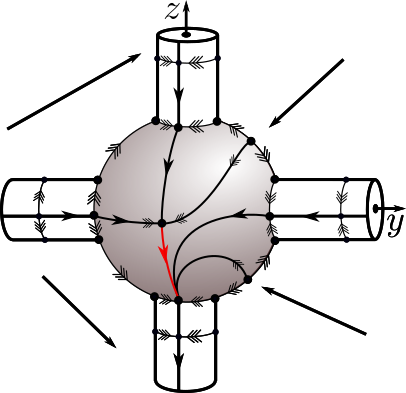}
		\caption{}
	\end{subfigure}
	\caption{The bifurcation in \figref{Lp}(b) comes in different generic form, shown in (b). The case (a) is at the boundary of these two. Here the unique center manifold coming from $\Pi_4$ coincides on one side with the strong stable manifold of the saddle-node.  }
	\figlab{LpNew}
\end{figure}

\subsection{Global dynamics and (non)uniqueness of solutions}
From a PWS perspective, the system in \figref{Lp} (c) with multiple sliding vectors along $\Lambda$ is ill-posed. Which sliding vector should one follow? But by regularization and blowup we can resolve this by replacing the Filippov system with the well-defined limit of solutions as the regularized system approaches the PWS one ($\epsilon\rightarrow 0$). Indeed, the stable manifold of the saddle produces a $2D$ separatrix in the full $3D$ space (the stable manifold of the saddle-type critical manifold). This manifold separates things above, that reaches the node, from things below that eventually follow $\Pi_4$. Along the node, we have reduced flow described by the associated sliding vector-field. The separatrix itself produces a canard phenomenon for $0<\varepsilon\ll 1$ where an exponentially thin set $(\mathcal O(e^{-c/\varepsilon})$) of initial conditions follows the sliding vector-field corresponding to the reduced problem along the saddle for an extended period of time. Away from the stable manifold, the forward flow is well-defined for $\epsilon\rightarrow 0$ and single valued. There is no nonuniquness of forward orbits in the PWS limit of the regularization. 

In the $(x,\hat y,\hat z)$-variables, the saddle-node bifurcation in \figref{Lp}(b) means that the critical manifold has a folded structure. See \figref{foldedC0}. Here Fenichel's theory breaks down. However, in the generic case, where the sliding vector-field does not vanish at the fold, there is only one orbit (red \figref{Lp}(b)) of the system for $\varepsilon=0$ that leaves this point in forward time. This orbit is the one-sided unstable manifold of the saddle-node that reaches $\Pi_4$ in forward time. Also, by reduction to a $2D$ center manifold near the fold it follows from general results in \protect\cite{krupa_extending_2001} that the system with $0<\varepsilon\ll 1$ will follow this unique forward orbit for $\varepsilon$ and therefore, once reaching $\Pi_4$, it can be approximated by the stable sliding along $\Pi_4$. Again, the forward flow is well-defined for $\epsilon\rightarrow 0$ and there is no nonuniquness of forward orbits in the PWS limit.

On the other hand, if the sliding vector-field vanishes at the fold, then generically there are several forward orbits of the $\varepsilon=0$-system that leave the fold following the saddle part of $C_0$, see \figref{foldedC0}. This situation is seen in \figref{foldedC0} using orbits of different colour. This situation produces a canard explosion. 
\begin{theorem}\thmlab{canards}
Let $\phi$ and $\psi$ be fixed regularization functions. Then canard explosions are generic (nondegenericity conditions depending upon $X_{1-4}$ only are stated in (a),(b) and (c) below) for regularizations $X_\epsilon$ \eqref{reg3a} of PWS systems \eqref{X14} depending on a single unfolding parameter. The canard point for $\epsilon\rightarrow 0$ is independent of $\phi$ and $\psi$. 
\end{theorem}
\begin{proof}
Suppose that 
\begin{itemize}
 \item[(a)] the linearization at the fold point has only one non-zero eigenvalue for $\epsilon=0$.
\end{itemize}
Then by Center Manifold Theory we can therefore reduce the situation in \figref{foldedC0} to a local $2D$ center manifold where $C_0$ is a quadratic, $1D$, folded curve of equilibria for $\epsilon=0$. This is precisely the situation studied in e.g. \protect\cite{krupa_relaxation_2001} and the existence of canard explosions therefore follow from  \protect\cite[Theorem 3.3]{krupa_relaxation_2001}. The canard point is for $\epsilon=0$ just determined by the condition that the $x$-nullcline, say $\mathcal N_x$, a $2D$ surface, intersects the curve $C_0$ at the fold. It therefore follows that the canard point for $\epsilon=0$ only depends upon $X_{1-4}$. The nondegenericity conditions of \protect\cite[Theorem 3.3]{krupa_relaxation_2001} are satisfied if 
\begin{itemize}
 \item[(b)]  the intersection of $\mathcal N_x$ with the critical $C_0$ is transverse at the fold,
\end{itemize}
and if 
\begin{itemize}
 \item[(c)] the $1$-parameter unfolding of $X_{1-4}$ transverses $\mathcal N_x$ along $C_0$ with non-zero speed.
\end{itemize}
%
\end{proof}

We discuss this result further in the conclusion.

\begin{figure}[h!]
	\centering
\includegraphics[scale = 0.35]{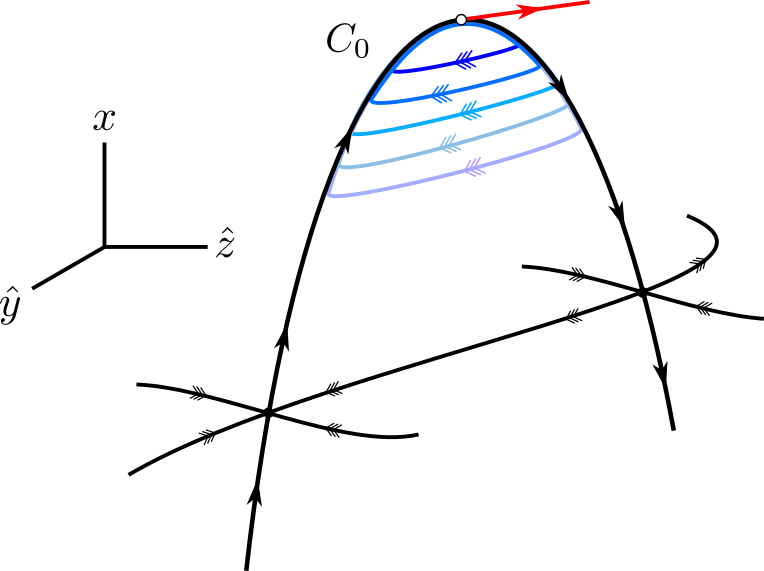}
\caption{Folded critical manifold appearing as a result of the sliding bifurcation \figref{Lp}(b) where $\tilde{\mathcal S}$ intersects $(y,z)=(0,0)$ along the parabolic line. Here we have drawn the case where the reduced flow has an equilibrium at the fold. The codimension-$1$ parameter unfolding of this situation with $0<\varepsilon\ll 1$ produces the canard explosion phenomenon where limit cycles follow the singular cycles indicated in the figure using different colours in an exponentially thin region of parameter space.}
\figlab{foldedC0}
\end{figure}

\section{Conclusion}\seclab{conclusion}
In this paper, we presented general results on the existence and multiplicity of sliding along a {codimension-2 intersection of two codimension-1} switching manifolds (see \secref{existence}, \propref{propdef} to \propref{conv-double}). Our approach was geometric and based upon studying the canopy of \protect\cite{jeffrey2014a}. By defining sliding through a regularization of the PWS system, we also introduced a concept of stability of sliding. We presented some general results on stability, most importantly showing (see \thmref{uniqueStable}) that there can be at most one stable sliding vector-field; when sliding vectors co-exists then one has to be of saddle type while the other sliding vector corresponds to a focus/node or center. Unfortunately, the downside of this definition of stability is that it generically depends upon the regularization functions used. Different regularizations may produce different stability of the focus/node or center; the result does not only depend upon the prescribed vector-fields $X_{1-4}$. We emphasize that another regularization approach to the problem would be to replace $\psi(\varepsilon^{-1} y)$ by $\psi(\delta^{-1} y)$ in \eqref{reg3a} and consider a regularization $X_{\delta,\varepsilon}$ depending on two small parameters $\varepsilon\ll 1$ and $\delta\ll 1$. This would produce different results from ours when $\delta\ll \varepsilon\ll 1$ (or  $\varepsilon \ll \delta\ll 1$).

In \secref{bifurcation} we demonstrated a blowup approach for the study of bifurcations of sliding. We showed how this approach could be used, together with the canopy surface $\mathcal S$ and its projection $\tilde{ \mathcal S}$, to analyse the emergence and disappearance of sliding and its global consequences. We focused on two specific cases. In the first case, recall \figref{sh-out} (b),  sliding along $\Lambda$ disappears because of two adjacent projected vectors $\tilde X_i$ and $\tilde X_{i+1}$ becoming anti-parallel. In terms of the projected area $\tilde{\mathcal S}$, this means that the origin intersects $\mathcal S$ along one of the straight edges. In the second case, recall \figref{Lp} (b), sliding along $\Lambda$ disappears in a way which is less apparent in terms of the PWS system. But in terms of the blowup system, the bifurcation appears as a saddle-node in the layer problem. Furthermore, for the set $\tilde{\mathcal S}$, the bifurcation occurs exactly when the origin is along the parabolic line. These two examples are generic in $\mathbb R^3$. It is $x$ that unfolds the bifurcations and these ``PWS bifurcations'' therefore replace the generic folds from classical PWS system along codimension-1 discontinuity sets. 
However, the bifurcation in \figref{Lp} (b) is also a fold of a critical manifold and this situation can therefore, under variation of one single unfolding parameter, produce canard explosions of limit cycles. We collected this result in \thmref{canards}. Notice that the canard point is independent of the regularization functions. This example demonstrates that, although we may have forward nonuniqueness of our $\varepsilon=0$-system, we may use more complicated results from geometric singular perturbation theory or simply blowup again, see e.g. \protect\cite{KristiansenHogan2018},  to capture a well-defined PWS limit of our regularization and in this way obtain ``higher order corrections'' to the PWS system. Some of the qualitative details of this approach will, however, in general depend on the regularization functions. 

Our analysis was presented for the case of PWS systems in $\mathbb{R}^3$ with $\Lambda$ being $1D$. In particular, $x\in \mathbb R$. But, since $x$ is parameter of the layer problem \eqref{layprob}, all results in \secref{existence} and \secref{stabsec} also apply to the case of PWS systems in $\mathbb{R}^n$ with $x\in \mathbb R^{n-2}$. The result in \thmref{canards}, however, only applies to $x\in \mathbb R$. In higher dimensions canards are generic, like canards in slow-fast systems in $\mathbb R^3$, see \protect\cite{szmolyan_canards_2001}. Interestingly, there are also different canards of the regularization $X_\epsilon$ of $X_{1-4}$ even in $\mathbb R^3$ that connect stable sliding along $\Pi_i$ with unstable sliding along $\Pi_j$ with $j\ne i$. These canards could be a direction for future research.

Another possible direction for future work would be to study the canard explosion phenomena in \thmref{canards} and the associated global dynamics in further details, in particular describing possible examples of relaxation oscillations that are produced by the explosion of the small, local limit cycles near $\Lambda$. Also, from a modelling perspective, one may view the fact that the stability of sliding depends upon the regularization function as lack of sufficient modelling. Knowing $X_{1-4}$ is not enough to determine the outcome of the system. In this regard, it would therefore be interesting to further classify all of the cases where the result is independent of the regularization function, and hence where additional modelling is not required. An example of such a case is shown in \figref{proj1} (a). This gives an attracting focus/node for every regularization function. In contrast, it is known that different regularization functions can change the stability of certain focus/node/center sliding vectors. Therefore Hopf bifurcations can be produced in this way. The emerging limit cycles therefore produce normally hyperbolic invariant (for the layer problem \eqref{layprob}) cylinders for $\varepsilon=0$. The reduced problem on such a manifold, see \protect\cite{fen3}, also defines a ``sliding vector-field'' $\Lambda$ upon projection $(x,\hat y,\hat z)\mapsto x$. To our knowledge, such ``sliding vectors'' have not been studied before. Finally, in this manuscript we only described the simplest possible bifurcation scenarios in \secref{bifurcation}. Other interesting bifurcations can also occur (for example when the parabolic line disappears) which require further analysis. 

\section*{Acknowledgements}\seclab{aknow}
The main results of this paper were obtained during the first author's affiliation with the Technical University of Denmark as an M.Sc. student (Jan-Aug 2017) and subsequently as a Research Assistant (Nov-Dec 2017) at the Department of Applied Mathematics \& Computer Science, under the supervision of the second author. The authors would like to thank John Hogan for critically reading previous versions of the manuscript and then providing constructive feedback.

\appendix 
\section{Calculations in charts}\applab{aa}
The results in \secref{bifurcation}, presented in \figref{sh-out}, \figref{sn-sh} and \figref{Lp}, are based upon calculations done in directional charts obtained by central projections. For example, setting $\bar z=1$ in \eqref{blowup1} gives
\begin{align}
 (x,y,z,\varepsilon) = (x,r_1 y_1, r_1, r_1\varepsilon_1), \eqlab{chartZBar1}
\end{align}
in local coordinates $(r_1,y_1, \varepsilon_1)$ defined by $r_1 = r\bar z$ and
\begin{align}
 (\bar y,\bar z,\bar \varepsilon ) \mapsto (y_1,\varepsilon_1) = \bar z^{-1} (\bar y,\bar \varepsilon).\eqlab{chartapp}
\end{align}
By \eqref{hatyhatz}
we can therefore change coordinates between $(\hat y,\hat z)$ and $(y_1,\varepsilon_1)$ as follows
\begin{align*}
 y_1 = \hat z^{-1} \hat z,\,\varepsilon_1 = \hat z^{-1}.
\end{align*}
Inserting \eqref{chartZBar1} into $((x,y,z)'=X_\epsilon,\epsilon'=0)$ gives the following system
 \begin{eqnarray}
\begin{aligned}
\dot{x}= &\frac{r_1}{2}\lp \frac{\alpha_1}{2}{\lp 1+\phi_+\lp \varepsilon_1\rp \rp}+\frac{\alpha_4}{2}{\lp 1-\phi_+\lp \varepsilon_1\rp \rp}\rp{\lp 1+\psi\lp \varepsilon_1^{-1}y_1\rp\rp}\\
&+\frac{r_1}{2}\lp \frac{\alpha_2}{2}{\lp 1+\phi_+\lp \varepsilon_1\rp \rp}+\frac{\alpha_3}{2}{\lp 1-\phi_+\lp \varepsilon_1\rp \rp}\rp{\lp 1-\psi\lp \varepsilon_1^{-1}y_1\rp\rp},\\
\dot{r}_1= &\frac{r_1}{2}\lp \frac{\gamma_1}{2}{\lp 1+\phi_+\lp \varepsilon_1\rp \rp}+\frac{\gamma_4}{2}{\lp 1-\phi_+\lp \varepsilon_1\rp \rp}\rp{\lp 1+\psi\lp \varepsilon_1^{-1}y_1\rp\rp}\\
&+\frac{r_1}{2}\lp \frac{\gamma_2}{2}{\lp 1+\phi_+\lp \varepsilon_1\rp \rp}+\frac{\gamma_3}{2}{\lp 1-\phi_+\lp \varepsilon_1\rp \rp}\rp{\lp 1-\psi\lp \varepsilon_1^{-1}y_1\rp\rp},\\
\dot{y}_1= &\frac{1}{2}\lp \frac{\beta_1}{2}{\lp 1+\phi_+\lp \varepsilon_1\rp \rp}+\frac{\beta_4}{2}{\lp 1-\phi_+\lp \varepsilon_1\rp \rp}\rp{\lp 1+\psi\lp \varepsilon_1^{-1}y_1\rp\rp}\\
&+\frac{1}{2}\lp \frac{\beta_2}{2}{\lp 1+\phi_+\lp \varepsilon_1\rp \rp}+\frac{\beta_3}{2}{\lp 1-\phi_+\lp \varepsilon_1\rp \rp}\rp{\lp 1-\psi\lp \varepsilon_1^{-1}y_1\rp\rp}\\
&-\frac{y_1}{2}\lp \frac{\gamma_1}{2}{\lp 1+\phi_+\lp \varepsilon_1\rp \rp}+\frac{\gamma_4}{2}{\lp 1-\phi_+\lp \varepsilon_1\rp \rp}\rp{\lp 1+\psi\lp \varepsilon_1^{-1}y_1\rp\rp}\\
&-\frac{y_1}{2}\lp \frac{\gamma_2}{2}{\lp 1+\phi_+\lp \varepsilon_1\rp \rp}+\frac{\gamma_3}{2}{\lp 1-\phi_+\lp \varepsilon_1\rp \rp}\rp{\lp 1-\psi\lp \varepsilon_1^{-1}y_1\rp\rp},\\
\dot{\varepsilon}_1 &=-\frac{\varepsilon_1}{2}\lp \frac{\gamma_1}{2}{\lp 1+\phi_+\lp \varepsilon_1\rp \rp}+\frac{\gamma_4}{2}{\lp 1-\phi_+\lp \varepsilon_1\rp \rp}\rp{\lp 1+\psi\lp \varepsilon_1^{-1}y_1\rp\rp}\\
&-\frac{\varepsilon_1}{2}\lp \frac{\gamma_2}{2}{\lp 1+\phi_+\lp \varepsilon_1\rp \rp}+\frac{\gamma_3}{2}{\lp 1-\phi_+\lp \varepsilon_1\rp \rp}\rp{\lp 1-\psi\lp \varepsilon_1^{-1}y_1\rp\rp},
\end{aligned}
\eqlab{here}
\end{eqnarray}
using \eqsref{RegSysS}{phipm},
after multiplication of the right hand side by $r_1$. 
By \eqref{chartapp}, \eqref{blowup2} becomes
\begin{align}
 (y_1,\varepsilon_1 ) = \rho (\bar{\bar y},\bar{\bar \epsilon}), \eqlab{blowup2App}
\end{align}
in the $(y_1, \varepsilon_1)$-coordinates.
Setting $\bar{\bar \epsilon}=1$ here gives
\begin{align*}
 (y_1,\varepsilon_1 ) = (\rho_1 y_{11},\rho_1),
\end{align*}
in new local coordinates $(\rho_1,y_{11})$. 
Therefore, in total
\begin{align*}
  y &=r_1 \rho_1 y_{11},\\
  z &=r_1,\\
 \epsilon &=r_1 \rho_1,
\end{align*}
 using  \eqref{chartZBar1}.
By eliminating $r_1$ and $\rho_1$, we simply obtain $y=\epsilon \hat y$, which is just \eqref{yhat}, and therefore also the equations in \eqref{SFF}. 
Then by \thmref{holy} we therefore obtain a normally hyperbolic critical manifold within this chart whenever the corresponding PWS system $(X_1,X_4)$ have sliding along $\Pi_1$. 
The advantages of using the coordinates $(r_1,\rho_1,y_{11})$, however, is that $C_0$ is normally hyperbolic all the way up to $r_1= 0$. This enables an extension of $C_0$ onto $\mathcal I\times S^2$, the blowup of $\Lambda$, as a local center manifold. This is the typical advantage of the blowup method, see also \protect\cite{krupa_extending_2001,krupa_extending2_2001,kuehn2015,KristiansenHogan2015,krihog2} where this approach is used in different contexts. Also, if the sliding flow is nonvanishing then the direction of the flow on the local center manifold is in the same direction, see e.g. the (nonunique) center manifold (in red in \figref{Lp} (b)) of the partially hyperbolic equilibrium on the blowup of $\Pi_4$: The dynamics on the local center manifold has $\dot{\hat y}<0$, in correspondence with $\dot y<0$ on the sliding flow along $\Pi_4$. 

If we instead put $\bar{\bar y}=1$ in \eqref{blowup2App} then 
\begin{align*}
 (y_1,\varepsilon_1 ) = (\rho_2,\rho_2 \varepsilon_{12}).
\end{align*}
in new local coordinates $(\rho_2,\varepsilon_{12})$. 
Inserting this into \eqref{here} gives the following equations
\begin{eqnarray}
\begin{aligned}
\dot{x}= &\frac{r_1\rho_2}{2}\lp \frac{\alpha_1}{2}{\lp 1+\phi\lp \rho_2 \varepsilon_{12}\rp\rp}+\frac{\alpha_4}{2}{\lp 1-\phi\lp \rho_2 \varepsilon_{12}\rp\rp}\rp{\lp 1+\psi_+\lp \varepsilon_{12}\rp \rp}\\
&+\frac{r_1\rho_2}{2}\lp \frac{\alpha_2}{2}{\lp 1+\phi\lp \rho_2 \varepsilon_{12}\rp\rp}+\frac{\alpha_3}{2}{\lp 1-\phi\lp\rho_2^{-1} \varepsilon_{12}^{-1}\rp\rp}\rp{\lp 1-\psi_+\lp \varepsilon_{12}\rp \rp},\\
\dot{r}_1= &\frac{r_1\rho_2 }{2}\lp \frac{\gamma_1}{2}{\lp 1+\phi\lp \rho_2 \varepsilon_{12}\rp\rp}+\frac{\gamma_4}{2}{\lp 1-\phi\lp \rho_2 \varepsilon_{12}\rp\rp}\rp{\lp 1+\psi_+\lp \varepsilon_{12}\rp \rp}\\
&+\frac{r_1\rho_2}{2}\lp \frac{\gamma_2}{2}{\lp 1+\phi\lp \rho_2 \varepsilon_{12}\rp\rp}+\frac{\gamma_3}{2}{\lp 1-\phi\lp \rho_2 \varepsilon_{12}\rp\rp}\rp{\lp 1-\psi_+\lp \varepsilon_{12}\rp \rp},\\
\dot{\rho}_2= &\frac{\rho_2}{2}\lp \frac{\beta_1}{2}{\lp 1+\phi\lp \rho_2 \varepsilon_{12}\rp\rp}+\frac{\beta_4}{2}{\lp 1-\phi_+\lp \varepsilon_1\rp \rp}\rp{\lp 1+\psi\lp \varepsilon_1^{-1}y_1\rp\rp}\\
&+\frac{\rho_2}{2}\lp \frac{\beta_2}{2}{\lp 1+\phi\lp \rho_2 \varepsilon_{12}\rp\rp}+\frac{\beta_3}{2}{\lp 1-\phi\lp \rho_2 \varepsilon_{12}\rp\rp}\rp{\lp 1-\psi_+\lp \varepsilon_{12}\rp \rp}\\
&-\frac{\rho_2^2}{2}\lp \frac{\gamma_1}{2}{\lp 1+\phi\lp \rho_2 \varepsilon_{12}\rp\rp}+\frac{\gamma_4}{2}{\lp 1-\phi\lp \rho_2 \varepsilon_{12}\rp\rp}\rp{\lp 1+\psi_+\lp \varepsilon_{12}\rp \rp}\\
&-\frac{\rho_2^2}{2}\lp \frac{\gamma_2}{2}{\lp 1+\phi\lp \rho_2 \varepsilon_{12}\rp\rp}+\frac{\gamma_3}{2}{\lp 1-\phi\lp \rho_2 \varepsilon_{12}\rp\rp}\rp{\lp 1-\psi_+\lp \varepsilon_{12}\rp \rp},\\
\dot{\varepsilon}_{12} &=-\frac{\varepsilon_{12}}{2}\lp \frac{\beta_1}{2}{\lp 1+\phi_+\lp \varepsilon_1\rp \rp}+\frac{\beta_4}{2}{\lp 1-\phi\lp \rho_2 \varepsilon_{12}\rp\rp}\rp{\lp 1+\psi_+\lp \varepsilon_{12}\rp \rp}\\
&-\frac{\varepsilon_{12}}{2}\lp \frac{\beta_2}{2}{\lp 1+\phi\lp \rho_2 \varepsilon_{12}\rp\rp}+\frac{\beta_3}{2}{\lp 1-\phi\lp \rho_2 \varepsilon_{12}\rp\rp}\rp{\lp 1-\psi_+\lp \varepsilon_{12}\rp \rp},
\end{aligned}\eqlab{here2}
\end{eqnarray}
after multiplication of $\rho_2$ on the right hand side. Here we have again used \eqref{phipm} to introduce $\psi_+(s) = \psi(s^{-1})$, $s\ge 0$. 
Now, $r_1=\varepsilon_{12}=0$ corresponds to the subset  of the equator $\bar \epsilon=1$ of the sphere in \figref{LambdaBlowup} with $\bar y>0,\bar z>0$. This is an invariant set for \eqref{here2} with the following dynamics: $\dot x=0$ and 
\begin{align}
  \dot \rho_2 &= {\rho_2} \left(\beta_1 -{\rho_2} \gamma_1\right). \eqlab{rho2Eqn}
\end{align}
Notice, that in \figref{sh-out} (a) for example, $\beta_1<0$ and $\gamma_1<0$ and therefore there exists two hyperbolic equilibrium for \eqref{rho2Eqn} at $\rho_2=0$ (stable, green in \figref{sh-out} (a)) and $\rho_2=\gamma^{-1} \beta_1$ (unstable, yellow in \figref{sh-out} (a)). By linearization of the full system \eqref{here2} about any point $\rho_2=\gamma(x)^{-1}\beta_1(x),r_1=0,\varepsilon_{12}=0,x\in \mathcal I$, we obtain an additional positive eigenvalue $\beta_1(x)$ with associated eigenvector contained in the $(\rho_2,\epsilon_{12})$-plane and notice that $\rho_2 = \gamma_1(x)^{-1} \beta_1(x)$, $r_1\ge 0$, $\epsilon_{12}=0,x\in \mathcal I$ is a $2D$ stable manifold of the curve $\rho_2=\gamma(x)^{-1} \beta_1(x), x\in \mathcal I$ of equilibrium points. For the PWS system, this invariant manifold corresponds to all the points in $\mathcal Q_1$ that reach $\Lambda$ by following $X_1$. Linearization about $\rho_2=0,r_1=0,\varepsilon_{12}=0$ gives one single positive eigenvalue $(-\beta_1)$ with an associated eigenvector purely in the $\varepsilon_{12}$-direction. In fact, any point $\rho_2=\varepsilon_{12}=0$ with $r_1\ge 0$, $x\in \mathcal I$ is an equilibrium of \eqref{here2} having a $1D$ stable and a $1D$ unstable manifolds. 

We obtain similar results depending on the signs of $\beta_i$ and $\gamma_i$ along the equator $\bar \epsilon=1$ of the sphere in \figref{LambdaBlowup}. Together with a phase portrait analysis of the \eqref{layprob} we can then produce the results that are collected in the diagrams \figref{sh-out}, \figref{sn-sh} and \figref{Lp} in \secref{bifurcation} (albeit with some limitations on the global dynamics in \figref{sn-sh} and \figref{Lp} that are explained in the text).

\bibliography{refs}
\bibliographystyle{plain}

\newpage
\end{document}